\documentclass[11pt]{amsart}
\usepackage{amsmath,latexsym,amsfonts,amssymb,amsthm}
\usepackage{geometry}
\usepackage{hyperref}
\usepackage{mathrsfs}
\usepackage{graphicx,color}
\usepackage[alphabetic]{amsrefs}
\usepackage[all,cmtip]{xy}

\newtheorem{prop}{Proposition}[section]
\newtheorem{thm}[prop]{Theorem}
\newtheorem{cor}[prop]{Corollary}

\newtheorem{lem}[prop]{Lemma}

\theoremstyle{definition}

\newtheorem{defn}[prop]{Definition}
\newtheorem{expl}[prop]{Example}
\newtheorem{rem}[prop]{\it Remark}

\newtheorem{lemdefn}[prop]{Lemma-Definition}

\newtheorem*{claim*}{Claim}

\newcommand{\bP}{\mathbb{P}}
\newcommand{\bC}{\mathbb{C}}
\newcommand{\bR}{\mathbb{R}}

\newcommand{\bQ}{\mathbb{Q}}
\newcommand{\bZ}{\mathbb{Z}}
\newcommand{\bN}{\mathbb{N}}

\newcommand{\tL}{\widetilde{L}}
\newcommand{\tQ}{\widetilde{Q}}
\newcommand{\tD}{\widetilde{D}}
\newcommand{\tG}{\widetilde{G}}
\newcommand{\tC}{\widetilde{C}}

\newcommand{\cO}{\mathcal{O}}

\newcommand{\cM}{\mathcal{M}}
\newcommand{\cF}{\mathcal{F}}
\newcommand{\cJ}{\mathcal{J}}

\newcommand{\cG}{\mathcal{G}}

\newcommand{\fa}{\mathfrak{a}}
\newcommand{\fb}{\mathfrak{b}}

\newcommand{\fm}{\mathfrak{m}}

\newcommand{\va}{\vec{a}}
\newcommand{\vb}{\vec{\bullet}}
\newcommand{\vg}{\vec{\gamma}}

\newcommand{\Supp}{\mathrm{Supp}}
\newcommand{\supp}{\mathrm{supp}}

\newcommand{\mult}{\mathrm{mult}}
\newcommand{\lct}{\mathrm{lct}}

\newcommand{\Pic}{\mathrm{Pic}}
\newcommand{\vol}{\mathrm{vol}}
\newcommand{\ord}{\mathrm{ord}}

\newcommand{\Gr}{\mathrm{Gr}}
\newcommand{\gr}{\mathrm{gr}}
\newcommand{\Cl}{\mathrm{Cl}}
\newcommand{\NS}{\mathrm{NS}}
\newcommand{\rd}{\mathrm{d}}
\newcommand{\Bs}{\mathrm{Bs}}
\newcommand{\wt}{\mathrm{wt}}
\newcommand{\Val}{\mathrm{Val}}
\newcommand{\Nlc}{\mathrm{Nlc}}
\newcommand{\ddiv}{\mathrm{div}}
\newcommand{\Y}{Y_{\bullet}}
\newcommand{\W}{W_{\vec{\bullet}}}
\newcommand{\V}{V_{\vec{\bullet}}}

\numberwithin{equation}{section}

\title{K-stability of Fano varieties via admissible flags}
\author{Hamid Abban}
\author{Ziquan Zhuang}
\date{}

\begin{document}

\address{\emph{Hamid Abban}
\newline
\textnormal{Department of Mathematical Sciences, Loughborough University, Loughborough LE11 3TU, UK
\newline
 \texttt{h.abban@lboro.ac.uk}}}

\address{\emph{Ziquan Zhuang}
\newline 
\textnormal{Department of Mathematics, MIT, Cambridge, MA, 02139, USA
\newline
\texttt{ziquan@mit.edu}}}

\maketitle

\begin{abstract}
     We develop a general approach to prove K-stability of Fano varieties. The new theory is used to (a)\,prove the existence of K\"ahler-Einstein metrics on all smooth Fano hypersurfaces of Fano index two, (b)\,to compute the stability thresholds for hypersurfaces at generalized Eckardt points and for cubic surfaces at all points, and (c)\,to provide a new algebraic proof of Tian's criterion for K-stability, amongst other applications.
    %As applications, we show that all smooth Fano hypersurfaces of Fano index two admit K\"ahler-Einstein metrics and compute the stability thresholds of hypersurfaces (at generalized Eckardt points) and cubic surfaces (at all points).

\end{abstract}

\section{Introduction}

Introduced by Tian \cite{Tian-K-stability-defn} and reformulated algebraically by Donaldson \cite{Don-K-stability-defn}, K-stability is an algebro-geometric property of Fano varieties that detects the existence of K\"ahler-Einstein metric. By the celebrated works of Chen-Donaldson-Sun \cites{CDS} and Tian \cite{Tian-YTD}, a Fano manifold admits a K\"ahler-Einstein metric if and only if it is K-polystable. However, it is in general very hard to verify the K-stability of a given Fano variety. Tian's criterion, introduced in \cite{Tian-criterion}, provides a sufficient condition for K-stability, and has arguably become the most famous validity criterion for K-stability. There are also a few variants \cites{DK01,OS-alpha,Fuj-alpha} of Tian's criterion and a notable application is the K-stability of smooth hypersurfaces of Fano index one \cite{Fuj-alpha}. More recently, \cite{SZ-bsr-imply-K} discovered another K-stability criterion in the particular case of birationally superrigid Fano varieties; and, as an application \cite{Z-cpi} proved that Fano complete intersections of index one and large dimension are K-stable. However, both criteria apply exclusively to certain Fano varieties of index one and except in a few sporadic cases it is unclear how to attack the problem when the required conditions in neither criterion are satisfied; see for example \cites{AGP,Der-finite-cover,SS-two-quadric,LX-cubic-3fold}. 

The purpose of this paper is to develop a systematic approach for proving K-stability of Fano varieties. As a major application, we confirm the K-stability of all smooth hypersurfaces of Fano index two.

\begin{thm}[=Theorem \ref{thm:index two}] \label{main:index two}
Let $X=X_n\subseteq \bP^{n+1}$ be a smooth Fano hypersurface of degree $n\ge 3$. Then $X$ is uniformly K-stable.
\end{thm}

In particular, this generalizes the work of \cite{LX-cubic-3fold} on K-stability of smooth cubic threefolds, although our argument is completely different. 

As another application, we prove the following K-stability criterion, giving a unified treatment for several Fano manifolds that are previously known to be K-(semi)stable; see Definition \ref{defn:beta} for the definition of $\beta_X(E)$ in the statement.

\begin{thm}[=Corollary \ref{cor:fano small deg}] \label{main:fano small deg}
Let $X$ be a Fano manifold of dimension $n$. Assume that there exists an ample line bundle $L$ on $X$ such that
\begin{enumerate}
    \item $-K_X\sim_\bQ rL$ for some $r\in\bQ$ with $(L^n)\le \frac{n+1}{r}$; and
    \item for every $x\in X$, there exists $H_1,\dots,H_{n-1}\in |L|$ containing $x$ such that ${H_1\cap \dots \cap H_{n-1}}$ is an integral curve that is smooth at $x$.
\end{enumerate}
Then $X$ is K-semistable. If it is not uniformly K-stable, then $(L^n) = \frac{n+1}{r}$ and there exists some prime divisor $E\subseteq X$ such that $\beta_X(E)=0$.
\end{thm}

For instance, this applies to projective spaces, hypersurfaces of Fano index one and double covers of $\bP^n$ branched along a hypersurface of degree at least $n+1$. We refer to Corollary \ref{cor:small deg list} for a more exhaustive list. While Tian's criterion or the criterion from \cite{SZ-bsr-imply-K} apply to some of them, the conditions in Theorem \ref{main:fano small deg} are usually easier to check; indeed, we never use Tian's criterion or the criterion from \cite{SZ-bsr-imply-K} in this paper, as most varieties considered here are of higher Fano index. On the other hand, it may be worth pointing out that our general approach also leads to a new proof of these two criteria; see Subsection~\ref{sec:Tian & bsr}.

Before we state further applications, let us recall that by \cites{FO-delta,BJ-delta}, K-stability of a Fano variety $X$ can be characterized by its stability threshold $\delta(X)$, defined via log canonical thresholds of anti-canonical $\bQ$-divisors of basis type; see Subsection \ref{sec:K-defn}. For example, $X$ is K-semistable if and only if $\delta(X)\ge 1$. One can also define local stability thresholds $\delta_x(X)$ at some $x\in X$ by taking log canonical thresholds around the point $x$ so that the global invariant $\delta(X)$ is the minimum of the local ones $\delta_x(X)$; see Subsection \ref{sec:K-defn}. It is again a challenging problem to find the precise value of these invariants, unless the variety has a large group of automorphisms \cites{BJ-delta,Gol-delta-large-aut,ZZ-delta-cone}; see \cites{PW-dP,CZ-cubic-delta} for some estimates on del Pezzo surfaces. 

We also compute these invariants in some non-trivial cases. As a first example, we study the local stability thresholds of hypersurfaces at generalized Eckardt points.

\begin{thm}[=Corollary \ref{cor:Fano Eckardt pt}] \label{main:delta at Eckardt pt}
Let $X\subseteq \bP^{n+1}$ be a smooth Fano hypersurface of degree $d$ and let $x\in X$ be a generalized Eckardt point $($the tangent hyperplane section at $x$ is the cone over a hypersurface $Y\subseteq \bP^{n-1}$ of degree $d)$. Assume that $Y$ is K-semistable, when $d\le n-1$ $($i.e. when it is Fano$)$. Then $\delta_x(X)=\frac{n(n+1)}{(n-1+d)(n+2-d)}$ and it is computed by the ordinary blow up of $x$.
\end{thm}

Since on smooth quadric hypersurfaces every closed point is a generalized Eckardt point, by induction on dimension we obtain an algebraic proof of their K-semistability. In general, we get $\delta_x(X)\ge 1$ as long as $Y$ is K-semistable. We expect that if $X$ has a generalized Eckardt point $x$ then $\delta(X)=\delta_x(X)$, and smooth Fano hypersurfaces of degree $d$ with smallest stability thresholds are those with generalized Eckardt points (see Theorem \ref{thm:cubic delta} and Corollary \ref{cor:cubic surface global delta} for some evidence on cubic surfaces). Thus the above theorem suggests a possible inductive approach to the K-stability of Fano hypersurfaces.

As a second example, we calculate the local stability thresholds of all cubic surfaces, from which we derive the following consequences.

\begin{thm}[see Theorems \ref{thm:cubic delta} and \ref{thm:HRFG counterexample}] \label{main:counterexamples}
Let $X\subseteq \bP^3$ be a smooth cubic surface. Then there exists some boundary divisor $\Delta$ such that $(X,\Delta)$ is log Fano and $\delta(X,\Delta)=\frac{9}{25-8\sqrt{6}}\not\in\bQ$. Moreover, there exists $C\in |-K_X|$ such that $(X,C)$ is log canonical and some valuation $v$ that is an lc place of $(X,C)$ such that the associated graded ring 
\[
\gr_v R:=\oplus_{m,\lambda} \Gr^\lambda_{\cF_v} H^0(X,-mK_X)
\]
is not finitely generated, where $\cF_v$ is the filtration induced by $v$.
\end{thm}

This is somewhat surprising as, by \cite{BLZ-optimal-destabilize}*{Theorem 1.4}, the global stability thresholds $\delta(X)$ are always rational on Fano manifolds that are not K-stable. Moreover, graded rings associated to lc places of $\bQ$-complement as in the above statement are usually expected to be finitely generated; see for example \cite{LX-higher-rank}*{Conjecture 1.2}. Thus our example shows that the situation is more complicated in general.

\subsection{Overview of the proof}

We now turn to describe our approach to proving K-stability of Fano varieties. In general, one would like to estimate, or perhaps calculate, the stability threshold of a Fano variety. A priori, we need to consider log canonical thresholds of all anti-canonical basis type divisors. Our first observation is that it suffices to consider a smaller class of them, i.e. those that are compatible with a given divisor over the Fano variety.

\begin{defn} \label{defn:compatible with E}
Let $X$ be a Fano variety and let $E$ be a divisor over $X$, i.e. a prime divisor on some birational model of $X$. Let $m\in\bN$ and let $D$ be an $m$-basis type $\bQ$-divisor on $X$, i.e. there exists a basis $s_1,\cdots,s_{N_m}$ of $V_m=H^0(X,-mK_X)$, where $N_m=h^0(X,-mK_X)$, such that 
\[
D=\frac{1}{mN_m}\sum_{i=1}^{N_m} \{s_i=0\}.
\]
We say that $D$ is compatible with $E$ if for every $j\in\bN$, the subspace 
\[
\cF_E^j V_m := \{s\in V_m\,|\,\ord_E(s)\ge j\} \subseteq V_m
\]
is spanned by some $s_i$.
\end{defn}

We may then define the stability threshold $\delta(X;\cF_E)$ of $X$ with respect to $E$ by restricting to basis type divisors that are compatible with $E$. It turns out that

\begin{prop}[see Proposition \ref{prop:delta = delta wrt filtration}]
$\delta(X)=\delta(X;\cF_E)$.
\end{prop}

In other words, only basis type divisors that are compatible with $E$ are relevant when computing stability thresholds. While basis type divisors can be hard to study in general, those that are compatible with a given divisor $E$ are often concentrated around $E$, making it convenient to apply the inversion of adjunction. As an illustration, we consider the example of projective spaces.

\begin{expl} \label{expl:P^n}
Let $X=\bP^n$ and let $E$ be a hyperplane. Then asymptotically basis type $\bQ$-divisors $D$ on $X$ that are compatible with $E$ can be written as $D=E+D_0$ where $D_0$ does not contain $E$ in its support and $D_0|_E$ is a convex linear combination of basis type $\bQ$-divisors on $E\cong \bP^{n-1}$.
\end{expl}

By induction and inversion of adjunction, this easily implies that $\delta(\bP^n)\ge 1$ and thus gives an algebraic proof of the K-semistability of $\bP^n$. In general, there is a lot of flexibility in the choice of the auxiliary divisors $E$, leading to various applications. In fact, as we will show in Subsection \ref{sec:Tian & bsr}, both Tian's criterion and the criterion from \cite{SZ-bsr-imply-K} are implied by taking $E$ to be a general member of $|-mK_X|$ for some sufficiently divisible integer $m$. On explicitly given Fano varieties, however, the geometry usually suggests more natural choices of $E$ and sometimes we can even start with the optimal one, i.e. a divisor that computes $\delta(X)$, as in Example \ref{expl:P^n} and no information will be lost in the process. This is exactly how we compute the stability thresholds in Theorems \ref{main:delta at Eckardt pt} and \ref{main:counterexamples}.

More generally, instead of using an auxiliary divisor to refine the class of basis type divisors, we can also use an admissible flag, which is an important tool in the construction of Okounkov bodies of line bundles, see for example \cite{LM-okounkov-body}. Indeed, in the inductive proof of the K-semistability of $\bP^n$ as outlined above, we already implicitly use the full flags of linear subspaces. One can similarly define the compatibility of a basis type divisor with an admissible flag and show that in order to compute the stability threshold it suffices to consider basis type divisors which are compatible with a chosen flag; see Section \ref{sec:adj} for details. To prove the K-stability of a Fano variety, it is often enough to make a careful choice of the auxiliary divisor or admissible flag and analyze the corresponding compatible basis type divisors through inversion of adjunction. In particular, the proofs of Theorems \ref{main:index two} and \ref{main:fano small deg} are obtained in this way, which involve several different auxiliary divisors and admissible flags.

\subsection{Structure of the paper}

This paper is organized as follows. In Subsections \ref{sec:K-defn}--\ref{sec:prelim-flag} we put together various preliminary materials. As we apply inversion of adjunction to basis type divisors which are compatible with an admissible flag, we get basis type divisors of some filtered multi-graded linear series in a natural way. We define and study the invariants associated to such linear series in Subsections \ref{sec:prelim-graded series} and \ref{sec:prelim-filtered graded series}. In Section \ref{sec:adj}, we develop the framework to study stability thresholds of Fano varieties, or more generally $\delta$-invariants of big line bundles, and derive a few inversion-of-adjunction type results for stability thresholds. The applications are presented in Section \ref{sec:application}: in Subsection \ref{sec:Tian & bsr} we give a new proof of Tian's criterion and the criterion from \cite{SZ-bsr-imply-K}; in Subsection \ref{sec:small deg} we study K-stability of Fano manifolds of small degree and prove Theorem \ref{main:fano small deg}; in Subsection \ref{sec:surface} we explain how to compute stability thresholds of log del Pezzo surfaces almost in complete generality and in particular we prove Theorem \ref{main:counterexamples}; in Subsection \ref{sec:eckardt} we prove Theorem \ref{main:delta at Eckardt pt} and finally Theorem \ref{main:index two} is proved in Subsection \ref{sec:index two}. 

\subsection*{Acknowledgement}

This project started during the second author's visit to Loughborough University in July 2019. We would like to thank Loughborough Institute for Advanced Studies for their hospitality. We are grateful to Harold Blum, Ivan Cheltsov, Kento Fujita, Yuchen Liu, Xiaowei Wang, Chenyang Xu and Kewei Zhang for helpful discussions and comments. We would also like to thank the referees for their careful reading and suggestions. HA is supported by EPSRC grants EP/T015896/1 and EP/V048619/1. ZZ is partially supported by NSF Grant DMS-2055531.

\section{Preliminaries} 

\subsection{Notation and conventions}

We work over $\bC$. Unless otherwise specified, all varieties are assumed to be normal and projective. A pair $(X,\Delta)$ consists of a variety $X$ and an effective $\bQ$-divisor $\Delta$ such that $K_X+\Delta$ is $\bQ$-Cartier. The notions of klt and lc singularities are defined as in \cite{Kol-mmp}*{Definition 2.8}. The non-lc center $\Nlc(X,\Delta)$ of a pair $(X,\Delta)$ is the set of closed points $x\in X$ such that $(X,\Delta)$ is not lc at $x$. If $\pi:Y\to X$ is a projective birational morphism and $E$ is a prime divisor on $Y$, then we say $E$ is a divisor over $X$. A valuation on $X$ will mean a valuation $v\colon \bC(X)^*\to \bR$ that is trivial on $\bC^*$. We write $C_X(E)$ (resp. $C_X(v)$) for the center of a divisor (resp. valuation) and $A_{X,\Delta}(E)$ (resp. $A_{X,\Delta}(v)$) for the log discrepancy of the divisor $E$ (resp. the valuation $v$) with respect to the pair $(X,\Delta)$ (see \cites{JM-valuation,BdFFU-valuation}). We write $\Val_X^*$ for the set of nontrivial valuations. Let $(X,\Delta)$ be a klt pair, $Z\subseteq X$ a closed subset (may be reducible) and $D$ an effective divisor on $X$, we denote by $\lct_Z(X,\Delta;D)$ the largest number $\lambda\ge 0$ such that $\Nlc(X,\Delta+\lambda D)$ does not contain $Z$. Given a $\bQ$-divisor $D$ on $X$, we set
\[
H^0(X,D):=\{0\neq s\in \bC(X)\,|\,\ddiv(s)+D\ge 0\}\cup \{0\}
\]
whose members can be viewed as effective $\bQ$-divisors that are $\bZ$-linearly equivalent to $D$. In particular if $D$ is $\bQ$-Cartier then $\ord_E(s):=\ord_E(\ddiv(s)+D)$ is well-defined for any $0\neq s\in H^0(X,D)$ and any divisor $E$ over $X$. We also define the sheaf $\cO_X(D)$ by localizing the above construction.

\subsection{K-stability and stability thresholds} \label{sec:K-defn}

Let $(X,\Delta)$ be a projective pair and let $L$ be a big  $\bQ$-Cartier $\bQ$-divisor on $X$. We denote by $M(L)$ the set of integers $m\in\bN_+$ such that $H^0(X,mL)\neq \{0\}$. 

\begin{defn}
Notation as above. Let $m\in M(L)$ and let $V\subseteq H^0(X,mL)$ be a linear series. We say that $D$ is a basis type divisor of $V$ if $D=\sum_{i=1}^N \{s_i=0\}$ for some basis $s_1,\cdots,s_N$ of $V$ (where by abuse of notation $\{s_i=0\}$ refers to the $\bQ$-divisor $\ddiv(s_i)+mL$). By convention, this means $D=0$ if $V=\{0\}$. We say that $D$ is an $m$-basis type $\bQ$-divisor of $L$ if $D=\frac{1}{m\cdot h^0(X,mL)}D_0$ for some basis type divisor $D_0$ of $H^0(X,mL)$ (in particular, $D\sim_\bQ L$).
\end{defn}

\begin{defn} \label{defn:S and T}
Let $m\in M(L)$ and let $v\in \Val_X^*$. In the above notation, we set 
\[
S_m(L;v) = \sup_{D\sim_\bQ L,\,m\text{-basis type}}v(D)
\]
where the supremum runs over all $m$-basis type $\bQ$-divisor of $L$. We define $S(L;v)$ to be the limit $\lim_{m\to \infty} S_m(L;v)$, which exists by \cites{BC-okounkov-body,BJ-delta}. We also define the pseudo-effective threshold as
\[
T(L;v)=\sup\{\lambda\ge 0\,|\,\vol(L;v\ge t)>0\}
\]
where 
\[
\vol(L;v\ge t)=\lim_{m\to \infty}\frac{\dim \{s\in H^0(X,mL)\,|\,v(s)\ge mt\}}{m^{\dim X}/(\dim X)!}.
\]
We say that $v$ is of linear growth if $T(L;v)<\infty$ (e.g. when $v$ is divisorial or have finite discrepancy; see \cite{BSKM-linear-growth}*{Section 2.3} and \cite{BJ-delta}*{Section 3.1}). 
By \cite{BJ-delta}*{Theorem 3.3}, for any valuation $v$ of linear growth we have
\[
S(L;v)=\frac{1}{\vol(L)}\int_0^\infty \vol(L;v\ge t) \rd t
\]
% We also set
% \begin{align*}
%     S(L;E) & = \frac{1}{\vol(L)}\int_0^\infty \vol(\pi^*L-tE) \rd t, \\
%     T(L;E) & = \sup \{t\ge 0\,|\,\pi^*L-tE \text{ is pseudo-effective}\}
% \end{align*}
where $\vol(L)$ denotes the volume of the divisor $L$ (see e.g. \cite{Laz-positivity-1}*{Section 2.2.C}). If $E$ is a divisor over $X$, we put $S(L;E)=S(L;\ord_E)$ and $T(L;E)=T(L;\ord_E)$.
% \[
% T(L;E) = \sup \{t\ge 0\,|\,\pi^*L-tE \text{ is pseudo-effective}\}.
% \]
% By \cites{BC-okounkov-body,BJ-delta}, we have $S(L;E)=\lim_{m\to \infty} S_m(L;E)$. 
We will simply write $S_m(E)$, $S(E)$, etc. if the divisor $L$ is clear from the context.
\end{defn}

\begin{defn} \label{defn:beta}
Let $(X,\Delta)$ be a log Fano pair, i.e. $(X,\Delta)$ is klt and $-(K_X+\Delta)$ is ample. We say $(X,\Delta)$ is K-semistable (resp. K-stable) if 
\[
\beta_{X,\Delta}(E):=A_{X,\Delta}(E) - S(-K_X-\Delta;E)\ge 0
\]
(resp. $\beta_{X,\Delta}(E) > 0$) for all divisors $E$ over $X$. We say that $(X,\Delta)$ is uniformly K-stable if 
\[
\beta_{X,\Delta}(v):=A_{X,\Delta}(v)-S(-K_X-\Delta;v)>0
\]
for all $v\in \Val_X^*$ such that $A_{X,\Delta}(v)<\infty$.
\end{defn}

By \cites{Fuj-valuative-criterion,Li-equivariant-minimize,BX-separatedness,BJ-delta}, this is equivalent to the original definition \cites{Tian-K-stability-defn,Don-K-stability-defn,BHJ-DH-measure,Der-uKs} of K-stability notions in terms of test configurations.

\begin{defn}
Let $(X,\Delta)$ be a klt pair and let $L$ be a $\bQ$-Cartier big divisor on $X$. The (adjoint) stability threshold (or $\delta$-invariant) of $L$ is defined as 
\begin{equation} \label{eq:delta as A/S}
    \delta(L) = \delta(X,\Delta;L) = \inf_E \frac{A_{X,\Delta}(E)}{S(L;E)}
\end{equation}
where the infimum runs over all divisors $E$ over $X$. Equivalently \cite{BJ-delta}, it can also be defined as the limit $\delta(L)=\lim_{m\to \infty}\delta_m(L)$ where 
\begin{equation} \label{eq:delta_m}
    \delta_m(L) = \sup \{\lambda\ge 0\,|\, (X,\Delta+\lambda D) \text{ is lc for all } m\text{-basis type }\bQ\text{-divisors } D\sim_\bQ L \}.
\end{equation}
We say that a divisor $E$ over $X$ computes $\delta(L)$ if it achieves the infimum in \eqref{eq:delta as A/S}. When $(X,\Delta)$ is log Fano, we write $\delta(X,\Delta)$ (or $\delta(X)$ when $\Delta=0$) for $\delta(-K_X-\Delta)$.
\end{defn}

We also introduce a local version of stability thresholds.

\begin{defn}
Let $(X,\Delta)$ be a klt pair and let $L$ be a $\bQ$-Cartier big divisor on $X$. Let $Z$ be a closed subset of $X$. We set
\[
\delta_{Z,m}(L)=\sup\{\lambda\ge 0\,|\,Z\not\subseteq \Nlc(X,\Delta+\lambda D) \text{ for all } m\text{-basis type }\bQ\text{-divisors } D\sim_\bQ L\}
\]
and define the (adjoint) stability threshold of $L$ along $Z$ as $\delta_Z(L)=\limsup_{m\to\infty}\delta_{Z,m}(L)$.
When $Z$ is irreducible, it is not hard to see (by a similar argument as in \cite{BJ-delta}*{\S 4}; see also Lemma \ref{lem:delta=inf A/S}) that the above limsup is a limit and we have
\[
\delta_Z(L) = \inf_{E,Z\subseteq C_X(E)} \frac{A_{X,\Delta}(E)}{S(L;E)} = \inf_{v,Z\subseteq C_X(v)} \frac{A_{X,\Delta}(v)}{S(L;v)}
\]
where the first infimum runs over all divisors $E$ over $X$ whose center contains $Z$, and the second infimum runs over all valuations $v\in\Val_X^*$ such that $A_{X,\Delta}(v)<\infty$ and $Z\subseteq C_X(v)$. If in addition $L$ is ample, then the second infimum is a minimum by (the same proof of) \cite{BJ-delta}*{Theorem E}. As in the global case, we then say that $E$ (resp. $v$) computes $\delta_Z(L)$ if it achieves the above infimum. When $(X,\Delta)$ is log Fano, we also write $\delta_Z(X,\Delta)$ (or $\delta_Z(X)$ when $\Delta=0$) for $\delta_Z(-K_X-\Delta)$.
\end{defn}

\subsection{Plt-type divisors}

\begin{defn} \label{defn:plt type div}
Let $(X,\Delta)$ be a pair and let $F$ be a divisor over $X$. When $F$ is a divisor on $X$ we write $\Delta=\Delta_1+aF$ where $F\not\subseteq \Supp(\Delta_1)$; otherwise let $\Delta_1=\Delta$.
\begin{enumerate}
	\item $F$ is said to be primitive over $X$ if there exists a projective birational morphism $\pi:Y\to X$ such that $Y$ is normal, $F$ is a prime divisor on $Y$ and $-F$ is a $\pi$-ample $\bQ$-Cartier divisor. We call $\pi:Y\to X$ the associated prime blowup (it is uniquely determined by $F$).
	\item $F$ is said to be of plt type if it is primitive over $X$ and the pair $(Y,\Delta_Y+F)$ is plt in a neighbourhood of $F$, where $\pi:Y\to X$ is the associated prime blowup and $\Delta_Y$ is the strict transform of $\Delta_1$ on $Y$. When $(X,\Delta)$ is klt and $F$ is exceptional over $X$, $\pi$ is called a plt blowup over $X$.
\end{enumerate}
\end{defn}

\begin{lem} \label{lem:D|_F along plt boundary}
Let $(Y,F+\Delta)$ be a plt pair with $\lfloor F+\Delta \rfloor =F$. Then for any $\bQ$-Cartier Weil divisor $D$ on $Y$, there exists a uniquely determined $\bQ$-divisor class $($i.e $\bQ$-divisor up to $\bZ$-linear equivalence$)$ $D|_F$ on $F$ and a canonical isomorphism
\[
\cO_Y(D)/\cO_Y(D-F)\cong \cO_F(D|_F).
\] 
\end{lem}

\begin{proof}
The $\bQ$-divisor class $D|_F$ is defined in \cite{HLS-acc-mld}*{Definition A.2 and A.4} by localizing at every codimension $1$ point of $F$ and the isomorphism is established by \cite{HLS-acc-mld}*{Lemma A.3}.
\end{proof}

\subsection{Filtrations and admissible flags} \label{sec:prelim-flag}

We recall the notation of filtrations as well as some constructions from the study of Okounkov bodies.

\begin{defn}
Let $V$ be a finite dimensional vector space. A filtration $\cF$ on $V$ is given by a collection of subspaces $\cF^\lambda V$ indexed by a totally ordered abelian monoid $\Lambda$ (in which case we also call the filtration a $\Lambda$-filtration) such that $\cF^{\lambda_0} V=V$, $\cF^{\lambda_1} V=0$ for some $\lambda_0,\lambda_1\in \Lambda$ and $\cF^\lambda V\subseteq \cF^{\lambda'} V$ whenever $\lambda\ge \lambda'$. When $\Lambda=\bR$, we will also require that the filtration is left continuous, i.e. for any $\lambda\in \bR$, we have $\cF^{\lambda-\varepsilon}V = \cF^\lambda V$ for all $0<\varepsilon\ll 1$. For each $\lambda\in\Lambda$, we set $\Gr_\cF^\lambda V = \cF^\lambda V / \cup_{\mu>\lambda}\cF^\mu V$. A basis $s_1,\cdots,s_N$ (where $N=\dim V$) of $V$ is said to be compatible with $\cF$ if every $\cF^\lambda V$ is the span of some $s_i$.
\end{defn}

Most filtrations we use are induced by a divisor or an admissible flag.

\begin{expl} \label{ex:filtration from div}
Let $L$ be a $\bQ$-Cartier $\bQ$-divisor on $X$ and let $V\subseteq H^0(X,L)$ be a subspace. Let $E$ be a divisor over $X$. Then it induces an $\bR$-filtration $\cF_E$ on $V$ by setting
\[
\cF_E^\lambda V := \{s\in V\,|\,\ord_E(s)\ge \lambda\}.
\]
More generally, every valuation $v$ on $X$ induces a filtration $\cF_v$ on $V$ with $\cF_v^\lambda V := \{s\in V\,|\,v(s)\ge \lambda\}$. 
\end{expl}

\begin{defn}[\cite{LM-okounkov-body}]
Let $X$ be a variety. An admissible flag $\Y$ over $X$ of length $\ell\le \dim X$ is defined as a flag of subvarieties
\[
\Y \quad:\quad Y=Y_0\supseteq Y_1\supseteq \cdots \supseteq Y_\ell
\]
on some projective birational model $\pi:Y\to X$ of $X$ where each $Y_i$ is an (irreducible) subvariety of codimension $i$ in $Y$ that is smooth at the generic point of $Y_\ell$.
\end{defn}

Given an admissible flag $\Y$ over $X$ as above and a $\bQ$-divisor $L$ on $X$ that is Cartier at the generic point of $Y_\ell$, one can define a valuation-like function
\begin{equation} \label{eq:nu}
    \nu=\nu_{\Y}=\nu_{\Y,X} : \left( H^0(X,L)\setminus\{0\} \right) \to \bN^\ell, \quad s \mapsto \nu(s)=\left(\nu_1(s),\cdots,\nu_\ell(s)\right)
\end{equation}
as follows: first $\nu_1=\nu_1(s)=\ord_{Y_1}(s)$; over an open neighborhood $U\subseteq Y$ of the generic point of $Y_\ell$, $s$ naturally determines a section $\tilde{s}_1\in H^0(U,\cO_U(\pi^*L-\nu_1 Y_1))$ which restricts to a non-zero section $s_1\in H^0(Y_1\cap U,\cO_{Y_1\cap U}(\pi^*L-\nu_1 Y_1))$; we set $\nu_2(s)=\ord_{Y_2}(s_1)$ and continue in this way to define the remaining $\nu_i(s)$ inductively. Via the lexicographic ordering on $\bZ^\ell$, every flag $\Y$ over $X$ induces a filtration $\cF_{\Y}$ (indexed by $\bN^\ell$) on $V=H^0(X,L)$ by setting
\[
\cF_{\Y}^\lambda V=\{s\in V\,|\,\nu(s)\ge \lambda\}.
\]
We also define the graded semigroup of $L$ (with respect to $\Y$) as the sub-semigroup
\[
\Gamma(L)=\Gamma_{\Y}(L)=\{(m,\nu_{\Y}(s))\,|\,m\in\bN, 0\neq s\in H^0(X,mL)\}
\]
of $\bN\times \bN^\ell = \bN^{\ell+1}$. The Okounkov body $\Delta(L)=\Delta_{\Y}(L)$ of $L$ is then the base of the closed convex cone $\Sigma(L)=\Sigma_{\Y}(L)\subseteq \bR^{\ell+1}$ spanned by $\Gamma(L)$, i.e. $\Delta(L)=\Sigma(L)\cap (\{1\}\times\bR^\ell)$.

For later use, we introduce some more notations. For a subspace $V\subseteq H^0(X,L)$ and an effective Weil divisor $E$ on some birational model $\pi:Y\to X$ of $X$, we set $V(-E):=V\cap H^0(Y,\pi^*L(-E))\subseteq H^0(X,L)$. Let $\Y$ is an admissible flag over $X$ of length $r$. Assume that $L$ is Cartier and that each $Y_i$ in the flag is a Cartier divisor in $Y_{i-1}$. Then for every $s$-tuple $(1\le s\le \ell)$ of integers $\va=(a_1,\cdots,a_s)\in\bN^s$, following \cite{Jow-NObody} we define
\[
V(\va)\subseteq H^0(Y_s, L\otimes \cO_{Y_s}(-a_1 Y_1 - a_2 Y_2 - \cdots -a_s Y_s))
\]
inductively so that $V(a_1)=V(-a_1 Y_1)|_{Y_1}$ and 
\[
V(a_1,\cdots,a_s) = V(a_1,\cdots,a_{s-1})(-a_s Y_s)|_{Y_s} \; (2\le s\le \ell).
\]
Note that $\cF_{\Y}$ induces a filtration on $V(a_1,\cdots,a_s)$ indexed by $\bN^{\ell-s}$.

\subsection{Multi-graded linear series} \label{sec:prelim-graded series}

\begin{defn}[\cite{LM-okounkov-body}*{\S 4.3}] \label{defn:multi-graded}
Let $L_1,\cdots,L_r$ be $\bQ$-Cartier $\bQ$-divisors on $X$. An $\bN^r$-graded linear series $W_{\vec{\bullet}}$ on $X$ associated to the $L_i$'s consists of finite dimensional subspaces 
\[
W_{\va} \subseteq H^0(X, \cO_X(a_1 L_1+\cdots+a_r L_r))
\]
for each $\va\in \bN^r$ such that $W_{\vec{0}}=\bC$ and $W_{\va_1}\cdot W_{\va_2} \subseteq W_{\va_1+\va_2}$ for all $\va_1,\va_2\in \bN^r$. The support $\Supp(\W)\subseteq \bR^r$ of $W_{\vec{\bullet}}$ is defined as the closed convex cone spanned by all $\va\in\bN^r$ such that $W_{\va}\neq 0$. We say that $\W$ has bounded support if $\Supp(\W)\cap (\{1\}\times \bR^{r-1})$ is bounded.
For such $\W$, we set
\[
h^0(W_{m,\vb}):=\sum_{\va\in\bN^{r-1}} \dim (W_{m,\va})
\]
for each $m\in\bN$ (it is a finite sum when $\W$ has bounded support) and define the volume of $\W$ as (where $n=\dim X$)
\[
\vol(\W):=\limsup_{m\to \infty} \frac{h^0(W_{m,\vb})}{m^{n+r-1}/(n+r-1)!}.
\]
We say that $\W$ contains an ample series if the following conditions are satisfied:
    \begin{enumerate}
        \item $\Supp(\W)\subseteq \bR^r$ contains a non-empty interior,
        \item for any $\va\in {\rm int}(\Supp(\W))\cap \bN^r$, $W_{k\va}\neq 0$ for $k\gg 0$,
        \item there exists some $\va_0\in {\rm int}(\Supp(\W))\cap \bN^r$ and a decomposition $\va_0\cdot \vec{L} = A+E$ (where $\vec{L}=(L_1,\cdots,L_r)$) with $A$ an ample $\bQ$-line bundle and $E$ an effective $\bQ$-divisor such that $H^0(X,mA)\subseteq W_{m\va_0}$ for all sufficiently divisible $m$.
    \end{enumerate}
If $\Y$ is an admissible flag of length $\ell$ over $X$ such that $L_1,\cdots,L_r$ are \emph{Cartier} at the generic point of $Y_\ell$, the multi-graded semigroup of $\W$ with respect to $\Y$ is defined to be
\[
\Gamma(\W)=\Gamma_{\Y}(\W):=\{(\va,\nu(s))\,|\,0\neq s\in W_{\va}\}\subseteq \bN^r\times \bN^\ell = \bN^{r+\ell}.
\]
\end{defn}

\begin{rem} \label{rem:NObody}
Note that the above definition is slightly more general than \cite{LM-okounkov-body} since we allow divisors $L_i$ that may not be Cartier or integral. However, most results of \cite{LM-okounkov-body}*{\S 4.3} carry over to our setting. In particular, when $\W$ contains an ample series, one can verify as in \cite{LM-okounkov-body}*{Lemma 4.20} that $\Gamma(\W)$ generates $\bZ^{r+\ell}$ as a group. If in addition $W$ has bounded support, then we can define the associated Okounkov body $\Delta(\W)=\Delta_{\Y}(\W)$ as $\Sigma(\W)\cap (\{1\}\times \bR^{r-1+\ell})$ where $\Sigma(\W)$ is the closed convex cone spanned by $\Gamma(\W)$. When $\ell=n=\dim X$, we let $\Gamma_m = \Gamma(\W)\cap (\{m\}\times \bN^{r-1+n})$ and let
\begin{equation} \label{eq:rho_m}
    \rho_m = \frac{1}{m^{r-1+n}} \sum_{a\in \Gamma_m} \delta_{m^{-1}a}
\end{equation}
be the atomic positive measure on $\Delta(\W)$. Then by \cite{Bou-NObody}*{Th\'eor\`eme 1.12}, $\rho_m$ converges weakly as $m\to \infty$ to the Lebesgue measure on $\Delta(\W)$. In particular, we have $\vol(\W)=(n+r-1)!\cdot \vol(\Delta(\W))$ as in \cite{LM-okounkov-body}*{Theorem 2.13}. By \cite{LM-okounkov-body}*{Corollary 4.22}, there is also a continuous function
\begin{equation} \label{eq:vol_W}
    \vol_{\W}\colon {\rm int}(\Supp(\W))\to \bR
\end{equation}
such that for any integer vector $\va\in {\rm int}(\Supp(\W))$, $\vol_{\W}(\va)$ equals the volume of the graded linear series $\{W_{m\va}\}_{m\in\bN}$.
\end{rem}

We give some examples of multi-graded linear series that naturally arise in our later analysis (i.e. when applying inversion of adjunction to basis type divisors that are compatible with a given divisor or admissible flag). The following lemma ensures that the graded linear series we construct contains an ample series.

\begin{lem} \label{lem:semigroup prpperty}
Let $\W$ be an $\bN^r$-graded linear series on $X$ with bounded support and containing an ample series. Then for any admissible flag $\Y$ of length $\ell$ over $X$ such that $L_1,\cdots,L_r$ are Cartier at the generic point of $Y_\ell$ and any $\gamma\in {\rm int}(\Sigma(\W))\cap \bN^{r+\ell}$, we have $k\gamma \in \Gamma(\W)$ when $k\gg 0$.
\end{lem}

\begin{proof}
By \cite{LM-okounkov-body}*{Lemma 4.20} as in the previous remark, the semigroup $\Gamma(\W)$ generates $\bZ^{r+\ell}$ as a group. Let $\Gamma\subseteq \Gamma(\W)$ be a finitely generated sub-semigroup that still generates $\bZ^{r+\ell}$ and such that $\gamma\in {\rm int}(\Sigma)$, where $\Sigma\subseteq \Sigma(\W)$ is the subcone generated by $\Gamma$. By \cite{Kho-polytope}*{Proposition 3}, there exists some $\gamma_0\in \Gamma$ such that
\[
(\Sigma+\gamma_0)\cap \bN^{r+\ell} \subseteq \Gamma \subseteq \Gamma(\W).
\]
As $\gamma\in {\rm int}(\Sigma)$, we have $k\gamma\in \Sigma+\gamma_0$ when $k\gg 0$, thus the lemma follows.
\end{proof}

\begin{expl} \label{ex:complete series}
Let $L$ be a big line bundle on $X$. The \emph{complete linear series associated to} $L$ is the $\bN$-graded linear series $\V$ on $X$ defined by $V_m=H^0(X,mL)$. It is clear that $\V$ has bounded support and contains an ample series.
\end{expl}

\begin{expl} \label{ex:refinement by div}
Let $L_1,\cdots,L_r$ be \emph{Cartier} divisors on $X$ and let $\V$ be an $\bN^r$-graded linear series associated to the $L_i$'s. Denote $\vec{L}=(L_1,\cdots,L_r)$. Let $F$ be a primitive divisor over $X$ with associated prime blowup $\pi:Y\to X$ and let $\cF$ be the induced filtration on $\V$ (see Example \ref{ex:filtration from div}). Assume that $F$ is either Cartier on $Y$ or of plt type. In the latter case, we define $F|_F$ as the $\bQ$-divisor class given by Lemma \ref{lem:D|_F along plt boundary}. Then in both cases,
\[
W_{\va,j}=\cF^j V_{\va}/\cF^{j+1} V_{\va}
\]
can be naturally identified with the image of $\cF^j V_{\va}$ under the composition 
\[
\cF^j V_{\va}\to H^0(Y,\pi^*(\va\cdot \vec{L})-jF)\to H^0(F,\pi^*(\va\cdot \vec{L})|_F-jF|_F)
\]
(this is clear if $F$ is Cartier on $Y$; when $F$ is of plt type, we use Lemma \ref{lem:D|_F along plt boundary}). It follows that $W_{\vb}$ is an $\bN^{r+1}$-graded linear series on $F$ (associated to the divisors $\pi^*L_1|_F,\cdots,\pi^*L_r|_F$ and $-F|_F$), called the \emph{refinement} of $\V$ by $F$. It is not hard to see that $\W$ has bounded support if $\V$ does (see e.g. \cite{LM-okounkov-body}*{Remark 1.12}). We show that $\W$ contains an ample series if $\V$ does. Indeed, condition (1) and (3) are easy to verify as $\V$ contains an ample series. For condition (2), consider the admissible flag $Y_0=Y$, $Y_1=F$, then we see that $W_{\va,j}\neq 0$ if and only if $(\va,j) \in \Gamma_{\Y}(\V)$ and hence condition (2) follows from Lemma \ref{lem:semigroup prpperty}. 
\end{expl}

% \begin{expl} \label{ex:graded series from div}
% Let $L$ be a big line bundle on $X$ and let $F$ be a primitive divisor over $X$ with associated prime blowup $\pi:Y\to X$. Assume that $F$ is either Cartier on $Y$ or of plt type. In the latter case, we define $F|_F$ as the $\bQ$-divisor class given by Lemma \ref{lem:D|_F along plt boundary}. Then in both cases, 
% \[
% W_{m,j}:=\text{Im}(H^0(Y,m\pi^*L-jF)\to H^0(F, m\pi|_F^*L-jF|_F))
% \]
% defines an $\bN^2$-graded linear series on $F$, called the graded linear series associated to $L$ and $F$. It is not hard to see that $\Supp(\W)\cap (\{1\}\times \bR)\subseteq [0,T(L;F)]$, hence $\W$ has bounded support. We show that $\W$ also contains an ample series. Condition (1) and (3) are obvious since $L$ is big. For condition (2), let $\cF$ be the filtration on $V_m=H^0(X,mL)$ induced by $F$ (see Example \ref{ex:filtration from div}), then $W_{m,j}$ can be naturally identified with $\Gr_\cF^j V_m$ (this is clear if $F$ is Cartier on $Y$; when $F$ is of plt type, we use Lemma \ref{lem:D|_F along plt boundary}); now consider the admissible flag $Y_0=Y$, $Y_1=F$, then we see that $W_{\va}\neq 0$ if and only if $\va \in \Gamma_{\Y}(L)$ and hence condition (2) follows from Lemma \ref{lem:semigroup prpperty}.
% \end{expl}

\begin{expl} \label{ex:refinement by flag}
More generally, let $L_1,\cdots,L_r$ be Cartier divisors on $X$, let $\V$ be an $\bN^r$-graded linear series associated to the $L_i$'s and let $\Y$ be an admissible flag of length $\ell$ over $X$. Assume that each $Y_i$ in the flag is a Cartier divisor in $Y_{i-1}$. Then in the notation of Section \ref{sec:prelim-flag},
\[
W_{\va,b_1,\cdots,b_\ell}=V_{\va}(b_1,\cdots,b_\ell)
\]
defines an $\bN^{r+\ell}$-graded linear series on $Y_\ell$. We call it the \emph{refinement} of $\V$ by $\Y$. As in the previous example, one can check that $\W$ has bounded support (resp. contains an ample series) if $\V$ does.
\end{expl}

% \begin{expl} \label{ex:graded series from flag}
% Let $L$ be a big line bundle on $X$ and let $\Y$ be an admissible flag of length $r$ over $X$. Assume that each $Z_i$ in the flag is a Cartier divisor in $Z_{i-1}$. Let $V_m=H^0(X,mL)$. Then $W_{m,\va}=V_m(\va)$ defines an $\bN^{r+1}$-graded linear series on $Y_r$. We call it the graded linear series associated to $L$ and $\Y$. As in the previous example, one can check that $\W$ has bounded support and contains an ample series.
% \end{expl}

\subsection{Invariants associated to filtered multi-graded linear series} \label{sec:prelim-filtered graded series}

\begin{defn}
Let $\W$ be an $\bN^r$-graded linear series. A filtration $\cF$ on $\W$ (indexed by $\Lambda$) is given by a filtration on each $W_{\va}$ ($\va\in\bN^r$) such that $\cF^{\lambda_1} W_{\va_1}\cdot \cF^{\lambda_2} W_{\va_2} \subseteq \cF^{\lambda_1+\lambda_2} W_{\va_1+\va_2}$ for all $\lambda_i\in\Lambda$ and all $\va_i\in\bN^r$. If $\Lambda\subseteq \bR$, we say the filtration $\cF$ is linearly bounded if there exist constants $C_1$ and $C_2$ such that $\cF^\lambda W_{\va}=W_{\va}$ for all $\lambda<C_1|\va|$ and $\cF^\lambda W_{\va}=0$ for all $\lambda>C_2|\va|$.
\end{defn}

One can generalize the definition of basis type divisors, $S$-invariants and stability thresholds to filtered multi-graded linear series.

\begin{defn} \label{defn:basis type for multi-graded}
Let $\W$ be an $\bN\times \bN^r$-graded linear series with bounded support. Let $M(\W)$ be the set of $m\in\bN_+$ such that $W_{m,\va}\neq 0$ for some $\va \in \bN^r$. Let $m\in M(\W)$ and let $N_m=h^0(W_{m,\vb})$. We say that $D$ is an $m$-basis type divisor (resp. $\bQ$-divisor) of $\W$ if there exist basis type divisors $D_{\va}$ of $W_{m,\va}$ for each $\va\in\bN^r$ such that
\[
D = \sum_{\va\in\bN^r} D_{\va} \quad \text{resp. } D = \frac{1}{mN_m} \sum_{\va\in\bN^r} D_{\va}.
\]
When $r=0$ and $\W$ is the complete linear series associated to $L$, this reduces to the usual definition of $m$-basis type ($\bQ$-)divisors of $L$ (c.f. Section \ref{sec:K-defn}). Let $\cF$ be a filtration on $\W$ and let $D$ be an $m$-basis type ($\bQ$-)divisor of $\W$. We say that $D$ is compatible with $\cF$ if all the $D_{\va}$ above has the form $D_{\va}=\sum_{i=1}^N \{s_i=0\}$ for some basis $s_i$ ($i=1,\cdots,N$) of $W_{m,\va}$ that is compatible with $\cF$. In particular, we say that $D$ is compatible with a divisor $E$ (resp. an admissible flag $\Y$) if it is compatible with the filtration induced by $E$ (resp. $\Y$). Note that the divisor class $c_1(D)\in \Cl (X)_\bQ$ of an $m$-basis type divisor does not depend on the choice of $D$. We denote it by $c_1(W_{m,\vb})$.
\end{defn}

\begin{defn}
Let $(X,\Delta)$ be a klt pair and let $Z$ be a closed subset of $X$. Let $\W$ be an $\bN\times \bN^r$-graded linear series on $X$ with bounded support, let $\cF,\cG$ be filtrations on $\W$ and let $v\in\Val_X^*$ be a valuation on $X$. Assume that $\cG$ is a linearly bounded, left continuous $\bR$-filtration and $A_{X,\Delta}(v)<\infty$. Associated to $\cG$ we have a valuation-like function $v_{\cG}\colon \W\to \bR$ given by
\[
s\in W_{\va} \mapsto \sup\{\lambda\in\bR\,|\,s\in \cG^\lambda W_{\va}\}.
\]
If $D=\frac{1}{mN_m}\sum_{i=1}^{N_m} \{s_i=0\}$ is an $m$-basis type $\bQ$-divisor $D$ of $\W$, where each $s_i\in W_{m,\va}$ for some $\va\in\bN^r$, then we define 
\[
v_{\cG}(D)=\frac{1}{mN_m}\sum_{i=1}^{N_m}v_{\cG}(s_i).
\]
Clearly $v_{\cG}=v$ if $\cG=\cF_v$ is the filtration induced by the valuation $v$. Similar to \S \ref{sec:K-defn}, for each $m\in M(\W)$ we set
\[
S_m(\W,\cF;\cG) = \sup_D v_{\cG}(D), \quad S_m(\W,\cF;v)=S_m(\W,\cF;\cF_v)=\sup_D v(D)
\]
where the supremum runs over all $m$-basis type $\bQ$-divisors $D$ of $\W$ that are compatible with $\cF$. We also set
\begin{align*}
    % S_m(\W,\cF;v) & = \sup_D v(D) \\
    \delta_m(\W,\cF) = \delta_m(X,\Delta;\W,\cF) & = \inf_D \lct(X,\Delta;D) \\
    \delta_{Z,m}(\W,\cF) = \delta_{Z,m}(X,\Delta;\W,\cF) & = \inf_D \lct_Z(X,\Delta;D)
\end{align*}
where the infimum runs over all $m$-basis type $\bQ$-divisors $D$ of $\W$ that are compatible with $\cF$. We then define 
\[
S(\W,\cF;\cG)=\limsup_{m\to\infty} S_m(\W,\cF;\cG),\quad S(\W,\cF;v)=S(\W,\cF;\cF_v)
\]
and similarly the (adjoint) stability thresholds $\delta(\W,\cF)$ (resp. $\delta_Z(\W,\cF)$) of a filtered multi-graded linear series $\W$. If $E$ is a divisor over $X$, we set $S(\W,\cF;E)=S(\W,\cF;\ord_E)$ and $S_m(\W,\cF;E)=S_m(\W,\cF;\ord_E)$. When the filtration $\cF$ is trivial (i.e. $\cF^\lambda W_{\va}$ equals $W_{\va}$ when $\lambda\le 0$ and is $0$ when $\lambda>0$), we simply write $S(\W;\cG)$, $\delta(\W)$, $\delta_Z(\W)$, etc. 
\end{defn}

\begin{rem}
When $L$ is a big line bundle on $X$ and $\W$ is the complete linear series associated to $L$, we have $S(\W;v)=S(L;v)$ for any valuation $v$ on $X$; similarly $\delta(\W)=\delta(L)$ and $\delta_Z(\W)=\delta_Z(L)$ for any closed subset $Z\subseteq X$.
\end{rem}

The following statement is the direct generalization of \cite{BJ-delta} to multi-graded linear series.

\begin{lem} \label{lem:delta=inf A/S}
Let $(X,\Delta)$ be a klt pair and let $Z\subseteq X$ be a subvariety. Let $\W$ be an $\bN\times \bN^r$-graded linear series with bounded support which contains an ample series. Then $S(\W;\cF)=\lim_{m\to \infty} S_m(\W;\cF)$ for any linearly bounded, left continuous $\bR$-filtration $\cF$ on $\W$ and we have
\[
\delta(\W) = \inf_E \frac{A_{X,\Delta}(E)}{S(\W;E)}=\inf_v \frac{A_{X,\Delta}(v)}{S(\W;v)}\quad \text{resp.}\quad 
\delta_Z(\W) = \inf_{E,Z\subseteq C_X(E)} \frac{A_{X,\Delta}(E)}{S(\W;E)}=\inf_{v,Z\subseteq C_X(v)} \frac{A_{X,\Delta}(v)}{S(\W;v)}
\]
where the first infimum runs over all divisors $E$ over $X$ $($resp. all divisors $E$ over $X$ whose center contains $Z)$, and the second infimum runs over all valuations $v\in\Val_X^*$ $($resp. all valuations $v\in\Val_X^*$ whose center contains $Z)$ such that $A_{X,\Delta}(v)<\infty$. Moreover, it holds that
\[
\delta(\W)=\lim_{m\to\infty}\delta_m(\W)\quad \text{and}\quad
\delta_Z(\W)=\lim_{m\to\infty}\delta_{Z,m}(\W).
\]
\end{lem}

In view of this lemma, we say that a divisor $E$ over $X$ (or a valuation $v\in \Val_X^*$) computes $\delta(\W)$ (resp. $\delta_Z(\W)$) if it achieves the above infimum.

\begin{proof}
The argument is almost identical to those in \cite{BJ-delta} (which is in turn based on \cite{BC-okounkov-body}). 
% Let $v\in\Val_X^*$ be a valuation with $A_{X,\Delta}(v)<\infty$ and let $\cF$ be its induced $\bR$-filtration on $\W$, i.e. $\cF^\lambda W_{\va} = \{s\in W_{\va}\,|\,v(s)\ge \lambda\}$. 
Using the filtration $\cF$, we define a family $\W^t$ of multi-graded linear series on $X$ (indexed by $t\in\bR$) where $W^t_{m,\va}=\cF^{mt} W_{m,\va}$. Set 
\[
T_m(\W;\cF)=\max\{j\in\bN\,|\,\cF^j W_{m,\va}\neq 0 \mbox{ for some } \va\}.
\]
It is easy to see that the sequence $T_m(\W;\cF)$ is super-additive and we set 
\[
T(\W;\cF)=\lim_{m\to \infty} \frac{T_m(\W;\cF)}{m} = \sup_{m\in \bN} \frac{T_m(\W;\cF)}{m}.
\]
One can check as in \cite{BC-okounkov-body}*{Lemma 1.6} that for any $t< T(\W;\cF)$ the multi-graded linear series $\W^t$ contains an ample series. Therefore, for any fixed admissible flag $\Y$ of length $n=\dim X$ centered at a general point of $X$, we have the associated Okounkov bodies $\Delta^t=\Delta_{\Y}(\W^t)$ ($t\in\bR$). The result is now simply a consequence of properties of Okounkov bodies. More precisely, consider the function $G\colon \Delta:=\Delta^0\to [0,T(\W;\cF)]$ given by 
\[
G(\gamma)=\sup\{t\in\bR\,|\,\gamma\in \Delta^t\}.
\]
It is straightforward to check that $G$ is concave and hence continuous in the interior of $\Delta$. By the exact same proof of \cite{BJ-delta}*{Lemma 2.9} (using \cite{BC-okounkov-body}*{Theorem 1.11}), we get the equality (where $\rho$ is the Lebesgue measure on $\Delta^0$)
\[
S(\W;\cF)=\frac{1}{\vol(\Delta)}\int_{\Delta}G \rd \rho = \lim_{m\to \infty} S_m(\W;\cF)
\]
and an estimate
\[
S_m(\W;\cF)\le \frac{m^{n+r}}{h^0(W_{m,\vb})} \int_{\Delta}G\rd \rho_m
\]
where $\rho_m$ is as in \eqref{eq:rho_m} (note that $\Delta=\Delta(\W)$). Applied to $\cF=\cF_v$, the argument of \cite{BJ-delta}*{Lemma 2.2 and Corollary 2.10} then implies that for any $\epsilon>0$, there exists some $m_0=m_0(\epsilon)$ such that $S_m(\W;v)\le (1+\epsilon)S(\W;v)$ for any valuation $v\in\Val_X^*$ with $A_{X,\Delta}(v)<\infty$ and any $m\ge m_0$ (the key point is that $m_0$ doesn't depend on $v$). The remaining equalities in the lemma now follow from the exact same proof of \cite{BJ-delta}*{Theorem 4.4}.
\end{proof}

The above proof also gives a formula for the $S$-invariants of multi-graded linear series, similar to the one in Definition \ref{defn:S and T}. 

\begin{cor} \label{cor:S formula}
Notation as above. Then $S(\W;\cF)=\frac{1}{\vol(\W)}\int_0^\infty \vol(\W^t) \rd t$. 
\end{cor}

\begin{proof}
We already have $S(\W;\cF)=\frac{1}{\vol(\Delta)}\int_{\Delta}G \rd \rho $. It is not hard to see that $\int_{\Delta}G \rd \rho = \int_0^\infty \vol(\Delta^t)\rd t$. Since $\vol(\W)=(n+r)!\cdot \vol(\Delta)$ and $\vol(\W^t)=(n+r)!\cdot \vol(\Delta^t)$ for all $t\ge 0$ (see Remark \ref{rem:NObody}), the result follows.
\end{proof}

We also provide a more explicit formula for the volumes $\vol(\W^t)$. To this end, let $\W$ and $\cF$ be as in Lemma \ref{lem:delta=inf A/S}, let $\Delta^t_{\supp}=\Supp(\W^t)\cap (\{1\}\times \bR^r)$ and let 
\[
\vol_{\W^t}\colon {\rm int}(\Delta^t_{\supp})\to \bR
\]
be the volume function as in \eqref{eq:vol_W}. Then we have

\begin{lem} \label{lem:vol(W^t)}
$\vol(\W^t) = \frac{(n+r)!}{n!} \int_{\Delta^t_{\supp}} \vol_{\W^t}(\gamma)\rd \gamma$.
\end{lem}

\begin{proof}
Let ${\rm pr} \colon \bR^{r+1+n}\to \bR^{r+1}$ be the projection to the first $r+1$ coordinates which induces a map $p\colon \Delta^t \to \Delta^t_{\supp}$. By \cite{LM-okounkov-body}*{Theorem 2.13 and 4.21}, we know that $\vol(\W^t)=(n+r)!\cdot \vol(\Delta^t)$ and $\vol_{\W^t}(\gamma)=n!\cdot\vol(p^{-1}(\gamma))$ for all $\gamma\in {\rm int}(\Delta^t_{\supp})$. The lemma then follows from the obvious identity $\vol(\Delta^t)=\int_{\Delta^t_{\supp}} \vol(p^{-1}(\gamma)) \rd \gamma$.
\end{proof}

Recall that for any $\bQ$-Cartier big divisor $L$ on $X$ and any integer $k>0$ we have $\delta(kL)=\frac{1}{k} \delta(L)$. This can be generalized to multi-graded linear series as follows. 
Let $L_1,\cdots,L_r$ be $\bQ$-Cartier $\bQ$-divisors on $X$ and let $\W$ be an $\bN^r$-graded linear series associated to them. Let $k>0$ be an integer such that $kL_i$ is Cartier for all $1\le i\le r$. Set $W'_{\va}=W_{k\va}$ ($\va\in\bN^r$), then $\W'$ is an $\bN^r$-graded linear series associated to $kL_1,\cdots,kL_r$.

\begin{lem} \label{lem:delta homogeneous}
In the above notation, assume that $\W$ contains an ample series and has bounded support. Then
\begin{enumerate}
    \item $S(\W';v)=k\cdot S(\W;v)$ for any valuation $v$ on $X$;
    \item $\delta(\W)=k\cdot\delta(\W')$ and $\delta_Z(\W)=k\cdot\delta_Z(\W')$ for any subvariety $Z$ of $X$.
\end{enumerate}
\end{lem}

In particular, this implies that for the calculation of stability thresholds, we only need to consider multi-graded linear series associated to \emph{Cartier divisors}.

\begin{proof}
We use the same notation as in the proof of Lemma \ref{lem:delta=inf A/S} and let $\Delta'$, $G'$ etc. be the counterparts on $\W'$. Let $f\colon \bR^{r+n}\to \bR^{r+n}$ be given by $$(x_1,\cdots,x_{r+n})\mapsto (kx_1,\cdots,kx_r,x_{r+1},\cdots,x_{r+n}).$$
We claim that
\begin{equation} \label{eq:OKbody of W vs W'}
    \Sigma(\W)=f(\Sigma(\W')).
\end{equation}
Indeed, it is clear from the construction that $f(\Gamma(\W'))\subseteq \Gamma(\W)$, hence $f(\Sigma(\W'))\subseteq \Sigma(\W)$. On the other hand, from the proof of Lemma \ref{lem:semigroup prpperty} we know that there exists some $\gamma_0\in \Gamma(\W)$ such that
\[
(\Sigma(\W)+\gamma_0)\cap \bN^{r+n} \subseteq \Gamma(\W),
\] 
hence as $f(\Gamma(\W'))=f(\bN^{r+n})\cap \Gamma(\W)$ we have $(\Sigma(\W)+\gamma_0)\cap f(\bN^{r+n})\subseteq f(\Gamma(\W'))$ and therefore $\Sigma(\W)\subseteq f(\Sigma(\W'))$, which proves the claim.

It follows from \eqref{eq:OKbody of W vs W'} that $\Delta(\W)=\frac{1}{k}f(\Delta(\W'))$ (recall that we identify $\Delta(\W)$ as a subset of $\{1\}\times \bR^{r-1+n}$). Replace $\W$ by $\W^{t/k}$, noting that $W'^t_{m,\va}:=\cF^{mt}W'_{m,\va}=W^{t/k}_{km,k\va}$, we deduce $\Delta^{t/k}=\frac{1}{k}f(\Delta'^t)$. Hence $\Delta=\frac{1}{k}f(\Delta')$ and 
\begin{equation} \label{eq:G of W vs W'}
    G\left(\frac{f(\gamma)}{k}\right)=\frac{G'(\gamma)}{k}
\end{equation}
for any $\gamma\in \Delta'$. Substitute it into the equality  $S(\W;v)=\frac{1}{\vol(\Delta)}\int_{\Delta}G \rd \rho $ from the proof of Lemma \ref{lem:delta=inf A/S} we obtain $S(\W;v)=\frac{1}{k}S(\W';v)$. The remaining parts of the lemma now follow immediately from Lemma \ref{lem:delta=inf A/S}.
\end{proof}

To further analyze basis type divisors of $\W$, for each $\va\in\bN^{r+1}$ with $W_{\va}\neq 0$ we let $M_{\va}$ (resp. $F_{\va}$) be the movable (resp. fixed) part of the linear system $|W_{\va}|$. Thus we have a decomposition $|W_{\va}|=|M_{\va}|+F_{\va}$. % We then say that $M_{\vb}$ (resp. $F_{\vb}$) is the $G$-movable (resp. $G$-fixed) part of $\W$.  
For each $m\in M(\W)$, let
\[
F_m = F_{m} (\W) := \frac{1}{m\cdot h^0(W_{m,\vb})}\sum_{\va\in\bN^r} \dim (W_{m,\va})\cdot F_{m,\va}.
\]
Then it is clear that every $m$-basis type $\bQ$-divisor $D$ of $\W$ can be decomposed as $D=D'+F_m$ where $D'$ is an $m$-basis type $\bQ$-divisor of $M_{\vb}$ (the definition of basis type divisors works for any collection of linear series indexed by $\bN\times \bN^r$). We next study the asymptotic behaviour of $D'$ and $F_m$.

\begin{lemdefn}
Let $L_0,\cdots,L_r$ be $\bQ$-Cartier $\bQ$-divisors on $X$ and let $\W$ be an associated $\bN\times \bN^r$-graded linear series which has bounded support and contains an ample series. Then in the notation of Definition \ref{defn:basis type for multi-graded}, the limit
\[
c_1(\W):=\lim_{m\to \infty} \frac{c_1(W_{m,\vb})}{m\cdot h^0(W_{m,\vb})}
\]
exists in $\Pic(X)_\bR$. Similarly, $\lim_{m\to\infty} \ord_D F_{m}$ exists for any prime divisor $D\subseteq X$. We will formally write 
\[
F(\W):=\sum_D (\lim_{m\to\infty} \ord_D (F_{m}))\cdot D.
\]
When this is a finite sum, we set $c_1(M_{\vb}):=c_1(\W)-F(\W)\in \Cl(X)_\bR$. 
\end{lemdefn}

\begin{proof}
Let $\vec{L}=(L_1,\cdots,L_r)$. In the notation of Definitions \ref{defn:multi-graded} and \ref{defn:basis type for multi-graded}, we have
\[
\frac{c_1(W_{m,\vb})}{m\cdot h^0(W_{m,\vb})} = L_0 + \frac{\sum_{\va\in \bN^r} h^0(W_{m,\va})\cdot (\va\cdot \vec{L})}{m\cdot h^0(W_{m,\vb})}.
\]
Thus for $c_1(\W)$ it suffices to show that $\lim_{m\to \infty} \frac{\sum_{\va\in \bN^r} h^0(W_{m,\va})\cdot a_i}{m\cdot h^0(W_{m,\vb})}$ exists for each $1\le i\le r$. In the notation of Remark \ref{rem:NObody}, we have
\[
\frac{\sum_{\va\in \bN^r} h^0(W_{m,\va})\cdot a_i}{m\cdot h^0(W_{m,\vb})} = \frac{\int x_i \rd \rho_m}{\int \rd \rho_m} 
\]
where $x_i$ denotes the $i$-th entry of an element of $\bR^{r+n}$. Hence by \cite{Bou-NObody}*{Th\'eor\`eme 1.12} the limit exists and equals $\frac{1}{\vol(\Delta)}\int_{\Delta} x_i \rd \rho$ where $\Delta = \Delta(\W)$.

For $F(\W)$ it suffices to show that $\lim_{m\to \infty}\ord_D(F_m)$ exists for any prime divisor $D$. First note that since $\W$ has bounded support, there exists some constant $C_1>0$ such that $|\va|\le C_1 m$ for any $\va\in\bN^{r}$ with $W_{m,\va}\neq 0$. Thus as $mL_0+\va\cdot \vec{L}-F_{m,\va}$ is effective, we further deduce $\ord_D(F_{m,\va})\le Tm$ for some absolute constant $T$. Let $\Delta_0:={\rm int}(\Supp(\W))\cap (\{1\}\times\bR^r)\subseteq \bR^r$.  Since $W_{\va}\cdot W_{\va'}\subseteq W_{\va+\va'}$, we have $F_{\va}+F_{\va'}\ge F_{\va+\va'}$ (whenever $W_{\va},W_{\va'}\neq 0$), thus if we let 
\[
f_{\W,D}(\gamma):=\inf_{m} \frac{\ord_D(F_{m,m\vec{\gamma}})}{m} = \lim_{m\to\infty} \frac{\ord_D(F_{m,m\vec{\gamma}})}{m}
\]
for $\gamma\in \Delta_0 \cap \bQ^r$ where the infimum and the limit are taken over sufficiently divisible integers $m$, then $f_{\W,D}(t\gamma_1+(1-t)\gamma_2)\le tf_{\W,D}(\gamma_1)+(1-t)f_{\W,D}(\gamma_2)$ for any $\gamma_1,\gamma_2\in \Delta_0$. Therefore it naturally extends to a convex (and hence continuous) function $f_{\W,D}$ on $\Delta_0$. For simplicity we denote $f_{\W,D}$ by $f$. By the previous discussion, $f(\gamma)\le T$ for all $\gamma\in\Delta_0$. 

We claim that $f(\gamma)=\lim_{m\to\infty} f_m(\gamma)$ for any $\gamma\in \Delta_0$ where 
\[
f_m(\gamma):=
\begin{cases}
    \frac{1}{m}\ord_D (F_{m,\lfloor m\vec{\gamma}\rfloor}) & \text{if } W_{m,\lfloor m\vec{\gamma}\rfloor}\neq 0 \\
    T & \text{if } W_{m,\lfloor m\vec{\gamma}\rfloor} = 0.
\end{cases}
\]
Indeed, as $f_m(\gamma)\ge f(\frac{\lfloor m\vec{\gamma}\rfloor}{m})$ ($m\gg 0$) by definition, we have 
\[
\liminf_{m\to \infty} f_m(\gamma)\ge \lim_{m\to\infty} f(\frac{\lfloor m\vec{\gamma}\rfloor}{m}) = f(\gamma).
\]
To get the reverse direction, let $\varepsilon>0$ and choose $\gamma_i\in \Delta_0\cap \bQ^r$ ($i=0,\cdots,r$) that are sufficiently close to $\gamma$ such that their convex hull contains $\gamma$ in the interior and $f(\gamma_i) < f(\gamma)+\varepsilon$. Then we may choose some sufficiently divisible $m_0\in\bN$ such that $f_{m_0}(\gamma_i)<f(\gamma)+\varepsilon$. Let $\Pi\subseteq \bR^{r+1}$ be the cone spanned by all the $\gamma_i$'s. From the proof of Lemma \ref{lem:semigroup prpperty}, we know that there exists some $\va_0\in \bN^{r+1}$ such that $W_{\va}\neq 0$ for all $\va\in (\Pi+\va_0)\cap \bN^{r+1}$ (consider the semigroup $\{\va\,\vert\,W_{\va}\neq 0\}\subseteq \bN^{r+1}$, choose a finitely generated sub-semigroup that generates $\bZ^{r+1}$ such that the cone it spans contains $\Pi$, and apply \cite{Kho-polytope}*{Proposition 3}). Then one can verify that there exists some constant $C>0$ such that for all $m\gg 0$, we have
\[
(m,\lfloor m\vec{\gamma}\rfloor) = \va + \sum_{i=0}^r k_i (m_0,m_0\gamma_i)
\]
for some $k_i\in\bN$ and some $\va\in\bN^{r+1}$ satisfying $W_{\va}\neq 0$ and $|\va|\le C$. In particular $|m-m_0\sum k_i|\le C$ and $\ord_D(W_{\va})\le CT$. It follows that
\begin{align*}
    \ord_D (F_{m,\lfloor m\vec{\gamma}\rfloor}) 
    & \le \ord_D (W_{\va})+\sum_{i=0}^r k_i \ord_D(F_{m_0,m_0\gamma_i}) = \ord_D (W_{\va})+\sum_{i=0}^r k_i m_0 f_{m_0}(\gamma_i) \\ 
    & \le \ord_D (W_{\va})+(m-C)(f(\gamma)+\varepsilon) \le CT+ (m-C)(f(\gamma)+\varepsilon).
\end{align*}
Hence $\limsup_{m\to\infty} f_m(\gamma)\le f(\gamma)+\varepsilon$. Since $\varepsilon>0$ is arbitrary, we get $\limsup_{m\to\infty} f_m(\gamma)\le f(\gamma)$ and this proves the claim. Note that the argument also shows $W_{m,\lfloor m\vec{\gamma}\rfloor}\neq 0$ for $m\gg 0$.

It is clear that
\[
\ord_D(F_m) = \frac{\int (f_m\circ p) \rd \rho_m}{\int \rd \rho_m}
\]
where $p\colon \Delta=\Delta(\W)\to \Delta_0$ is the natural projection. By dominated convergence and the above claim, the latter limit exists and equals $\frac{1}{\vol(\Delta)}\int_\Delta (f\circ p) \rd \rho$.
\end{proof}

For later calculations, we extract a formula for $F(\W)$ from the above proof.

\begin{cor} \label{cor:F(W) formula}
Let $\W$ be an $\bN\times \bN^r$-graded linear series on $X$ which has bounded support and contains an ample series and let $D$ be a prime divisor. Then
\[
\ord_D(F(\W)) = \frac{(n+r)!}{n!}\cdot  \frac{1}{\vol(\W)}\int_{\Delta_{\supp}} f(\gamma)\vol_{\W}(\gamma) \rd \gamma
\]
where $\Delta_{\supp}=\Supp(\W)\cap (\{1\}\times \bR^r)$, $f(\gamma)=f_{\W,D}(\gamma):=\lim_{m\to \infty} \frac{1}{m}\ord_D(F_{m,\lfloor m\vec{\gamma} \rfloor})$, $n=\dim X$ and $\vol_{\W}(\cdot)$ is as in \eqref{eq:vol_W}.
\end{cor}

\begin{proof}
The above proof gives $\ord_D(F(\W))=\frac{1}{\vol(\Delta)}\int_\Delta (f\circ p) \rd \rho$. We have $\vol(\W)=(n+r)!\cdot \vol(\Delta)$, $p(\Delta)=\Delta_{\supp}$ and $\vol_{\W}(\gamma)=n!\cdot \vol(p^{-1}(\gamma))$ for any $\gamma\in {\rm int}(\Delta_{\supp})$. These together imply the given formula.
\end{proof}

Most multi-graded linear series considered in this paper come from the refinement of some complete linear series by a divisor or a flag. To simplify computations, we often carefully choose the divisor (or flag) so that the corresponding multi-graded linear series behaves like complete linear systems associated to multiples of a fixed line bundle. 

\begin{defn} \label{defn:almost complete}
Let $L$ be a big line bundle on $X$ and let $\W$ be an $\bN\times \bN^r$-graded linear series. % Let $|W_{\va}|=|M_{\va}|+F_{\va}$ be the decomposition into $G$-movable and $G$-fixed part.
We say that $\W$ is \emph{almost complete} (with respect to $L$) if the following two conditions are both satisfied:
\begin{enumerate}
    \item there are at most finitely many prime divisors $D\subseteq X$ with $\ord_D (F(\W))>0$ (so that $F(\W)$ is an $\bR$-divisor),
    \item for every $\vg\in \bQ^r$ in the interior of $\Delta_{\supp}:=\Supp(\W)\cap (\{1\}\times \bR^r)$ and all sufficiently divisible integers $m$ (depending on $\vg$), we have $|M_{m,m\vg}|\subseteq|L_{m,\vg}|$ for some $L_{m,\vg}\equiv \ell_{m,\vg}L$ and some $\ell_{m,\vg}\in \bN$ (where $M_{\vb}$ is the movable part of $\W$) such that 
\[
\frac{h^0(W_{m,m\vg})}{h^0(X,\ell_{m,\vg}L)}=\frac{h^0(M_{m,m\vg})}{h^0(X,\ell_{m,\vg}L)}\to 1
\]
as $m\to \infty$.
\end{enumerate}
\end{defn}

% $|M_{m,m\vg}|\subseteq|\ell_{m,\vg}L|$ for some $\ell_{m,\vg}\in \bN$ (where $M_{\vb}$ is the movable part of $\W$) and 
% \[
% \frac{h^0(W_{m,m\vg})}{h^0(X,\ell_{m,\vg}L)}=\frac{h^0(M_{m,m\vg})}{h^0(X,\ell_{m,\vg}L)}\to 1
% \]
% as $m\to \infty$.

\begin{expl} \label{ex:graded series from ample div}
Let $L$ be an ample line bundle on $X$, and let $H\in|L|$. Assume that $H$ is irreducible and reduced. Let $\V$ be the complete linear series associated to $rL$ for some positive integer $r$ and let $\W$ be its refinement by $H$ (Example \ref{ex:refinement by div}). Then the $\bN^2$-graded linear series $\W$ is almost complete. Indeed, we have $W_{m,j}=|(mr-j)L|_H$; but since $L$ is ample, the natural restriction $H^0(X,kL)\to H^0(H,kL|_H)$ is surjective when $k\gg 0$, hence $W_{m,j}=|(mr-j)L_0|$ (where $L_0=L|_H$) and $F_{m,j}=0$ when $mr-j\gg 0$, so the conditions of Definition \ref{defn:almost complete} are satisfied and $F(\W)=0$. More generally, if $\Y$ is an admissible flag \emph{on} $X$ (i.e. $Y_0=X$) such that each $Y_i$ is Cartier on $Y_{i-1}$ and $Y_i\sim m_i L|_{Y_{i-1}}$ for some $m_i\in\bN$, then the refinement of $\V$ by $\Y$ (Example \ref{ex:refinement by flag}) is almost complete as well.
\end{expl}

\begin{lem} \label{lem:S of ac W}
Let $L$ be a big line bundle on $X$ and let $\W$ be an $\bN\times \bN^r$-graded linear series. Assume that $\W$ has bounded support, contains an ample series and is almost complete with respect to $L$. Then 
\begin{enumerate}
    \item $F(\W)$ is $\bR$-Cartier $($i.e. it is an $\bR$-linear combination of Cartier divisors$)$,
    \item there exists a constant $\mu=\mu(X,L,\W)$ such that $c_1(M_{\vb})=\mu L$ in $\NS(X)_\bR$ and
    \begin{equation} \label{eq:S(W)=S(L)+F}
	    S(\W;v) = \mu\cdot S(L;v) + v (F(\W))
    \end{equation}
    for all valuations $v\in\Val_X^*$ of linear growth. 
\end{enumerate}
\end{lem}

\begin{proof}
Let $M(\gamma):=\lim_{m\to\infty} \frac{1}{m} c_1(M_{m,\lfloor m\vec{\gamma}\rfloor})\in \Cl(X)_\bR$ and $F(\gamma)=\lim_{m\to\infty} \frac{1}{m} F_{m,\lfloor m\vec{\gamma}\rfloor}$ for $\gamma\in {\rm int}(\Delta_{\supp})$. As in the previous proof, the limit exists: indeed $M(\gamma)=\vec{\gamma}\cdot \vec{L}-\sum_D f_{\W,D}(\gamma)\cdot D$ and $F(\gamma)=\sum_D f_{\W,D}(\gamma)\cdot D$ in the notation of Corollary \ref{cor:F(W) formula}. Moreover, $M$ is continuous and we have
\[
c_1(M_{\vb})=\frac{1}{\vol(\Delta)}\int_\Delta (M\circ p) \rd \rho 
%\quad \text{and} \quad \ord_E(F(\W))=\frac{1}{\vol(\Delta)}\int_\Delta (h\circ p) \rd \rho
\]
where $\Delta=\Delta(\W)$ and $p\colon \Delta\to \Delta_{\supp}$ is the natural projection. Since $\W$ is almost complete, we see that $M(\gamma)$ is $\bR$-Cartier and $M(\gamma) \equiv g(\gamma)L$ for some $g(\gamma)\in\bR$. It follows that $c_1(M_{\vb})$ is also $\bR$-Cartier and $c_1(M_{\vb})=\mu L$ in $\NS(X)_\bR$ where
\[
\mu=\frac{1}{\vol(\Delta)}\int_\Delta (g\circ p) \rd \rho=\frac{1}{\vol(\Delta)}\int_{\Delta_{\supp}} \vol(p^{-1}(\gamma))\cdot g(\gamma) \rd \gamma.
\]
Since $F(\W)\sim_\bR c_1(\W)-c_1(M_{\vb})$, we also see that $F(\W)$ is $\bR$-Cartier. It remains to prove \eqref{eq:S(W)=S(L)+F}.

As $F(\gamma)\sim_\bR \vec{\gamma}\cdot \vec{L}-M(\gamma)$ is also $\bR$-Cartier, we may define $h(\gamma)=v(F(\gamma))$ and as in the proof of Corollary \ref{cor:F(W) formula} we have
\[
v(F(\W))=\frac{1}{\vol(\Delta)}\int_\Delta (h\circ p) \rd \rho.
\]
% Let $\pi\colon Y\to X$ be a log resolution that extracts the divisor $E$. 
We claim that 
\begin{equation} \label{eq:compare vol(M) to vol(L)}
    \vol_{W^t_{\vb}}(\gamma)=\vol(g(\gamma)L;v\ge t-h(\gamma))
\end{equation}
in the notation of Corollary \ref{cor:S formula} and Lemma \ref{lem:vol(W^t)}. For this we may assume that $\gamma\in\bQ^r$. Let $\cF$ be the filtration induced by $v$ and let $m$ be a sufficiently divisible integer. From the exact sequence 
\[
0\to \cF^{\lambda}M_{m,m\vg}\to \cF^{\lambda}H^0(X,L_{m,\vg})\to H^0(X,L_{m,\vg})/M_{m,m\vg}
\]
and the obvious equality 
\[
|\cF^{mt}W_{m,m\vg}|=|\cF^{mt-v(F_{m,m\vg})}M_{m,m\vg}|+F_{m,m\vg}
\]
we deduce that
\begin{equation} \label{eq:section dim estimate}
    \left| \dim (\cF^{mt}W_{m,m\vg}) - \dim (\cF^{mt-v(F_{m,m\vg})}H^0(X,L_{m,\vg})) \right| \le h^0(X,L_{m,\vg})-h^0(M_{m,m\vg}).
\end{equation}
By \cite{Laz-positivity-1}*{Lemma 2.2.42}, there exists a fixed effective divisor $N$ on $X$ such that $N\pm (L_{m,\vg} - \ell_{m,\vg}L)$ is effective. In particular we have the inclusions
\[
H^0(X,\ell_{m,\vg}L-N)\hookrightarrow H^0(X,L_{m,\vg})\hookrightarrow H^0(X,\ell_{m,\vg}L+N),
\]
\[
H^0(X,\ell_{m,\vg}L-N)\hookrightarrow H^0(X,\ell_{m,\vg}L)\hookrightarrow H^0(X,\ell_{m,\vg}L+N),
\]
which implies
\[
\lim_{m\to\infty} \frac{\dim (\cF^{mt-v(F_{m,m\vg})}H^0(X,L_{m,\vg}))}{m^n/n!} = \lim_{m\to\infty} \frac{\dim (\cF^{mt-v(F_{m,m\vg})}H^0(X,\ell_{m,\vg}L))}{m^n/n!}
\]
% \[
% H^0(X,\ell_{m,\vg}L-N)\hookrightarrow H^0(X,\ell_{m,\vg}L)\hookrightarrow H^0(X,\ell_{m,\vg}L+N).
% \]
Thus as we divide \eqref{eq:section dim estimate} by $m^n/n!$ and letting $m\to \infty$, the right side of the inequality becomes $0$ by the definition of almost completeness and the equality \eqref{eq:compare vol(M) to vol(L)} follows as $g(\gamma)=\lim_{m\to\infty}\frac{1}{m}\ell_{m,\vg}$. 

By Corollary \ref{cor:S formula} and Lemma \ref{lem:vol(W^t)}, we have \[
S(\W;v) = \frac{1}{n!\vol(\Delta)} \iint_{\Delta_{\supp}\times \bR_+} \vol_{W^t_{\vb}}(\gamma) \rd t \rd \gamma.
\]
Combined with \eqref{eq:compare vol(M) to vol(L)}, we then obtain
\begin{align*}
    S(\W;v) & = \frac{1}{n!\vol(\Delta)} \iint_{\Delta_{\supp}\times \bR_+} \vol(g(\gamma)L; v\ge t-h(\gamma)) \rd t \rd \gamma \\
    & = \frac{1}{n!\vol(\Delta)} \int_{\Delta_{\supp}}\left(\int_0^{h(\gamma)}+\int_{h(\gamma)}^\infty\right)\vol(g(\gamma)L; v\ge t-h(\gamma)) \rd t \rd \gamma \\
    & = \frac{1}{n!\vol(\Delta)} \int_{\Delta_{\supp}}\left(h(\gamma)\cdot g(\gamma)^n\vol(L) + \int_0^\infty g(\gamma)^{n+1} \vol(L; v\ge t) \rd t\right) \rd \gamma.
\end{align*}
Notice that $\vol_{\W}(\gamma) = \vol_{\W^0}(\gamma)=g(\gamma)^n\vol(L)$ by \eqref{eq:compare vol(M) to vol(L)}, thus we deduce that
\begin{align*}
    S(\W;v) & = \frac{1}{n!\vol(\Delta)} \int_{\Delta_{\supp}} \vol_{W_{\vb}}(\gamma)\cdot (h(\gamma)+g(\gamma)S(L;v)) \rd \gamma\\
    & = \frac{1}{\vol(\Delta)} \int_{\Delta_{\supp}} \vol(p^{-1}(\gamma))\cdot (h(\gamma)+g(\gamma)S(L;v)) \rd \gamma\\
    & = v(F(\W))+\mu\cdot S(L;v).
\end{align*}
This finishes the proof.
\end{proof}

\begin{cor} \label{cor:delta on curve}
Let $C$ be a smooth curve and let $\W$ be an almost complete multi-graded linear series on $C$ that has bounded support and contains an ample series. Then
\[
\delta_P(C;\W) = \frac{2}{\deg (c_1(\W)-F(\W)) + 2\cdot \mult_P F(\W)}
\]
for all closed point $P\in C$. In particular, $\delta(C;\W)=\frac{2}{\deg c_1(\W)}$ if $F(\W)=0$.
\end{cor}

\begin{proof}
We have $S(L;P)=\frac{1}{2}\deg L$ for any ample line bundle $L$ and any closed point $P$ on $C$. Combining with Lemma \ref{lem:S of ac W}, we see that $S(\W;P)=S(c_1(\W)-F(\W);P)+\mult_P F(\W) = \frac{1}{2}\deg (c_1(\W)-F(\W))+\mult_P F(\W)$. Since $\delta_P(C;\W)=\frac{1}{S(\W;P)}$ and $\delta(C;\W)=\inf_{P\in C}\delta_P(C;\W)$ by definition, the result follows.
\end{proof}

\section{Adjunction for stability thresholds} \label{sec:adj}

In this section, we develop a framework to estimate stability thresholds. The starting point is the following elementary observation (c.f. \cite{BE-compatible-basis}*{Proposition 1.14}).

\begin{lem} \label{lem:basis compatible with two filtrations}
Let $V$ be a finite dimensional vector space and let $\cF$, $\cG$ be two filtrations on $V$. Then there exists some basis $s_1,\cdots,s_N$ of $V$ that is compatible with both $\cF$ and $\cG$.
\end{lem}

\begin{proof}
By enumerating all different subspaces $\cF^\lambda V$ and $\cG^\mu V$, we may assume that $\cF$ and $\cG$ are both $\bN$-filtrations. Note that $\cF$ (resp. $\cG$) induces a filtration (which is also denoted by $\cF$ resp. $\cG$) on each graded quotient $\Gr_\cG^i V$ (resp. $\Gr_\cF^j V$). It is not hard to check that
\[
\Gr_\cF^j \Gr_\cG^i V \cong (\cF^j V\cap \cG^i V) / (\cF^{j+1}V\cap \cG^i V + \cF^j V\cap \cG^{i+1}V) \cong \Gr_\cG^i \Gr_\cF^j V
\]
for each $i,j\in\bN$. To construct a basis of $V$ that is compatible with $\cF$, it suffices to lift a basis of each $\Gr_\cF^i V$ to $\cF^i V$ and take their union. In particular, we may lift basis of $\Gr_\cF^i V$ that are compatible with the induced filtration $\cG$. By the above isomorphism, such basis can be obtained by lifting a basis of $(\cF^j V\cap \cG^i V) / (\cF^{j+1}V\cap \cG^i V + \cF^j V\cap \cG^{i+1}V)$ to $\cF^j V\cap \cG^i V$ (for each $i,j\in\bN$) and then taking the union. But since the construction is symmetric in $\cF$ and $\cG$, it follows that the basis obtained in this way is also compatible with $\cG$.
\end{proof}

As an immediate consequence, we have

\begin{prop} \label{prop:delta = delta wrt filtration}
Let $(X,\Delta)$ be a pair and let $\V$ be a multi-graded linear series containing an ample series and with bounded support. Let $\cF$ be a filtration on $\V$. Then for any valuation $v$ of linear growth on $X$ and any subvariety $Z\subseteq X$ we have
\[
S(\V;v)=S(\V,\cF;v),\,\,\delta(\V)=\delta(\V,\cF),\,\,\text{and}\;\;\delta_Z(\V)=\delta_Z(\V,\cF).
\]
\end{prop}

\begin{proof}
It suffices to show that for any $m\in M(\V)$ we have
\[
S_m(\V;v)=S_m(\V,\cF;v),\,\,\delta_m(\V)=\delta_m(\V,\cF),\,\,\text{and}\;\;\delta_{Z,m}(\V)=\delta_{Z,m}(\V,\cF),
\]
the result then follows by taking the limit as $m\to \infty$. Let $\cF_v$ be the filtration on $\V$ induced by $v$ (see Example \ref{ex:filtration from div}). It is clear from the definition that $S_m(\V;v)=v(D)$ for any $m$-basis type $\bQ$-divisor $D$ of $\V$ that is compatible with $\cF_v$. In particular, if we choose an $m$-basis type $\bQ$-divisor $D$ of $\V$ that is compatible with both $\cF_v$ and $\cF$ (which exists by Lemma \ref{lem:basis compatible with two filtrations}), then we see that $S_m(\V;v)=v(D)\le S_m(\V,\cF;v)$. But the reverse inequality $S_m(\V,\cF;v)\le S_m(\V;v)$ is trivial and thus we prove the first equality $S_m(\V;v)=S_m(\V,\cF;v)$. By definition it is not hard to see that
\[
\delta_{Z,m}(\V)=\inf_E \frac{A_{X,\Delta}(E)}{S_m(\V;E)}\quad\text{and}\quad\delta_{Z,m}(\V,\cF)=\inf_E \frac{A_{X,\Delta}(E)}{S_m(\V,\cF;E)}
\]
where both infimums run over divisors $E$ over $X$ whose centers contain $Z$ (here we use the fact that $Z$ is irreducible), hence the equality $\delta_{Z,m}(\V)=\delta_{Z,m}(\V,\cF)$ follows. The proof of the equality $\delta_m(\V)=\delta_m(\V,\cF)$ is similar.
\end{proof}

Typically we will apply Proposition \ref{prop:delta = delta wrt filtration} to some Fano variety $X$ and the complete linear series associated to $-rK_X$ for some sufficiently divisible integer $r>0$. By choosing different filtrations $\cF$ on $\V$, we get various consequences. Here we explore two of them, corresponding to filtrations induced by primitive divisors or admissible flags. Throughout the remaining part of this section, we fix a klt pair $(X,\Delta)$, some Cartier divisors $L_1,\cdots,L_r$ on $X$ and an $\bN^r$-graded linear series $\V$ associated to the $L_i$'s such that $\V$ contains an ample series and has bounded support.

\subsection{Filtrations from primitive divisors} \label{sec:filter by div}

Let $F$ be a primitive divisor over $X$ with associated prime blowup $\pi:Y\to X$. Let $\cF$ be the induced filtration on $\V$ and let
\[
D=\frac{1}{mN_m}\sum_{\va}\sum_i \{s_{\va,i}=0\}
\]
(where $N_m=h^0(V_{m,\vb})$ and for each $\va\in\bN^{r-1}$, $s_{\va,i}$ ($1\le i\le \dim(V_{m,\va})$) form a basis of $V_{m,\va}$) be an $m$-basis type $\bQ$-divisor of $\V$ that is compatible with $\cF$. We may write
\[
D=\frac{1}{mN_m}\sum_{\va}\sum_{j=0}^{\infty}D'_{\va,j}
\]
where 
\[
D'_{\va,j}=\sum_{i,\ord_F(s_{\va,i})=j}\{s_{\va,i}=0\}.
\]
Since $D$ is compatible with $\cF$, for each $\va\in\bN^{r-1}$, the $s_{\va,i}$'s that appear in the expression of $D'_{\va,j}$ restrict to form a basis of $\Gr_\cF^j V_{m,\va}$. Now assume that $F$ is either Cartier on $Y$ or of plt type and let $\W$ be the refinement of $\V$ by $F$ (Example \ref{ex:refinement by div}). Then after combining coefficients of $F$ in $\pi^*D$, we see that
\[
\pi^*D= S_m(\V;F)\cdot F + \frac{1}{mN_m}\sum_{\va}\sum_{j=0}^{\infty} D_{\va,j} =: S_m(\V;F)\cdot F + \Gamma
\]
where each $D_{\va,j}$ doesn't contain $F$ in its support and $D_{\va,j}|_F$ is a basis type divisor for $W_{m,\va,j}$. In other words, $\Gamma|_F$ is an $m$-basis type $\bQ$-divisor of $\W$ (notice that $h^0(W_{m,\vb})=h^0(V_{m,\vb})$). Letting $m\to \infty$, we obtain
\begin{equation} \label{eq:c_1 of refinement}
    c_1(\W) = (\pi^*c_1(\V) - S(\V;F)\cdot F)|_F.
\end{equation}
These observations also allow us to relate the stability thresholds of $\V$ and $\W$ via inversion of adjunction. In particular, we get the following consequence:

\begin{thm} \label{thm:delta adj via div}
With the above notation and assumptions, let $Z\subseteq X$ be a subvariety and let $Z_0$ be an irreducible component of $Z\cap C_X(F)$. Let $\Delta_Y$ be the strict transform of $\Delta$ on $Y$ $($but remove the component $F$ as in Definition \ref{defn:plt type div}$)$ and let $\Delta_F=\mathrm{Diff}_F(\Delta_Y)$ be the different so that $(K_Y+\Delta_Y+F)|_F=K_F+\Delta_F$. Then we have
%Let $(X,\Delta)$ be a klt pair, let $F$ be a primitive divisor over $X$ with associated prime blowup $\pi:Y\to X$, let $Z\in X$ be a subvariety and let $Z_0$ be an irreducible component of $Z\cap C_X(F)$. Let $L_1,\cdots,L_r$ be line bundles on $X$ and let $\V$ be an associated $\bN^r$-graded linear series. Assume that $F$ is either Cartier on $Y$ or of plt type and that $\V$ has bounded support and contains an ample series. Let $\W$ be the refinement of $\V$ by $F$, let $\Delta_Y$ be the strict transform of $\Delta$ on $Y$ $($but remove the component $F$ as in Definition \ref{defn:plt type div}$)$ and let $\Delta_F=\mathrm{Diff}_F(\Delta_Y)$ be the different so that $(K_Y+\Delta_Y+F)|_F=K_F+\Delta_F$. Then we have
\begin{equation} \label{eq:adj for delta}
	\delta_Z(X,\Delta;\V)\ge \min\left\{\frac{A_{X,\Delta}(F)}{S(\V;F)},\inf_{Z'}\delta_{Z'}(F,\Delta_F;\W)\right\}
\end{equation}
when $Z\subseteq C_X(F)$ and otherwise
\begin{equation} \label{eq:adj for delta, Z not in F}
	\delta_Z(X,\Delta;\V)\ge \inf_{Z'}\delta_{Z'}(F,\Delta_F;\W),
\end{equation}
where the infimums run over all subvarieties $Z'\subseteq Y$ such that $\pi(Z')=Z_0$. Moreover, if equality holds and $\delta_Z(\V)$ is computed by some valuation $v$ on $X$, then either $Z\subseteq C_X(F)$ and $F$ computes $\delta_Z(\V)$, or $C_Y(v)\not\subseteq F$ and for any irreducible component $S$ of $C_Y(v)\cap F$ with $Z_0\subseteq \pi(S)$, there exists some valuation $v_0$ on $F$ with center $S$ computing $\delta_{Z'}(\W)=\delta_Z(\V)$ for all subvarieties $Z'\subseteq S$ with $\pi(Z')=Z_0$.
% If in addition $v=\ord_E$ for some divisor $E$ over $X$ and $F$ does not compute $\delta_Z(\V)$, then for any irreducible component $S$ of $C_Y(E)\cap F$ with $Z_0\subseteq \pi(S)$, there exists some divisor $E_0$ over $F$ with center $S$ computing $\delta_{Z'}(\W)=\delta_Z(\V)$ for all subvarieties $Z'\subseteq S$ with $\pi(Z')=Z_0$.
\end{thm}

Loosely speaking, this means that $\delta(\V)$ is either computed by the auxiliary divisor $F$ or bounded from below by the stability threshold $\delta(\W)$ of the refinement by $F$, and in the latter case the inequality is usually strict.

% \begin{thm} \label{thm:adj for delta}
% Let $(X,\Delta)$ be a klt pair, let $F$ be a primitive divisor over $X$ with associated prime blowup $\pi:Y\to X$ and let $x\in C_X(F)$. Let $L$ be a big line bundle on $X$. Assume that $F$ is either Cartier on $Y$ or of plt type. Let $\W$ be the graded linear series associated to $L$ and $F$, let $\Delta_Y$ be the strict transform of $\Delta$ on $Y$ and let $\Delta_F=\mathrm{Diff}_F(\Delta_Y)$ be the different so that $(K_Y+\Delta_Y+F)|_F=K_F+\Delta_F$. Then we have
% \begin{equation} \label{eq:adj for delta}
% 	\delta_x(X,\Delta;L)\ge \min\left\{\frac{A_{X,\Delta}(F)}{S(L;F)},\inf_{y\in\pi^{-1}(x)}\delta_y(F,\Delta_F;\W)\right\}.
% \end{equation}
% Moreover, if equality holds and $\delta_x(X,\Delta;L)$ is computed by some divisor $E$ over $X$, then either it is computed by $F$, or $C_Y(E)\not\subseteq F$ and for any irreducible component $Z$ of $C_Y(E)\cap F$ that intersects $\pi^{-1}(x)$, there exists some divisor $S$ over $F$ with center $Z$ computing $\delta_y(F,\Delta_F;\W)=\delta_x(X,\Delta;L)$ for all $y\in Z\cap \pi^{-1}(x)$.
% \end{thm}

\begin{proof}
We only prove \eqref{eq:adj for delta}, since the proof for \eqref{eq:adj for delta, Z not in F} is almost identical. By Proposition \ref{prop:delta = delta wrt filtration}, we have $\delta_Z(\V)=\delta_Z(\V,\cF)$ (where $\cF$ is the filtration on $\V$ induced by $F$), thus it suffices to show that
\begin{equation} \label{eq:adj for delta_m}
	\delta_{Z,m}(\V,\cF)\ge \min\left\{\frac{A_{X,\Delta}(F)}{S_m(\V;F)},\inf_{\pi(Z')=Z_0}\delta_{Z',m}(F,\Delta_F;\W)\right\}
\end{equation}
for all $m\in M(\V)$; letting $m\to \infty$ we obtain \eqref{eq:adj for delta}. Let $D$ be an $m$-basis type $\bQ$-divisor of $\V$ that's compatible with $F$. From the discussion before, we have 
\begin{equation} \label{eq:D=aF+?}
	\pi^*D = S_m(\V;F)\cdot F + \Gamma
\end{equation}
where $\Gamma=D_Y$ is the strict transform of $D$ on $Y$ and $\Gamma|_F$ is an $m$-basis type $\bQ$-divisor of $\W$. Let $\lambda_m$ (resp. $\lambda$) be the right hand side of \eqref{eq:adj for delta_m} (resp. \eqref{eq:adj for delta}). Then we have $\pi^*(K_X+\Delta+\lambda_m D)=K_Y+\Delta_Y+a_m F+\lambda_m \Gamma$ where $a_m=1-A_{X,\Delta}(F)+\lambda_m S_m(\V;F)\le 1$. In addition, the non-lc center of $(F,\Delta_F+\lambda_m \Gamma|_F)$ doesn't contain $Z_0\subseteq Z$ in its image (under the morphism $\pi$) by the definition of stability thresholds and hence by inversion of adjunction the same is true for $(Y,\Delta_Y+F+\lambda_m \Gamma)$. It follows that $(X,\Delta+\lambda_m D)$ is lc at the generic point of $Z$ and indeed
\[
A_{X,\Delta}(v)\ge \lambda_m v(D) + (1-a_m)\cdot v(F)
\]
for all valuations $v$ on $X$ whose center contains $Z$ (when $Z\not\subseteq C_X(F)$, the value of $a_m$ doesn't matter to us since $v(F)=0$; this is the main difference between the proof of \eqref{eq:adj for delta} and \eqref{eq:adj for delta, Z not in F}). Since $D$ is arbitrary, we get $\delta_{Z,m}(\V,\cF)\ge \lambda_m$, which proves \eqref{eq:adj for delta_m}, and
\[
A_{X,\Delta}(v)\ge \lambda_m  S_m(\V,\cF;v) + (A_{X,\Delta}(F) - \lambda_m S_m(\V;F))\cdot v(F),
\]
As in the proof of Lemma \ref{lem:delta=inf A/S}, we have $\lim_{m\to \infty} \lambda_m  = \lambda$. Thus letting $m\to \infty$ and noting that $S(\V;v)=S(\V,\cF;v)$ by Proposition \ref{prop:delta = delta wrt filtration}, we obtain \eqref{eq:adj for delta} as well as the following inequality
\begin{equation} \label{eq:refined A>=S}
A_{X,\Delta}(v)\ge \lambda \cdot  S(\V;v) + (A_{X,\Delta}(F) - \lambda \cdot  S(\V;F))\cdot v(F).
\end{equation}

Now assume that equality holds in \eqref{eq:adj for delta} and $\delta_Z(\V)$ is computed by some valuation $v\in\Val_X^*$, i.e. $Z\subseteq C_X(v)$ and $A_{X,\Delta}(v) = \lambda \cdot S(\V;v)$. By \eqref{eq:refined A>=S}, we see that either $A_{X,\Delta}(F) = \lambda \cdot S(\V;F)$, in which case $F$ computes $\delta_Z(\V)$ and we are done, or
\begin{equation} \label{eq:F doesn't compute delta}
    \lambda = \inf_{\pi(Z')=Z_0}\delta_{Z'}(\W) < \frac{A_{X,\Delta}(F)}{S(\V;F)}
\end{equation}
and $v(F)=0$, i.e. $C_Y(v)\not\subseteq F$. Now assume that we are in the latter case and let $S$ be an irreducible component of $C_Y(v)\cap F$ with $Z_0\subseteq \pi(S)$. After rescaling the valuation $v$ we may also assume that $A_{Y,\Delta_Y}(v)=A_{X,\Delta}(v)=1$. Let $\fa_\bullet(v)\subseteq \cO_Y$ be the valuation ideals and let $\fb_\bullet=\fa_\bullet(v)|_F$. Clearly $\lct_x(Y,\Delta_Y;\fa_\bullet(v))\le \frac{A_{Y,\Delta_Y}(v)}{v(\fa_\bullet(v))}\le 1$ for any $x\in C_Y(v)$, hence by inversion of adjunction we have $\lct(F,\Delta_F;\fb_\bullet)\le 1$ at the generic point of $S$. By \cite{JM-valuation}*{Theorem A}, there exists some valuation $v_0$ on $F$ with center $S$ such that
\begin{equation} \label{eq:v_0 computes lct}
    \frac{A_{F,\Delta_F}(v_0)}{v_0(\fb_\bullet)}\le 1.
\end{equation}
To finish the proof, it suffices to show that this valuation computes $\delta_{Z'}(\W)$ for any subvarieties $Z'\subseteq S$ with $\pi(Z')=Z_0$, and that $\delta_{Z'}(\W)=\lambda$. To see this, let $D$ be an $m$-basis type $\bQ$-divisor of $\V$ that's compatible with both $F$ and $v$ (which exists by Lemma \ref{lem:basis compatible with two filtrations}) and let $D_Y$ be its strict transform on $Y$. Then as before we have $v(D_Y)=v(D)=S_m(\V;v)$ (here we use the fact that $C_Y(v)\not\subseteq F$) and $D_Y|_F$ is an $m$-basis type $\bQ$-divisor of $\W$. Using \eqref{eq:v_0 computes lct} we further see that 
\[
A_{F,\Delta_F}(v_0)\le \frac{v_0(D_Y|_F)}{v(D_Y)}=\frac{v_0(D_Y|_F)}{S_m(\V;v)},
\]
hence $S_m(\W;v_0)\ge v_0(D_Y|_F)\ge S_m(\V;v)\cdot A_{F,\Delta_F}(v_0)$. Letting $m\to \infty$ we obtain
\[
\delta_{Z'}(\W)\le \frac{A_{F,\Delta_F}(v_0)}{S(\W;v_0)}\le \frac{1}{S(\V;v)}=\frac{A_{X,\Delta}(v)}{S(\V;v)}=\lambda.
\]
Combined with \eqref{eq:F doesn't compute delta}, this implies $\delta_{Z'}(\W)=\lambda$ and it's computed by $v_0$.
\end{proof}

Theorem \ref{thm:delta adj via div} reduces the question of estimating stability thresholds to similar problems in lower dimensions. Certainly the lower bounds we get depend on the choice of the auxiliary divisor $F$. In general, if we want to calculate the precise value of the stability threshold, we should pick an ``optimal'' $F$, i.e. a divisor that computes $\delta(\V)$, although the resulting refinement $\W$ can be quite complicated. On the other hand, if we are merely interested in an estimate, we could also choose some divisor $F$ such that $\W$ is relatively simple. As a typical example, we have the following direct consequence of Theorem \ref{thm:delta adj via div}.

\begin{cor} \label{cor:adj delta ac}
Let $(X,\Delta),L_i,\V,Z,F,Z_0,\pi\colon Y\to X,\Delta_F$ and $\W$ be as in Theorem \ref{thm:delta adj via div}. Assume that
\begin{enumerate}
	\item $\W$ is almost complete $($Definition \ref{defn:almost complete}$)$,
	\item $\delta_{Z'}(F,\Delta_F+\lambda F(\W);c_1(M_{\vb}))\ge \lambda$ for some $0\le \lambda \le \frac{A_{X,\Delta}(F)}{S(\V;F)}$ and all subvarieties $Z'\subseteq Y$ with $\pi(Z')=Z_0$ $($where $M_{\vb}$ is the movable part of $\W)$.
\end{enumerate}
Then $\delta_Z(X,\Delta;\V)\ge \lambda$. If equality holds and $\delta_Z(\V)$ is computed by some valuation $v$ on $X$, then either $Z\subseteq C_X(F)$ and $F$ computes $\delta_Z(\V)$, or $C_Y(v)\not\subseteq F$ and for any irreducible component $S$ of $C_Y(v)\cap F$ with $Z_0\subseteq \pi(S)$, there exists some valuation $v_0$ on $F$ with center $Z$ computing
\[
\delta_{Z'}(F,\Delta_F+\lambda F(\W);c_1(M_{\vb}))=\lambda
\]
for all $Z'\subseteq S$ with $\pi(Z')=Z_0$.
\end{cor}

\begin{proof}
This is immediate from Theorem \ref{thm:delta adj via div} and Lemma \ref{lem:S of ac W}.
\end{proof}

\subsection{Filtrations from admissible flags}

One can inductively apply Theorem \ref{thm:delta adj via div} to refine the original graded linear series while lowering the dimension of the ambient variety. This is essentially equivalent to filtering the graded linear series via an admissible flag. For simplicity, consider the following situation. Let
\[
\Y \quad:\quad X=Y_0\supseteq Y_1\supseteq \cdots \supseteq Y_\ell
\]
be an admissible flag of length $\ell$ \emph{on} $X$. Assume that each $Y_i$ in the flag is a Cartier divisor on $Y_{i-1}$. Then for each $1\le j\le \ell$, we can define a boundary divisor $\Delta_j$ on $Y_j$ inductively as follows: first set $\Delta_0=\Delta$; for each $\Delta_i$ that's already defined, write $\Delta_i=a_i Y_{i+1}+\Gamma_i$ where $\Gamma_i$ doesn't contain $Y_{i+1}$ in its support and set $\Delta_{i+1}=\Gamma_i|_{Y_{i+1}}$. We also let $\Y^{(j)}$ be the flag given by $Y_0\supseteq \cdots\supseteq Y_j$ and let $\W^{(j)}$ be the refinement of $\V$ by $\Y^{(j)}$ (Example \ref{ex:refinement by flag}), i.e. it is the $\bN^{r+j}$-graded linear series on $Y_j$ given by
\[
W^{(j)}_{\va,b_1,\cdots,b_j}=V_{\va}(b_1,\cdots,b_j).
\]
Note that $\W^{(0)}=\V$. Also recall from Section \ref{sec:prelim-flag} that the flag $\Y$ induces a filtration $\cF=\cF_{\Y}$ on each $\W^{(j)}$.

\begin{thm} \label{thm:delta adj via flag}
With the above notation and assumptions, we have
\begin{equation} \label{eq:delta adj wrt filtration}
    \delta_Z(X,\Delta;\V)\ge \min\left\{ \min_{0\le i\le j-1}\left\{\frac{A_{Y_i,\Delta_i}(Y_{i+1})}{S(\W^{(i)};Y_{i+1})}\right\},\delta_{Z\cap Y_j}(Y_j,\Delta_j;\W^{(j)},\cF)\right\}
\end{equation}
for any $1\le j\le \ell$ and any subvariety $Z\subseteq X$ that intersects $Y_j$.
\end{thm}

This will be a key ingredient in our proof of Theorem \ref{main:index two}. Compared with Theorem \ref{thm:delta adj via div}, the main difference is that we allow (possibly) reducible centers $Z\cap Y_j$ when applying inversion of adjunction. In this case we only have an inequality $\delta_{Z\cap Y_j}(\W^{(j)},\cF)\ge \delta_{Z\cap Y_j}(\W^{(j)})$ (as opposed to the equality in Proposition \ref{prop:delta = delta wrt filtration}). As such, we also need to keep track of the filtration $\cF$ in the proof below.

\begin{proof}
By Proposition \ref{prop:delta = delta wrt filtration}, we have $\delta_Z(\V)=\delta_Z(\V,\cF)=\delta_{Z\cap Y_0}(Y_0,\Delta_0;\W^{(0)},\cF)$. Thus it suffices to prove that
\begin{equation} \label{eq:delta adj wrt filtration-induction step}
    \delta_{Z\cap Y_i}(Y_i,\Delta_i;\W^{(i)},\cF) \ge \min\left\{ \frac{A_{Y_i,\Delta_i}(Y_{i+1})}{S(\W^{(i)};Y_{i+1})}, \delta_{Z\cap Y_{i+1}}(Y_{i+1},\Delta_{i+1};\W^{(i+1)},\cF) \right\}
\end{equation}
for all $0\le i\le j-1$; \eqref{eq:delta adj wrt filtration} then follows by induction.

As in the proof of Theorem \ref{thm:delta adj via div}, let $D$ be an $m$-basis type $\bQ$-divisor of $\W^{(i)}$ that's compatible with $\cF$. Then in particular it is compatible with $Y_{i+1}$ and we may write
\[
D = S_m(\W^{(i)};Y_{i+1})\cdot Y_{i+1} + \Gamma
\]
where $\Gamma$ doesn't contain $Y_{i+1}$ in its support and $\Gamma|_{Y_{i+1}}$ is an $m$-basis type $\bQ$-divisor of $\W^{(i+1)}$ (since this is the same as the refinement of $\W^{(i)}$ by $Y_{i+1}$) that is compatible with $\cF$ (since the same is true for $D$ and the filtration $\cF$ on $\W^{(i)}$ is a refinement of the filtration induced by $Y_{i+1}$). Thus by inversion of adjunction as in the proof of Theorem \ref{thm:delta adj via div} we get
\[
\delta_{Z\cap Y_i,m}(Y_i,\Delta_i;\W^{(i)},\cF) \ge \min\left\{ \frac{A_{Y_i,\Delta_i}(Y_{i+1})}{S_m(\W^{(i)};Y_{i+1})}, \delta_{Z\cap Y_{i+1},m}(Y_{i+1},\Delta_{i+1};\W^{(i+1)},\cF) \right\}.
\]
Letting $m\to \infty$ we obtain \eqref{eq:delta adj wrt filtration-induction step} and this finishes the proof.
\end{proof}

\section{Applications} \label{sec:application}

\subsection{Tian's criterion and connection to birational superrigidity} \label{sec:Tian & bsr}

As a first application of the general framework developed in Section \ref{sec:adj}, we give a new proof of Tian's criterion for K-stability \cite{Tian-criterion} (see e.g. \cites{OS-alpha,FO-delta} for other proofs).

\begin{thm}[Tian's criterion] \label{thm:Tian's criterion}
Let $(X,\Delta)$ be a log Fano pair of dimension $n$. Assume that $(X,\Delta+\frac{n}{n+1}D)$ is log canonical $($resp. klt$)$ for any effective $\bQ$-divisor $D\sim_\bQ -(K_X+\Delta)$. Then $(X,\Delta)$ is K-semistable $($K-stable$)$.
\end{thm}

The proof is based on the following lemma, which is known to imply Tian's criterion (in fact this is the strategy used in \cite{FO-delta}). When $v$ is a divisorial valuation, the statement is proved in \cite{Fuj-plt-blowup}*{Proposition 2.1} and \cite{BJ-delta}*{Proposition 3.11}. Here we give a different proof using compatible divisors, which naturally generalizes the statement to all valuations (see also \cite{BJ-delta}*{Remark 3.12}).

\begin{lem} \label{lem:S<=n/(n+1)T}
Let $X$ be a projective variety of dimension $n$, let $L$ be an ample line bundle on $X$ and let $v$ be a valuation of linear growth on $X$. Then
\[
S(L;v)\le \frac{n}{n+1}T(L;v).
\]
\end{lem}

\begin{proof}
Let $r$ be a sufficiently large integer such that $rL$ is very ample and let $H\in |rL|$ be a general member. Let $\V$ be the complete linear series associated to $L$ and let $\cF$ be the filtration on $\V$ induced by $H$. By Proposition \ref{prop:delta = delta wrt filtration}, we have $S(L;v)=S(\V;v)=S(\V,\cF;v)$. Let $D$ be an $m$-basis type $\bQ$-divisor of $L$ that's compatible with $\cF$. By the same discussion as at the beginning of Section \ref{sec:filter by div}, we have
\[
D=S_m(L;H)\cdot H + \Gamma
\]
for some effective $\bQ$-divisor $\Gamma$ whose support doesn't contain $H$. Since $H$ is general, we have $C_X(v)\not\subseteq H$. Thus $v(D)=v(\Gamma)\le T(L-S_m(L;H)\cdot H; v))$ and $S_m(\V,\cF;v)\le T(L-S_m(L;H)\cdot H; v))$. Letting $m\to \infty$ we see that 
\[
S(L;v)\le T(L-S(L;H)\cdot H; v)).
\]
By direct calculation for any irreducible divisor $H\in |rL|$ we have
\begin{equation} \label{eq:S(L;H)=1/(n+1)}
    S(L;H)=\int_0^{1/r}(1-rx)^n \rd x=\frac{1}{r(n+1)};
\end{equation}
putting it into the previous inequality we get $S(L;v)\le \frac{n}{n+1}T(L;v)$ as desired.
\end{proof}

\begin{proof}[Proof of Theorem \ref{thm:Tian's criterion}]
We only prove the K-stability part since the argument for K-semistability is similar (and simpler).
Let $r>0$ be an integer such that $-r(K_X+\Delta)$ is Cartier.  Following \cite{Fuj-valuative-criterion}, we say that a divisor $E$ over $X$ is dreamy if the double graded algebra $\bigoplus_{k,j\in\bN} H^0(Y,-kr\pi^*(K_X+\Delta)-jE)$ is finitely generated (where $\pi\colon Y\to X$ is a proper birational morphism such that the center of $E$ is a prime divisor on $Y$). For such $E$, there exists some effective $\bQ$-divisor $D\sim_\bQ -(K_X+\Delta)$ such that $T(-K_X-\Delta;E)=\ord_E(D)$. By assumption, $(X,\Delta+\frac{n}{n+1}D)$ is klt, hence $\frac{n}{n+1}T(-K_X-\Delta;E)<A_{X,\Delta}(E)$ and by Lemma \ref{lem:S<=n/(n+1)T} we have $\beta_{X,\Delta}(E)=A_{X,\Delta}(E)-S(-K_X-\Delta;E)>0$. Since this holds for any dreamy divisor $E$ over $X$, $(X,\Delta)$ is K-stable by \cite{Fuj-valuative-criterion}*{Theorem 1.6 and \S 6}.
\end{proof}

Using the same strategy, we can also give a new proof of the following statement, which implies the K-stability criterion from \cite{SZ-bsr-imply-K}.

\begin{thm}[\cite{Z-cpi}*{Theorem 1.5}] \label{thm:Tian's criterion rho=1}
Let $(X,\Delta)$ be a log Fano pair where $X$ is $\bQ$-factorial of Picard number $1$ and dimension $n$. Assume that for every effective $\bQ$-divisor $D\sim_\bQ -(K_X+\Delta)$ and every movable boundary $M\sim_\bQ -(K_X+\Delta)$, the pair $(X,\Delta+\frac{1}{n+1}D+\frac{n-1}{n+1}M)$ is log canonical $($resp. klt$)$. Then $X$ is K-semistable $($resp. K-stable$)$.
\end{thm}

For the proof we need some notation. Let $X$ be a projective variety of dimension $n$ and let $v$ be a valuation on $X$ whose center has codimension at least two on $X$. Let $L$ be an ample line bundle on $X$. We define the movable threshold $\eta(L;v)$ (see \cite{Z-cpi}*{Definition 4.1}) as the supremum of all $\eta>0$ such that the base locus of the linear system $|\cF_v^{m\eta} H^0(X,mL)|$ has codimension at least $2$ for some $m\in\bN$. Analogous to Lemma \ref{lem:S<=n/(n+1)T} we have
% every divisor in the stable base locus of $\pi^*L-\eta E$ (where $\pi\colon Y\to X$ is a proper birational morphism on which the center of $E$ is a prime divisor) is exceptional over $X$. 

\begin{lem}[\cite{Z-cpi}*{Lemma 4.2}] \label{lem:S<=T+eta}
Notation as above. Assume that $X$ is $\bQ$-factorial and $\rho(X)=1$. Then we have
\[
S(L;v)\le \frac{1}{n+1}T(L;v)+\frac{n-1}{n+1}\eta(L;v).
\]
\end{lem}

\begin{proof}
We may assume that $T(L;v)>\eta(L;v)$, otherwise the statement follows from Lemma \ref{lem:S<=n/(n+1)T}. We claim that there exists a unique \emph{irreducible} $\bQ$-divisor $G\sim_\bQ L$ such that $v(G)>\eta$. The uniqueness simply follows from the definition of movable threshold. To see the existence, let $\widetilde{G}\sim_\bQ L$ be an effective $\bQ$-divisor on $X$ such that $v(\tG)>\eta$ (such $\tG$ exists by the definition of pseudo-effective thresholds). Since $X$ is $\bQ$-factorial and has Picard number one, every divisor on $X$ is $\bQ$-linearly equivalent to a rational multiple of $L$. In particular, we may write $\tG=\sum \lambda_i G_i$ where $\sum \lambda_i = 1$ and each $G_i\sim_\bQ L$ is irreducible. As $v(\tG)>\eta$, we have $v(G_i)>\eta$ for some $i$ which proves the claim. Note that by the definition of pseudo-effective threshold we then necessarily have $v(G)=T(L;v)$. Write $G=\lambda G_0$ where $G_0$ is a prime divisor on $X$.

As in the proof of Lemma \ref{lem:S<=n/(n+1)T}, let $r\in \bZ_+$ be such that $rL$ is very ample and let $H\in |rL|$ be a general member. Let $\V$ be the complete linear series associated to $L$ and let $\cF$ be the filtration on $\V$ induced by $H$. By Proposition \ref{prop:delta = delta wrt filtration}, we have $S(L;v)=S(\V,\cF;v)$. Let $D$ be an $m$-basis type $\bQ$-divisor of $L$ that's compatible with both $\cF$ and $\cF_v$ (which exists by Lemma \ref{lem:basis compatible with two filtrations}). We have
\[
D=S_m(L;H)\cdot H + \Gamma
\]
for some effective $\bQ$-divisor $\Gamma$ whose support doesn't contain $H$. We further decompose $\Gamma=\mu G_0 + \Gamma_0$ where the support of $\Gamma_0$ doesn't contain $G_0$. Note that $v(\Gamma_0)\le \eta(\Gamma_0;v)$ by our choice of $G_0$. As $H$ is general and $D$ is of $m$-basis type we have $\mu=\ord_{G_0}(\Gamma)=\ord_{G_0}(D)\le S_m(L;G_0)$, thus
\begin{align*}
    S_m(\V,\cF;v) & = v(D) = v(\Gamma) = \mu\cdot v(G_0) + v(\Gamma_0) \\
    & \le S_m(L;G_0)\cdot v(G_0) + \eta(\Gamma - S_m(L;G_0)\cdot G_0;v) \\
    & = T(S_m(L;G_0)\cdot G_0;v) + \eta(L- S_m(L;H)\cdot H - S_m(L;G_0)\cdot G_0;v).
\end{align*}
Since $\rho(X)=1$, for any prime divisor $F$ on $X$ we have $S(L;F)\cdot F \sim_\bQ \frac{1}{n+1}L$ as in the proof of Lemma \ref{lem:S<=n/(n+1)T}, hence letting $m\to \infty$ in the above inequality we obtain
\[
S(L;v)=\lim_{m\to\infty} S_m(\V,\cF;v)\le \frac{1}{n+1}T(L;v)+\frac{n-1}{n+1}\eta(L;v)
\]
as desired.
\end{proof}

\begin{proof}[Proof of Theorem \ref{thm:Tian's criterion rho=1}]
As in Theorem \ref{thm:Tian's criterion} we only prove the K-stability part. Let $E$ be a dreamy divisor over $X$. If the center of $E$ is a prime divisor on $X$, then we have $-(K_X+\Delta)\sim_\bQ \lambda E$ for some $\lambda>0$ as $X$ has Picard number one. By assumption $(X,\Delta+\frac{\lambda}{n+1} E)$ is klt, hence $\beta_{X,\Delta}(E) = A_{X,\Delta}(E)-S(-K_X-\Delta;E) = A_{X,\Delta}(E) - \frac{\lambda}{n+1} > 0$. If the center of $E$ has codimension at least two on $X$, then since $E$ is dreamy there are effective $\bQ$-divisor $D\sim_\bQ -(K_X+\Delta)$ and movable boundary $M\sim_\bQ -(K_X+\Delta)$ such that $\ord_E(D)=T(-K_X-\Delta;E)$ and $\ord_E(M)=\eta(-K_X-\Delta;E)$. By assumption $(X,\Delta+\frac{1}{n+1}D+\frac{n-1}{n+1}M)$ is klt, thus
\begin{align*}
    A_{X,\Delta}(E) & > \frac{1}{n+1}\ord_E(D)+\frac{n-1}{n+1}\ord_E(M) \\
    &  = \frac{1}{n+1}T(-K_X-\Delta;E)+\frac{n-1}{n+1}\eta(-K_X-\Delta;E) \\
    & \ge S(-K_X-\Delta;E)
\end{align*}
where the last inequality follows from Lemma \ref{lem:S<=T+eta}. Therefore $\beta_{X,\Delta}(E)>0$ for all dreamy divisors $E$ over $X$ and $(X,\Delta)$ is K-stable by \cite{Fuj-valuative-criterion}*{Theorem 1.6 and \S 6}.
\end{proof}

\begin{cor}[\cite{SZ-bsr-imply-K}*{Theorem 1.2}]
Let $X$ be a birationally superrigid Fano variety. Assume that $(X,\frac{1}{2}D)$ is lc for all effective $\bQ$-divisor $D\sim_\bQ -K_X$. Then $X$ is K-stable.
\end{cor}

\begin{proof}
By \cite{CS-Noether-Fano}*{Theorem 1.26}, $X$ is $\bQ$-factorial of Picard number one and $(X,M)$ has canonical singularities (in particular it is klt) for every movable boundary $M\sim_\bQ -K_X$. Let $D\sim_\bQ -K_X$ be an effective $\bQ$-divisor. By assumption, $(X,\frac{1}{2}D)$ is lc. As $\frac{1}{n+1}D+\frac{n-1}{n+1}M=\frac{2}{n+1}\cdot \frac{1}{2}D+\frac{n-1}{n+1}M$ is a convex linear combination of $\frac{1}{2}D$ and $M$, we see that the conditions of Theorem \ref{thm:Tian's criterion rho=1} are satisfied and therefore $X$ is K-stable.
\end{proof}

\subsection{Fano manifolds of small degrees} \label{sec:small deg}

As a second application of our general framework, we study K-stability of Fano manifolds of small degree using flags of complete intersection subvarieties. To do so we first specialize Corollary \ref{cor:adj delta ac} to the case when the auxiliary divisor is an ample Cartier divisor on the given variety.

\begin{lem} \label{lem:F=ample Cartier}
Let $X$ be a variety of dimension $n$, let $L$ be an ample line bundle on $X$ and let $H\in |L|$. Assume that $H$ is irreducible and reduced. Then
\[
\delta_x(L)\ge \min\left\{n+1,\frac{n+1}{n}\delta_x(L|_H)\right\}
\]
at every $x\in H$. If equality holds, then either $\delta_x(L)=n+1$ and it is computed by $H$, or $\delta_x(L)=\frac{n+1}{n}\delta_x(L|_H)$ and $C_X(v)\not\subseteq H$ for any valuation $v$ that computes $\delta_x(L)$. Moreover, in the latter case, for every irreducible component $Z$ of $C_X(v)\cap H$ containing $x$, there exists a valuation $v_0$ on $H$ with center $Z$ computing $\delta_x(L|_H)$.
\end{lem}

\begin{proof}
Let $\V$ be the complete linear series associated to $L$ and let $\W$ be its refinement by $H$. By Example \ref{ex:graded series from ample div}, $\W$ is almost complete and $F(\W)=0$. By \eqref{eq:S(L;H)=1/(n+1)}, we have $S(L;H)=\frac{1}{n+1}$. As discussed in \S \ref{sec:adj}, any $m$-basis type $\bQ$-divisor $D\sim_\bQ L$ that is compatible with $H$ can be written as $D=S_m(L;H)\cdot H + \Gamma$ where $\Gamma|_H$ is an $m$-basis type $\bQ$-divisor of $\W$, thus letting $m\to \infty$ we see that
\[
c_1(\W)\sim_\bQ L|_H- S(L;H)\cdot H|_H\sim_\bQ \frac{n}{n+1}L|_H
\]
and $\delta_x(c_1(\W))=\frac{n+1}{n}\delta_x(L|_H)$. The result now follows directly from Corollary \ref{cor:adj delta ac} with $F=H$.
\end{proof}

Applying induction, we further deduce:

\begin{lem} \label{lem:delta small deg}
Let $X$ be a variety of dimension $n$ and let $L$ be an ample line bundle on $X$. Let $x\in X$ be a smooth point. Assume that 
\begin{enumerate}
    \item[(*)] there exists $H_1,\cdots,H_{n-1}\in |L|$ containing $x$ such that $H_1\cap \cdots \cap H_{n-1}$ is an integral curve that is smooth at $x$.
\end{enumerate}
Then $\delta_x(L)\ge \frac{n+1}{(L^n)}$. If equality holds, then either $(L^n)=1$, or every valuation that computes $\delta_x(L)$ is divisorial and is induced by some prime divisor $E$ \emph{on} $X$.
\end{lem}

\begin{proof}
First assume that $n=1$, i.e. $X$ is a curve that is smooth at $x$ (in this case the statement should be well-known to experts). By direct calculation we have $S(L;x)=\frac{1}{2}\deg L$. Hence $\delta_x(L) = \frac{2}{\deg L}$ as desired.

Assume now that the statement has been proved in smaller dimensions. Let $H\in |L|$ be a general divisor containing $x$. By (*), $H$ is smooth at $x$ and $L|_H$ also satisfies (*). By induction hypothesis, we have $\delta_x(L|_H)\ge \frac{n}{(L^{n-1}\cdot H)}=\frac{n}{(L^n)}$, hence by Lemma \ref{lem:F=ample Cartier} we see that $\delta_x(L)\ge \frac{n+1}{(L^n)}$. Suppose that equality holds, $(L^n)>1$ and let $v$ be a valuation on $X$ that computes $\delta_x(L)$. Then by Lemma \ref{lem:F=ample Cartier}, we see that the center $C_X(v)$ of $v$ is not contained in $H$, $\delta_x(L|_H)=\frac{n}{(L^n)}$ and it is computed by some valuation $v_0$ on $H$ with center $Z\subseteq C_X(v)\cap H$. But by induction hypothesis, $v_0$ is divisorial and its center $Z$ is a prime divisor on $H$, hence $C_X(v)$ has to be a divisor on $X$. It follows that $v$ is divisorial as well and is induced by a divisor \emph{on} $X$.
\end{proof}

We now restrict our attention to Fano manifolds of small degree:

\begin{cor} \label{cor:fano small deg}
Let $X$ be a Fano manifold of dimension $n$. Assume that there exists an ample line bundle $L$ on $X$ such that
\begin{enumerate}
    \item $-K_X\sim_\bQ rL$ for some $r\in\bQ$ with $(L^n)\le \frac{n+1}{r}$; and
    \item for every $x\in X$, there exists $H_1,\cdots,H_{n-1}\in |L|$ containing $x$ such that $H_1\cap \cdots \cap H_{n-1}$ is an integral curve that is smooth at $x$.
\end{enumerate}
Then $X$ is K-semistable. If it is not uniformly K-stable, then $(L^n) = \frac{n+1}{r}$ and there exists some prime divisor $E\subseteq X$ such that $\beta_X(E)=0$.
% Let $X$ be a Fano manifold of dimension $n$. Assume that there exists an ample line bundle $L$ on $X$ and some $r\in \bQ$ such that
% \begin{enumerate}
%     \item $-K_X\sim_\bQ rL$;
%     \item $L$ satisfies condition $(*)$ of Lemma \ref{lem:delta small deg} at every $x\in X$; and
%     \item $(L^n)\le \frac{n+1}{r}$.
% \end{enumerate}
% Then $X$ is K-semistable. If it is not K-stable, then equality holds in $(3)$ and there exists some divisor $E$ \emph{on} $X$ such that $\beta_X(E)=0$.
\end{cor}

\begin{proof}
By Lemma \ref{lem:delta small deg}, we have $\delta_x(L)\ge \frac{n+1}{(L^n)}$ at every $x\in X$, hence $\delta(L)\ge \frac{n+1}{(L^n)}$. By (1) we then obtain $\delta(-K_X)\ge \frac{n+1}{r\cdot (L^n)}\ge 1$ and $X$ is K-semistable. Assume that $X$ is not uniformly K-stable, i.e., $\delta(-K_X)=1$. Then equality holds in (1) and $\delta(L) = \frac{n+1}{(L^n)}$. By Lemma \ref{lem:delta small deg}, either $(L^n)=1$ or $\delta(L)$ is computed by some prime divisor $E$ on $X$. In the latter case, there is nothing left to prove. In the former case, we have $r=\frac{n+1}{(L^n)}=n+1$, hence $X\cong \bP^n$ by \cite{KO-P^n} and $\beta_X(H)=0$ for any hyperplane $H$ on $X$.
\end{proof}

In particular, taking $L$ to be the hyperplane class on $\bP^n$, Corollary \ref{cor:fano small deg} gives a new algebraic proof of the K-semistability of $\bP^n$ (see e.g. \cites{Li-equivariant-minimize,PW-dP} for other proofs). It also gives a unified treatment of the uniform K-stability of the following Fano manifolds.

\begin{cor} \label{cor:small deg list}
The following Fano manifolds are all uniformly K-stable:
\begin{enumerate}
    \item \cite{Tian-dP} del Pezzo surfaces of degree $\le 3$;
    \item \cite{Fuj-alpha} hypersurfaces $X\subseteq \bP^{n+1}$ of degree $n+1$;
    \item \cite{Der-finite-cover} double covers of $\bP^n$ branched over some smooth divisor $D$ of degree $d\ge n+1$.
    \item cyclic covers $\pi:X\to Y$ of degree $s$ $($where $Y\subseteq \bP^{n+1}$ is a smooth hypersurface of degree $m)$ branched along some smooth divisor $D\in |dH|$ $($where $H$ is the hyperplane class$)$ with $0\le n+2-m-(1-\frac{1}{s})d\le \frac{n+1}{ms}$.
    \item del Pezzo threefolds of degree $1$, i.e. smooth weighted hypersurfaces $X_6\subseteq \bP(1^3,2,3)$.
\end{enumerate}
\end{cor}

\begin{proof}
Note that (3) is a special case of (4) with $m=1$. We will also treat (5) separately. In each remaining case, we will find an ample line bundle $L$ on the Fano variety that satisfies the assumptions of Corollary \ref{cor:fano small deg}. Indeed, for del Pezzo surfaces $X$ of degree $2$ or $3$ (resp. degree $1$), we take $L=-K_X$ (resp. $L=-2K_X$). We also set $L=-K_X$ for hypersurfaces $X\subseteq \bP^{n+1}$ of degree $n+1$. In case (4), we choose $L=\pi^*H$. It is straightforward to verify that they all satisfy the assumptions of Corollary \ref{cor:fano small deg}, hence by Corollary \ref{cor:fano small deg}, all Fano manifolds $X$ in (1)--(4) are K-semistable. Moreover, del Pezzo surface of degree $1$ or $2$ are uniformly K-stable since $(L^n)<\frac{n+1}{r}$ for our choice of $L$. It remains to check that there are no divisors $E$ on $X$ with $\beta_X(E)=0$ in the other cases.

Let $\tau=T(-K_X;E)$ be the pseudo-effective threshold (Definition \ref{defn:S and T}). If $\dim X\ge 3$ (so we are in case (2) or (4)), then $X$ has Picard number one and $-K_X\sim_\bQ \tau E$. A direct calculation gives $\beta_X(E)=1-\frac{\tau}{n+1}$. Since $X$ is not isomorphic to $\bP^n$, we have $\tau<n+1$ by \cite{KO-P^n} and thus $\beta_X(E)>0$ in this case (c.f. \cite{Fuj-div-stable}*{Corollary 9.3}). If $\dim X=2$, then $X$ is a cubic surface. Clearly $S(E)<\tau$. Since $-K_X-\tau E$ is pseudo-effective, it has non-negative intersection with $-K_X$ and thus $\tau\le \frac{3}{(-K_X\cdot E)}$. It follows that if $\beta_X(E)=1-S(E)=0$, then $\tau > S(E) = 1$ and $(-K_X\cdot E)\le 2$, i.e. $E$ is a line or a conic. But in both cases we have $\tau=1$ and hence $S(E)<1$: if $E$ is a line, then $|-K_X-E|$ is base point free and defines a conic bundle $X\to \bP^1$; if $E$ is a conic and $L_0$ is the residue line (the other component of the hyperplane section that contains $E$), then $-K_X-E\sim L_0$ is a $(-1)$-curve. Thus $\beta_X(E)=1-S(E)>0$ for all divisors $E$ on the cubic surface $X$ as well. We therefore conclude that all Fano manifolds in (1)--(4) are uniformly K-stable.

It remains to prove every Fano threefold $X$ in (5) is uniformly K-stable. For such $X$, we have $-K_X=2H$ for some ample line bundle $H$ on $X$. We claim that for every $x\in X$ there exists a smooth member $S\in |H|$ that contains $x$. Indeed, it is not hard to check that $h^0(X,H\otimes \fm_x) \ge 2$. Let $S_1\neq S_2 \in |H\otimes \fm_x|$ and let $\cM\subseteq |H\otimes \fm_x|$ be the pencil they span. As $H$ generates $\Pic(X)$, $S_1$ and $S_2$ doesn't have common component and we have a well-defined $1$-cycle $W=(S_1\cdot S_2)$ on $X$. Since $(H\cdot W)=(H^3)=1$, $W$ is an integral curve. As $W$ is also the complete intersection of any two members of $\cM$, every $S\in \cM$ is smooth at the smooth points of $W$. Let $y$ be a singular point of $W$ and let $S'$ be a general member of $|2H\otimes \fm_y|$. Then as $|2H|$ is base point free, $S_1\cap S_2\cap S$ is zero dimensional. If both $S_1$ and $S_2$ are singular at $y$, then we have $\mult_y S_i\ge 2$ ($i=1,2$) and thus
\[
2=2(H^3)=(S_1\cdot S_2\cdot S')\ge \mult_y S_1\cdot \mult_y S_2 \cdot \mult_y S' \ge 4,
\]
a contradiction. Hence a general member of $\cM$ is smooth at $y$. Since there are only finitely many singular point of $W$ and $\cM$ is base point free outside $W$, we see that a general member of $\cM$ is smooth, proving the claim.

Now let $x\in X$ and choose a smooth member $S\in |H|$ containing $x$. Note that $S$ is a del Pezzo surface of degree $1$. By Lemma \ref{lem:delta small deg} with $L=-2K_S$, we have $\delta(H|_S)=2\delta(L)\ge \frac{3}{2}$ and if the equality is computed by some divisor $E$ over $S$, then $E$ is a divisor \emph{on} $S$. By Lemma \ref{lem:F=ample Cartier}, it follows that $\delta_x(-K_X)=\frac{1}{2}\delta_x(H)\ge \frac{2}{3}\delta(H|_S)\ge 1$ for all $x\in X$, thus $X$ is K-semistable. If it is not K-stable, then by another use of Lemma \ref{lem:F=ample Cartier} and the same argument as in Corollary \ref{cor:fano small deg}, we have $\beta_X(E)=0$ for some divisor $E$ \emph{on} $X$. But since $X$ has Picard number one and is not $\bP^3$, this is a contradiction as before and therefore $X$ is uniformly K-stable.
\end{proof}

\subsection{Surface case} \label{sec:surface}

We next investigate the surface case where almost everything can be explicitly computed. Recall from \cite{Fuj-plt-blowup}*{Theorem 1.5} that it is enough test K-stability of log Fano pairs using divisors of plt type. The nice feature in the surface case is that the corresponding refinements are always almost complete.

% First recall from \cite{BLX-openness}*{Theorem 1.5} that if a log Fano pair $(X,\Delta)$ is not K-stable, then its stability threshold is computed by a quasi-monomial valuation $v$ that is an lc place of a $\bQ$-complement (i.e. there exists some effective $\bQ$-divisor $\Gamma$ such that $(X,\Delta+\Gamma)$ is lc, $K_X+\Delta+\Gamma\sim_\bQ 0$ and $A_{X,\Delta+\Gamma}(v)=0$). This suggests that in order to verify the K-stability of a log Fano pair, we should consider the refinements of its anti-log-canonical linear series by lc places of complements. 

\begin{lem} \label{lem:surface almost complete}
Let $(S,\Delta)$ be a surface pair and let $L$ be an big line bundle on $S$. Let $E$ be a plt type divisor over $S$. Let $\V$ be the complete linear series associated to $L$ and let $\W$ be the refinement of $\V$ by $E$. Then $\W$ is almost complete.
\end{lem}

% \begin{lem} \label{lem:surface almost complete}
% Let $(S,\Delta)$ be a log del Pezzo surface pair and let $E$ be a plt type divisor over $X$ that is also an lc place of a $\bQ$-complement of $(S,\Delta)$. Let $r>0$ be an integer such that $-r(K_S+\Delta)$ is Cartier, let $\V$ be the complete linear series associated to $-r(K_S+\Delta)$ and let $\W$ be the refinement of $\V$ by $E$. Then $\W$ is almost complete.
% \end{lem}

As in Example \ref{ex:graded series from ample div}, the almost completeness of a refinement is related to the surjectivity of the natural restriction map on sections, hence the proof of Lemma \ref{lem:surface almost complete} essentially boils down to the following vanishing-type result.

\begin{lem} \label{lem:surface H^1}
Let $(S,\Delta)$ be a surface pair. Then there exists some constant $A>0$ such that $h^1(S,\cO_S(D))\le A$ for all $\bQ$-Cartier Weil divisor $D$ on $S$ such that $D-(K_S+\Delta)$ is nef and big.
\end{lem}

\begin{proof}
Let $f\colon T\to S$ be the minimal log resolution of $(S,\Delta)$ and let $(T,\Delta_T)$ be the crepant pullback of $(S,\Delta)$, i.e. $K_T+\Delta_T = f^*(K_S+\Delta)$. Let $E$ be the sum of all exceptional divisors. Since $D$ has integer coefficients, $\{f^*D\}$ is exceptional over $S$, hence we have $\lfloor f^*D \rfloor + E \ge f^*D$ and $f_*\cO_T(\lfloor f^*D \rfloor + E)=\cO_S(D)$. Let $L=\lfloor f^*D \rfloor + E$ and let $\Delta'=\Delta_T+L-f^*D$. Then it is easy to check that $0\le \Delta'\le \Delta_T+E$ and 
\[
L - (K_T+\Delta') \sim_\bQ f^*(D-K_S-\Delta),
\]
which is nef and big by assumption. By Lemma \ref{lem:snc surface H^1}, we know that there exists some constant $A$ depending only on the pair $(T,\Delta_T+E)$ such that $h^1(T,\cO_T(L))\le A$. The lemma then follows as $h^1(S,\cO_S(D))=h^1(S,\pi_*\cO_T(L))\le h^1(T,\cO_T(L))$.
\end{proof}

The following result is used in the above proof.

\begin{lem} \label{lem:snc surface H^1}
Let $S$ a smooth surface and let $\Delta$ be an effective divisor on $S$ with simple normal crossing support. Then there exists some constant $A$ such that $h^1(T,\cO_T(L))\le A$ for all Cartier divisor $L$ such that $L-(K_T+\Delta')$ is nef and big for some $\bQ$-divisor $0\le \Delta'\le \Delta$.  
\end{lem}

\begin{proof}
We prove by induction on the sum of all coefficients of $\lfloor \Delta \rfloor$. First note that if $\lfloor \Delta' \rfloor=0$ then $(T,\Delta')$ is klt and $h^1(T,\cO_T(L))=0$ by Kawamata-Viehweg vanishing. Thus it suffices to consider the case when $\lfloor \Delta' \rfloor\neq 0$. In particular, we may just take $A=0$ when $\lfloor \Delta \rfloor=0$. In general, let $C$ be an irreducible component of $\lfloor \Delta' \rfloor\le \lfloor \Delta \rfloor$. By assumption $(L-K_T-\Delta')\cdot C\ge 0$ which gives $\deg (L|_C-K_C)\ge (\Delta'-C)\cdot C$, thus by Serre duality $h^1(C,\cO_C(L))=h^0(C,\omega_C(-L))\le 1+\deg (K_C-L|_C)\le 1+\left((C-\Delta')\cdot C\right)$ is bounded by some constants $A_1$ that only depends on $\Delta$. By induction hypothesis (applied to the pairs $(T,\Delta-C)$ for various components $C$ of $\lfloor \Delta \rfloor$), we also have $h^1(T,\cO_T(L-C))\le A_2$ for some constant $A_2$ that only depends on $\Delta$, thus $h^1(T,\cO_T(L))\le A_1+A_2$ via the exact sequence $0\to \cO_T(L-C)\to \cO_T(L)\to \cO_C(L)\to 0$.
\end{proof}

\begin{proof}[Proof of Lemma \ref{lem:surface almost complete}]
Let $T_1\to S$ be the prime blow up associated to $E$. Note that $E$ is a smooth curve on $T_1$ and $T_1$ is klt along $E$ (as $(T_1,\Delta_{T_1}+E)$ is plt by assumption). Let $T\to T_1$ be the minimal resolution of $T_1$ over its non-klt locus and let $\pi\colon T\to S$ be the induced morphism. Note that $T$ is $\bQ$-factorial. Let $I = \Supp(\W)\cap (\{1\}\times \bR)$, let $\gamma\in I^\circ \cap \bQ$ and let $\pi^*L-\gamma E = P_\gamma + N_\gamma$ be the Zariski decomposition where $P_\gamma$ (resp. $N_\gamma$) is the nef (resp. negative) part. We claim that there exists a divisor $G\subseteq T$ such that $\Supp(N_\gamma)\subseteq G$ for all $\gamma$. Indeed, for any $\gamma_1<\gamma<\gamma_2$, since $\pi^*L-\gamma E$ is a convex linear combination of $\pi^*L-\gamma_1 E$ and $\pi^*L-\gamma_2 E$, we see that $\Supp(N_\gamma)\subseteq \Supp(N_{\gamma_1})\cup \Supp(N_{\gamma_2})$. On the other hand, by \cite{Nak-Zariski-decomp}*{Proposition III.1.10}, there are at most $\rho(T)$ irreducible components in each $N_\gamma$. It follows that $\cup_{\gamma} \Supp(N_{\gamma})$ is a finite union of divisors in $T$ and we may simply take $G=\cup \Supp(N_\gamma)$. Note that $E\not\subseteq \Supp(G)$ as otherwise $E\subseteq \Bs(\pi^*L-\gamma E)$ for some $\gamma$ and thus $W_{m,m\gamma}=0$ for all $m$.

Now fix $\gamma\in I^\circ \cap \bQ$ and write $P$ (resp. $N$) for $P_\gamma$ (resp. $N_\gamma$). Then for sufficiently divisible $m$, we have $|m(\pi^*L-\gamma E)|=|mP|+mN$. It follows that $|M_{m,m\gamma}|\subseteq |D_m|$ for some divisor $D$ with $0<\deg D_m\le m(P\cdot E)$ and in particular $(P\cdot E)>0$. Since $E$ is a curve, any divisors on $E$ are numerically proportional. Thus $\W$ is almost complete (with respect to any line bundle of degree $1$ on $E$) as long as
\begin{equation} \label{eq:section portion->1}
    \lim_{m\to \infty} \frac{h^0(W_{m,m\gamma})}{m(P\cdot E)} = 1,
\end{equation}
% We always have $|M_{m,j}|\subseteq |\cO_E(L_{m,j})|$ for some divisor $L_{m,j}$. Thus $\W$ is almost complete as long as
% \begin{equation} \label{eq:section portion->1}
%     \lim_{m\to \infty} \frac{h^0(W_{m,m\gamma})}{\deg L_{m,m\gamma}} = 1
% \end{equation}
% for all $\gamma\in (0,T(L;E))\cap \bQ$, 
where the limit is taken over sufficiently divisible integers $m$. Indeed, if \eqref{eq:section portion->1} holds, then as 
\[
h^0(W_{m,m\gamma})=h^0(M_{m,m\gamma})\le h^0(E,D_m)\le\deg D_m +1\le m(P\cdot E)+1
\]
we clearly have $\lim_{m\to \infty}\frac{h^0(W_{m,m\gamma})}{h^0(E,D_m)}=1$, which verifies condition (2) in Definition \ref{defn:almost complete}. It also gives $\lim_{m\to \infty} \frac{\deg D_m}{m}=(P\cdot E)$, hence $\lim_{m\to\infty} \frac{F_{m,m\gamma}}{m}=N|_E$ for sufficiently divisible $m$ (c.f. the proof of Lemma \ref{lem:F(W) formula on surface} below). Since $\Supp(N)\subseteq G$, we see that $F(\W)$ is supported on $G\cap E$, which verifies condition (1) in Definition \ref{defn:almost complete}.

It remains to prove \eqref{eq:section portion->1}. To see this, we note that $P$ is big (since $\gamma\in I^\circ$) and hence $m_0 P - E-K_T$ is effective for some divisible enough integer $m_0$. Let $Q\in |m_0 P - E-K_T|$. Then by Lemma \ref{lem:surface H^1}, there exists some constant $A$ depending only on $(T,Q)$ such that $h^1(T,\cO_T(mP-E))\le A$ for all sufficiently divisible $m>m_0$ (as $mP-E-(K_T+Q)\sim (m-m_0)P$ is nef and big). 
% and let $L=-r(K_S+\Delta)$. By assumption, $E$ is an lc center of some $\bQ$-complement $\Gamma$ of $(S,\Delta)$. In particular, we have 
% \[
% K_T+D_T+E\sim_\bQ \pi^*(K_S+\Delta+\Gamma)\sim_\bQ 0
% \]
% for some effective $\bQ$-divisor $D_T$ on $T$ and $(T,D_T+E)$ is lc. Note that $E$ is a smooth curve on $T$ (as $(T,\Delta_T+E)$ is plt by assumption) and $\Supp(\W)\cap (\{1\}\times \bR)=[0,T(L;E)]$.

% Suppose first that $E\cong \bP^1$. Then we always have $|M_{m,j}|\subseteq |\cO_{\bP^1}(\ell)|$ for some $\ell\in\bN$. Thus $\W$ is almost complete as long as (in the notation of Definition \ref{defn:almost complete})
% \begin{equation} \label{eq:section portion->1}
%     \lim_{m\to \infty} \frac{h^0(W_{m,m\gamma})}{\ell_{m,\gamma}} = 1
% \end{equation}
% for all $\gamma\in (0,T(L;E))\cap \bQ$, where the limit is taken over sufficiently divisible $m$. To see this, we fix $\gamma$, let $H$ be a very ample divisor on $T$ and choose sufficiently divisible $m\gg 0$ such that $m(L-\gamma E)-H$ is effective (this is possible since $L-\gamma E$ is big). Let $|m(L-\gamma E)|=|M|+F$ be the decomposition into movable and fixed part. Then $M$ is nef (it has at most isolated base point) and big (as $M-H$ is effective). As $(T,D_T)$ is lc and $M-E-(K_T+D_T)\sim_\bQ M$, we see from Lemma \ref{lem:surface H^1} that there exists a constant $A>0$ (independent on $M$) such that $h^1(T,\cO_T(M-E))\le A$. 
Using the exact sequence
\[
0\to \cO_T(mP-E)\to \cO_T(mP)\to \cO_E(mP)\to 0
\]
from Lemma \ref{lem:D|_F along plt boundary}, we obtain
\begin{align*}
    h^0(W_{m,m\gamma}) & =\dim {\rm Im}(H^0(T,\cO_T(mP))\to H^0(E,\cO_E(mP))) \\
    & \ge h^0(E,\cO_E(mP))-h^1(T,\cO_T(mP-E)) \\
    & \ge h^0(E,\cO_E(mP))-A \\
    & \ge m(P\cdot E) +1 -g(E) -A
\end{align*}
where the last inequality follows from Riemann-Roch.
% and the fact that $M_{m,m\gamma}$ is a sub-linear series of $|M|_E|$ and thus $\deg L_{m,m\gamma}\le \deg (\lfloor M|_E\rfloor)$. 
Letting $m\to \infty$ we get \eqref{eq:section portion->1} and hence $\W$ is almost complete as desired.
% If $E$ is not a rational curve, then $E$ is not exceptional over $S$; i.e. $E$ is a curve on $S$ and $T=S$. Write $\Delta=aE+\Delta_1$, $\Gamma=(1-a)E+\Gamma_1$ where $E\not\subseteq \Supp(\Delta_1+\Gamma_1)$ and let $D=\Delta_1+\Gamma_1$. By adjunction we get $K_E+{\rm Diff}_E(D)\sim_\bQ (K_S+E+D)|_E\sim_\bQ 0$, thus $E$ is an elliptic curve and ${\rm Diff}_E(D)=0$. In other words, $E$ lies in the smooth locus of $X$ and $D$ is disjoint from $E$ (see e.g. \cite{Kol-mmp}*{(3.36.1) and (4.2.10)}). It follows that $(\Gamma_1\cdot E)=0$, $(K_S+E\cdot E)=0$ and $(1-a)(E^2)=(-K_S-aE\cdot E)\ge (-K_S-\Delta\cdot E)>0$ (since $-(K_S+\Delta)$ is ample), hence $E$ is nef and big. Suppose that $\Gamma_1\neq 0$, then by Hodge index theorem we have $(\Gamma_1^2)<0$; but $(\Gamma_1^2)=(-K_X-\Delta-(1-a)E\cdot \Gamma_1)=(-K_S-\Delta\cdot \Gamma_1)>0$, a contradiction. Therefore $\Gamma_1=0$ and $-r(K_S+\Delta)\sim_\bQ r(1-a)E$. In particular, $E$ is ample and $\W$ is almost complete as in Example \ref{ex:graded series from ample div}.
\end{proof}

For actual calculations, it would be convenient to have a formula for $F(\W)$ before we apply Corollary \ref{cor:adj delta ac} to the almost complete refinement $\W$. This can be done using Zariski decomposition on surfaces.

\begin{lem} \label{lem:F(W) formula on surface}
In the setup of Lemma \ref{lem:surface almost complete}, assume that $L$ is ample and let $\pi\colon T\to S$ be the prime blow up associated to $E$. Then we have
\begin{equation} \label{eq:F(W) surface}
    F(\W)=\frac{2}{(L^2)}\int_0^\infty \left( \vol_{T|E}(\pi^*L-tE)\cdot N_\sigma (\pi^*L-tE)|_E \right) \rd t
\end{equation}
where $\vol_{T|E}(\cdot)$ is the restricted volume function $($see \cite{ELMNP}$)$ and $N_\sigma (\cdot)$ denotes the negative part in the Zariski decomposition of a $($pseudo-effective$)$ divisor.
\end{lem}

\begin{proof}
Since $L$ is ample, it is easy to see that $\Supp(\W)\cap (\{1\}\times \bR)=[0,T(L;E)]$. By Corollary \ref{cor:F(W) formula}, we then have
\[
F(\W)=\frac{2}{\vol(\W)}\int_0^{T(L;E)} F(\gamma)\vol_{\W}(\gamma) \rd \gamma
\]
where $F(\gamma)=\lim_{m\to\infty}\frac{1}{m} F_{m,\lfloor m\gamma \rfloor}$. By construction, we have $\vol(\W)=\vol(\V)=\vol(L)$, $\vol_{\W}(\gamma)=\vol_{T|E}(\pi^*L-\gamma E)$ and $\vol_{T|E}(\pi^*L-\gamma E)=0$ when $\gamma>T(L;E)$. Thus it suffices to show that
\begin{equation} \label{eq:F=N_sigma}
    F(\gamma)=N_\sigma (\pi^*L-\gamma E)|_E.
\end{equation}
% the formula \eqref{eq:F(W) surface} then follows from a change of variable $\gamma = rt$.
By continuity, it is enough to check \eqref{eq:F=N_sigma} when $\gamma\in (0,T(L;E))\cap \bQ$. Let $\pi^*L-\gamma E = P + N$ be the Zariski decomposition as in the proof of Lemma \ref{lem:surface almost complete} and let $m$ be a sufficiently divisible integer. Since $L$ is ample, $E$ is not contained in the stable base locus $\Bs(\pi^*L)$ of $\pi^*L$. Since there always exists some $\gamma'\ge \gamma$ such that $E\not\subseteq \Bs(\pi^*L-\gamma' E)$ (e.g. we take $\gamma'=\ord_E(D)$ for any $D\sim_\bQ L$ with $\ord_E(D)\ge \gamma$) and $\pi^*L-\gamma E$ is a convex linear combination of $\pi^*L$ and $\pi^*L-\gamma' E$, we see that $E\not\subseteq \Bs(\pi^*L-\gamma E)$ as well. In particular, $E\not\subseteq \Supp(N)$. Then clearly $F_{m,m\gamma}\ge mN|_E$ and hence $F(\gamma)\ge N|_E$. From the proof of Lemma \ref{lem:surface almost complete} we also see that there exists some constant $A$ (depending only on $(S,\Delta)$ and $E$) such that the restricted linear series $|mP|_E$ has codimension at most $A$ in $|\cO_E(mP)|$, thus the degree of $F_{m,m\gamma} - mN|_E$ is at most $A$. Letting $m\to \infty$ we obtain $\deg F(\gamma) = \deg (N|_E)$, which implies \eqref{eq:F=N_sigma} as $F(\gamma)\ge N|_E$.
% Let
% \[
% \pi^*L-\gamma E=P(\gamma)+N(\gamma)=:P+N
% \]
% be the Zariski decomposition as before, where $P=P(\gamma)$ (resp. $N=N(\gamma)$) is the positive (resp. negative) part. Note that $\gamma\mapsto N(\gamma)$ is a convex function; it follows that $E\not\subseteq \Supp(N)$ as we can always find $\gamma'\ge \gamma$ such that $E\not\subseteq \Supp(N(\gamma'))$. It further implies that $\gamma\mapsto (P(\gamma)\cdot E)\ge 0$ is concave and since $P(\epsilon)=r\pi^*L-\epsilon E$ is ample when $0<\epsilon\ll 1$, we have $(P\cdot E)>0$. Now let $m\ge 0$ be a sufficiently divisible integer. Then clearly $F_{m,m\gamma}\ge mN|_E$ and hence $F(\gamma)\ge N|_E$. From the proof of Lemma \ref{lem:surface almost complete} we also see that there exists some constant $A>0$ (depending only on $(S,\Delta)$ and $E$) such that the restricted linear series $|mP|_E$ has codimension at most $A$ in $|\cO_E(mP)|$, thus the degree of $F_{m,m\gamma} - mN|_E$ is at most $A$. Letting $m\to \infty$ we obtain $\deg F(\gamma) = \deg (N|_E)$, which implies \eqref{eq:F=N_sigma} as $F(\gamma)\ge N|_E$.
\end{proof}

As an illustration, we compute the $\delta$-invariants of all smooth cubic surfaces. Some of these will be useful in our proof of the K-stability of cubic threefolds (Lemma \ref{lem:cubic 3fold}).

\begin{thm} \label{thm:cubic delta}
Let $X\subseteq \bP^3$ be a smooth cubic surface and let $x\in X$ be a closed point. Let $C=T_x(X)\cap X$ be the tangent hyperplane section. Then
\[
\delta_x(X) = 
\begin{cases}
    3/2 & \text{if } \mult_x C = 3,\\
    27/17 & \text{if } C \text{ has a tacnode at } x,\\
    5/3 & \text{if } C \text{ has a cusp at } x,\\
    18/11 & \text{if } C \text{ is the union of three lines and } \mult_x C=2,\\
    12/7 & \text{if } C \text{ is irreducible and has a node at } x,\\
    \frac{9}{25-8\sqrt{6}} & \text{if } C \text{ is the union of a line and a conic that intersects transversally}.
\end{cases}
\]
Moreover, in the first three cases, $\delta_x(X)$ is computed by the $($unique$)$ divisor that computes $\lct_x(X,C)$; in the next two cases, $\delta_x(X)$ is computed by the ordinary blow up of $x$; in the last case, $\delta_x(X)$ is computed by the quasi-monomial valuation over $x\in X$ with weights $1+\sqrt{6}$ on the line and $2$ on the conic, and if $0<\varepsilon\ll 1$ then the log del Pezzo pair $(X,(1-\varepsilon)C)$ satisfies $\delta(X,(1-\varepsilon)C)=\frac{9}{25-8\sqrt{6}}\not\in\bQ$. 
\end{thm}

\begin{proof}
See Appendix \ref{sec:appendix}.
\end{proof}

\begin{cor} \label{cor:cubic surface global delta}
Let $X\subseteq \bP^3$ be a smooth cubic surface. Then
\[
\delta(X) = 
\begin{cases}
    3/2 & \text{if } X \text{ has an Eckardt point},\\
    27/17 & \text{otherwise}.
\end{cases}
\]
\end{cor}

It has been expected (see e.g. \cite{AIM}) that given a klt Fano variety $X$ with a $\bQ$-complement $\Delta$, the graded rings
\[
\gr_v R:=\oplus_{m,\lambda} \Gr^\lambda_{\cF_v} H^0(X,-mrK_X)
\]
are finitely generated for all lc places $v$ of $(X,\Delta)$, where $r>0$ is an integer such that $-rK_X$ is Cartier and $\cF_v$ is the filtration induced by $v$. Unfortunately this is not true in general and we identify a counterexample through the calculations in Theorem \ref{thm:cubic delta}.

\begin{thm} \label{thm:HRFG counterexample}
Let $X\subseteq \bP^3$ be a smooth cubic surface and let $C\subseteq X$ be a hyperplane section such that $C=L\cup Q$ is the union of a line and a conic that intersects transversally. Then there exists an lc place $v$ of $(X,C)$ such that $\gr_v R$ is not finitely generated. \end{thm}

\begin{proof}
This can be deduced from the fact that $\delta_x(X)\not\in\bQ$ where $x\in L\cap Q$. Here we give a more direct (and simpler) argument.

Let $x\in L\cap Q$ and let $a,b>0$ be coprime integers. Let $\pi\colon Y=Y_{a,b}\to X$ be the weighted blow up at $x$ with $\wt(L)=a$ and $\wt(Q)=b$. Let $E$ be the exceptional divisor and let $\tL$ (resp. $\tQ$) be the strict transform of $L$ (resp. $Q$). Assume that $b<2a$. We have $(\tL^2)=-1-\frac{a}{b}$, $(\tQ^2)=-\frac{b}{a}$, $(\tL\cdot \tQ)=1$ and in particular the intersection matrix of $\tL$ and $\tQ$ is negative definite. As $-\pi^*K_X-(a+b)E\sim \tL+\tQ$, it follows that $T(-K_X;E)=a+b$, the stable base locus of $-\pi^*K_X-tE$ is supported on $\tL\cup\tQ$ for all $0\le t\le a+b$, and hence $N_\sigma(-\pi^*K_X-tE)=f(t)\tL+g(t)\tQ$ for some $f(t),g(t)\ge 0$. The coefficients $f(t)$ and $g(t)$ are computed as the smallest numbers such that $-\pi^*K_X-tE-f(t)\tL-g(t)\tQ$ is nef, and it is enough to check nefness against $\tL$ and $\tQ$. A straightforward computation then gives
\begin{equation} \label{eq:negative part line+conic}
    N_\sigma(-\pi^*K_X-tE) =
    \begin{cases}
        0 & \text{if } 0\le t\le b, \\
        \frac{t-b}{a+b}\tL & \text{if } b<t\le \frac{a(2a+3b)}{2a+b}, \\
        \frac{2t-2a-b}{b}\tL + \frac{(2a+b)t-a(2a+3b)}{b^2}\tQ & \text{if } \frac{a(2a+3b)}{2a+b}<t\le a+b.
    \end{cases}
\end{equation}
The key point is that as a rational function, $\frac{a(2a+3b)}{(2a+b)}$ is not a linear combination of $a$ and $b$. In particular, we may choose $a_0,b_0\in\bR_+$ such that $b_0<2a_0$ and $a_0,b_0,\frac{a_0(2a_0+3b_0)}{(2a_0+b_0)}$ are linearly independent over $\bQ$ (thus $\frac{a_0}{b_0}$ is necessarily irrational). Let $v_0$ be the quasi-monomial valuation centered at $x$ given by $\wt(L)=a_0$ and $\wt(Q)=b_0$. We claim that $\gr_{v_0} R$ is not finitely generated.

Suppose that $\gr_{v_0} R$ is finitely generated and let $f_i$ ($i=1,2,\cdots,s$) be a finite set of homogeneous generators ($\gr_{v_0} R$ is naturally graded by $\bN\times (\bN a_0+\bN b_0)$). Let $\deg(f_i)=(m_i,\lambda_i=p_i a_0+q_i b_0)$ where $m_i,p_i,q_i\in\bN$. We may assume that $0=\frac{\lambda_1}{m_1}\le \frac{\lambda_2}{m_2} \cdots\le \frac{\lambda_s}{m_s}$. Clearly $\frac{\lambda_s}{m_s}\ge a_0+b_0>\frac{a_0(2a_0+3b_0)}{(2a_0+b_0)}$ (otherwise $v_0(s)<a_0+b_0$ for all $s\in H^0(X,-K_X)$; but $v_0(L+Q)= a_0+b_0$). Since $a_0,b_0$ and $\frac{a_0(2a_0+3b_0)}{(2a_0+b_0)}$ are linearly independent over $\bQ$, there exists $1\le \ell<s$ such that
\[
\frac{\lambda_\ell}{m_\ell} < \frac{a_0(2a_0+3b_0)}{(2a_0+b_0)} < \frac{\lambda_{\ell+1}}{m_{\ell+1}}.
\]
We may lift each $f_i$ to $g_i\in R_{m_i}=H^0(X,-m_i K_X)$ such that ${\rm in}_{v_0}(g_i)=f_i$. Then for all $\alpha=(a,b)\in\bQ^2$ with $|\alpha-(a_0,b_0)|\ll 1$, we have $\mu_i:=v_\alpha (g_i)=p_i a +q_i b$ (where $v_\alpha$ is the quasi-monomial valuation with $\wt(L)=a$ and $\wt(Q)=b$); in particular, $v_\alpha (g_i)>m_i\cdot \frac{a(2a+2b)}{2a+b}$ when $i\ge \ell+1$, thus by \eqref{eq:negative part line+conic}, $g_i$ vanishes on $Q$ for all $i\ge \ell+1$. We may also assume that
\begin{equation} \label{eq:slopes}
    0=\frac{\mu_1}{m_1}\le \cdots \le \frac{\mu_\ell}{m_\ell} < \frac{a(2a+2b)}{2a+b} < \frac{\mu_{\ell+1}}{m_{\ell+1}} \le \cdots \le \frac{\mu_s}{m_s}.
\end{equation}
By \cite{LX-higher-rank}*{Lemma 2.10}, $g_i$ restrict to a finite set of generators of $\gr_{v_\alpha} R$. It follows that for any $g\in R_m=H^0(X,-mK_X)$, we have
\begin{equation} \label{eq:v_alpha}
    v_\alpha(g) = \max \{\wt(F)\,|\,F\in \bC[x_1,\cdots,x_s]\text{ s.t. } F(g_1,\cdots,g_s)=g\}
\end{equation}
where we set $\wt(x_i)=\mu_i$ (clearly if $g=F(g_1,\cdots,g_s)$ then $v_\alpha(g)\ge \wt(F)$; conversely, as $g_i$ generate $\gr_{v_\alpha} R$, there exists $F$ such that $\wt(F)=v_\alpha(g)$ and $g = F(g_1,\cdots,g_s) \mod \cF_{v_\alpha}^{>v_\alpha(g)}R_m $, one can then prove by induction on $v_\alpha(g)$ that $g=F(g_1,\cdots,g_s)$ for some $\wt(F)=v_\alpha(g)$). Now let $\lambda = \frac{a(2a+2b)}{2a+b}$. By \eqref{eq:negative part line+conic}, for sufficiently divisible integers $m>0$, there exists $f\in H^0(X,-mK_X)$ such that $v_\alpha(f) = \lambda m$ and $f$ does not vanish on $Q$. By \eqref{eq:v_alpha} we have $f=F(g_1,\cdots,g_s)$ for some $F$ with $\wt(F)=m\lambda$. However, by \eqref{eq:slopes} we see that each monomial in $F$ must contain some $g_i$ with $i\ge \ell+1$; it follows that $f=F(g_1,\cdots,g_s)$ vanishes along $Q$, a contradiction. Therefore, $\gr_{v_0} R$ is not finitely generated.
\end{proof}

\subsection{Hypersurfaces with Eckardt points} \label{sec:eckardt}

Let $X\subseteq \bP^{n+1}$ be a smooth hypersurface of degree $d\ge 2$. Recall that $x\in X$ is called a generalized Eckardt point if the tangent hyperplane section $D=T_x X\cap X\subseteq X$ at $x$ satisfies $\mult_x D=d$. In this case $D$ is isomorphic to the cone over $F(X,x)$, the Hilbert scheme of lines in $X$ passing through $x$, which is a hypersurface of degree $d$ in $\bP^{n-1}$. It is in fact smooth by the following easy lemma.

\begin{lem}
Let $X\subseteq \bP^{n+1}$ be a smooth hypersurface of degree $d\ge 2$ and let $x\in X$ be a generalized Eckardt point. Then $F(X,x)$ is smooth.
\end{lem}

\begin{proof}
We may assume that $x=[0:\cdots:0:1]$. Up to a change of coordinates $X$ is defined by an equation of the form $x_0 f(x_1,\cdots,x_{n+1}) + g(x_1,\cdots,x_n) = 0$ where $\deg f = d-1$, $\deg g=d$ and $f$ contains the monomial $x_{n+1}^{d-1}$. We then have $F(X,x)\cong (g=0)\subseteq \bP^{n-1}$. If $[a_1:\cdots:a_n]$ is a singular point of $F(X,x)$, then for any $a_{n+1}$ with $f(a_1,\cdots,a_{n+1})=0$ (such $a_{n+1}$ exists since $f$ contains the monomial $x_{n+1}^{d-1}$) it is not hard to check that $X$ is singular at $[0:a_1:\cdots:a_{n+1}]$. This is a contradiction as $X$ is smooth. Thus $F(X,x)$ is smooth.
\end{proof}

\begin{thm} \label{thm:delta at Eckardt pt}
Let $X\subseteq \bP^{n+1}$ be a smooth hypersurface of degree $d\ge 2$ and let $x\in X$ be a generalized Eckardt point. Assume that $F(X,x)$ is K-semistable if $d\le n-1$ $($i.e. when it's Fano$)$. Then $\delta_x(H)=\frac{n(n+1)}{d+n-1}$ $($where $H$ is the hyperplane class on $X)$ and it is computed by the ordinary blow up of $x$.
\end{thm}

\begin{proof}
Let $\pi\colon Y\to X$ be the blow up of $x$ and let $E$ be the exceptional divisor. Let $\V$ be the complete linear series associated to $H$ and let $\W$ be its refinement by $E$. Since $x\in X$ is a generalized Eckardt point, the tangent hyperplane section $x\in D\subseteq X$ has $\mult_x D=d$. Let $\tD$ be the strict transform of $D$ on $Y$. Let $j,m\in \bN$. Note that $|m\pi^*H-jE|\neq \emptyset$ if and only if $0\le j\le dm$ and it is base point free when $0\le j\le m$. We first show that
\begin{equation} \label{eq:mov+fix}
|m\pi^*H-jE| = \left| \left(m-\lceil \frac{j-m}{d-1}\rceil\right) \pi^*H - \left(j-d\cdot\lceil \frac{j-m}{d-1}\rceil\right) E \right| +  \lceil \frac{j-m}{d-1} \rceil \tD
\end{equation}
is the decomposition into movable and fixed part when $m\le j\le dm$. 

Suppose first that $n\ge 3$. Then $D$ is irreducible. Let $D'\sim_\bQ H$ be another effective $\bQ$-divisor that doesn't contain $D$ in its support. We have 
\[
d\cdot \mult_x D'\le (D\cdot D'\cdot H_1\cdot\cdots\cdot H_{n-2})=d
\]
where $H_1,\cdots,H_{n-2}$ are general hyperplane sections passing through $x$, hence $\mult_x D'\le 1$. It follows that for any $G\in |m\pi^*H-jE|$, if we write $G=a\tD + G'$ where $\tD\not\subseteq \Supp(G')$, then $G'\in |(m-a)\pi^*H - (j-ad)E|$ and 
\[
j-ad\le \mult_x \pi(G')\le m-a.
\]
In other words, $a\ge \lceil \frac{j-m}{d-1} \rceil$ which implies \eqref{eq:mov+fix}. If $n=2$, then in the above same notation $\tD$ is a disjoint union of $d$ lines $L_1,\cdots,L_d$. If we take $G\in |m\pi^*H-jE|$ and write $G=\sum a_i L_i + G'$ where $G'$ doesn't contain any $L_i$ ($i=1,\cdots,d$) in its support then as $(G'\cdot L_i)\ge 0$ we obtain $m-j=(G\cdot L_i)\ge a_i(L_i^2) = a_i(1-d)$, thus $a_i\ge \lceil \frac{j-m}{d-1} \rceil$ for all $1\le i\le d$ and \eqref{eq:mov+fix} still holds.

It is straightforward to check that for all $0\le j\le m$, the natural restriction maps
\[
H^0(Y,\cO_Y(m\pi^*H-jE))\to H^0(E,\cO_E(j))
\]
are surjective. It follows that 
\[
\vol(\pi^*H-tE)=
    \begin{cases}
        d-t^n & \text{if } 0\le t\le 1, \\
        \frac{(d-t)^n}{(d-1)^{n-1}} & \text{if } 1<t\le d.
    \end{cases}
\]
and
\[
W_{m,j}=
    \begin{cases}
        H^0(E,\cO_E(j)) & \text{if } 0\le j\le m, \\
        {\rm Im}(
        H^0(E,\cO_E(j-d\lceil \frac{j-m}{d-1}\rceil)) \xrightarrow{\cdot \lceil \frac{j-m}{d-1}\rceil D_0}  H^0(E,\cO_E(j))
        ) & \text{if } m\le j\le dm \\
        0 & \text{otherwise}.
    \end{cases}
\]
where $D_0=\tD\cap E\cong F(X,x)$. In particular, $\W$ is almost complete and through direct calculations we see that $S(H;E)=\frac{d+n-1}{n+1}$, $F(\W)=\frac{1}{n+1}(1-\frac{1}{d})D_0$ (by Corollary \ref{cor:F(W) formula}) and $c_1(\W)\sim_\bQ (\pi^*H-S(H;E)\cdot E)|_E\sim \frac{d+n-1}{n+1}H_0$ (see \eqref{eq:c_1 of refinement}) where $H_0$ is the hyperplane class on $E\cong\bP^{n-1}$. 

Clearly $\delta_x(H)\le \lambda$ where $\lambda = \frac{n(n+1)}{d+n-1} = \frac{A_X(E)}{S(H;E)}$. It remains to prove $\delta_x(H)\ge \lambda$. Let $M_{\vb}$ be the movable part of $\W$. By Corollary \ref{cor:adj delta ac}, it suffices to prove
\begin{equation} \label{eq:delta(E) over Eckardt}
    \delta(E,\lambda F(\W);c_1(M_{\vb}))\ge \lambda.
\end{equation}
Note that by the above calculations we have 
\[
\lambda c_1(M_{\vb})+\lambda F(\W) \sim_\bQ \lambda c_1(\W)\sim_\bQ nH_0\sim -K_E,
\]
thus \eqref{eq:delta(E) over Eckardt} is equivalent to saying that the pair
\[
(E,\lambda F(\W))\cong (\bP^{n-1},\frac{n(d-1)}{d(d+n-1)}D_0)
\]
is K-semistable. By \cite{Der-finite-cover}*{Lemma 2.6} this would be true if $(\bP^{n-1},\mu D_0)$ is K-semistable for some $\mu\ge \frac{n(d-1)}{d(d+n-1)}$ (as $\bP^{n-1}$ is K-semistable). When $d\ge n$, $(\bP^{n-1},\frac{n}{d} D_0)$ is a log canonical log Calabi-Yau pair (note that $D_0$ is smooth) and therefore is K-semistable by \cite{Oda-slcCY-Kss}*{Corollary 1.1}; thus we may take $\mu=\frac{n}{d}$. When $d\le n-1$, $D_0$ is Fano and K-semistable by assumption. We claim that $(\bP^{n-1},\mu D_0)$ is K-semistable where $\mu=1-\frac{1}{d}+\frac{1}{n}>\frac{n(d-1)}{d(d+n-1)}$. Indeed, the divisor $D_0$ induces a special degeneration of $(\bP^{n-1},\mu D_0)$ to $(V,\mu V_\infty)$ where $V=C_p(D_0,N_{D_0/E})$ is the projective cone over $D_0$. By \cite{LX-Kollar-comp}*{Proposition 5.3}, $(V,\mu V_\infty)$ is K-semistable, thus $(\bP^{n-1},\mu D_0)$ is also K-semistable by the openness of K-semistability \cites{Xu-quasimonomial,BLX-openness}. This proves the claim and also concludes the proof of the theorem.
\end{proof}

Restricting to Fano hypersurfaces, we have

\begin{cor} \label{cor:Fano Eckardt pt}
Let $X\subseteq \bP^{n+1}$ be a smooth Fano hypersurface of degree $d$ and let $x\in X$ be a generalized Eckardt point. Assume that $F(X,x)$ is K-semistable if $d\le n-1$. Then $\delta_x(X)=\frac{n(n+1)}{(n-1+d)(n+2-d)}$ and it is computed by the ordinary blow up of $x$.
\end{cor}

\begin{proof}
If $d=1$ then $X\cong\bP^n$ and $\delta_x(X)=1=\frac{n(n+1)}{(n-1+d)(n+2-d)}$ for all $x\in X$. If $d\ge 2$, then as $-K_X\sim (n+2-d)H$ where $H$ is the hyperplane class, we have $\delta_x(X)=\frac{1}{n+2-d}\delta_x(H)=\frac{n(n+1)}{(n-1+d)(n+2-d)}$ by Theorem \ref{thm:delta at Eckardt pt}. 
\end{proof}

Since every point of a smooth quadric hypersurface is a generalized Eckardt point, we obtain a new algebraic proof of the following well-known result.

\begin{cor}
Quadric hypersurfaces are K-semistable.
\end{cor}

\begin{proof}
Let $d=2$ in Corollary \ref{cor:Fano Eckardt pt}. Since every $x\in X$ is a generalized Eckardt point and $F(X,x)$ is a smooth quadric hypersurface of smaller dimension, we get $\delta(X)=1$ by induction on the dimension.
\end{proof}

\subsection{Hypersurfaces of index two} \label{sec:index two}

The goal of this section is to prove the following result.

\begin{thm} \label{thm:index two}
Let $X=X_n\subseteq \bP^{n+1}$ be a smooth Fano hypersurface of degree $n\ge 3$ $($i.e. it has Fano index $2)$. Then $X$ is uniformly K-stable.
\end{thm}

Note that when $n=3$, i.e. $X$ is a cubic threefold, the result is already known by \cite{LX-cubic-3fold}. Here we give a different proof using techniques developed in previous sections.

\begin{lem} \label{lem:cubic 3fold}
Let $X\subseteq \bP^4$ be a smooth cubic threefold. Then $X$ is uniformly K-stable.
\end{lem}

\begin{proof}
It suffices to show that $\delta_x(X)>1$ for all $x\in X$. If $x$ is a generalized Eckardt point, then $\delta_x(X)=\frac{1}{2}\delta_x(H)=\frac{6}{5}>1$ by Theorem \ref{thm:delta at Eckardt pt}. If $x$ is not a generalized Eckardt point, then there are only finitely many lines on $X$ passing through $x$, thus if $Y\subseteq X$ is a general hyperplane section passing through $x$, then $Y$ is a smooth cubic surface such that $x$ is not contained in any lines on $Y$. By Theorem \ref{thm:cubic delta}, we see that $\delta_x(Y)\ge\frac{5}{3}$. It then follows from Lemma \ref{lem:F=ample Cartier} that $\delta_x(X)=\frac{1}{2}\delta_x(H)\ge \frac{2}{3}\delta_x(Y)\ge \frac{10}{9}>1$. This completes the proof.
\end{proof}

In the remaining part of this section, we will henceforth assume that $n\ge 4$. As a key step towards the proof of Theorem \ref{thm:index two}, we observe the following K-stability criterion.

\begin{lem} \label{lem:(n+1)/n-criterion}
Let $X$ be a Fano manifold of dimension $n$. Assume that 
\begin{enumerate}
    \item $\delta_Z(X)\ge \frac{n+1}{n}$ for any subvariety $Z\subseteq X$ of dimension $\ge 1$,
    \item $\beta_X(E_x) > 0$ for any $x\in X$ where $E_x$ denotes the exceptional divisor of the ordinary blowup of $x$.
\end{enumerate}
Then $X$ is uniformly K-stable.
\end{lem}

\begin{proof}
We need to show that for any valuation $v\in\Val_X^*$ with $A_X(v)<\infty$ we have $\beta_X(v)>0$. By our first assumption, this holds if the center of $v$ has dimension at least one. Thus we may assume that the center of $v$ is a closed point $x\in X$ and by our second assumption we may assume that $v\neq c\cdot \ord_{E_x}$. Let $r$ be a sufficiently large integer such that $-rK_X$ is very ample and let $H\in |-rK_X|$ be a general member (in particular, $x\not\in H$). By Proposition \ref{prop:delta = delta wrt filtration}, we have $S(-K_X;v)=S(\V,\cF;v)$ where $\V$ is the complete linear series associated to $-K_X$ and $\cF$ is the filtration induced by $H$. Let $m\gg 0$ and let $D$ be an $m$-basis type $\bQ$-divisor of $-K_X$ that's compatible with $\cF$. As in the proof of Lemma \ref{lem:S<=n/(n+1)T}, we have $D=\mu_m\cdot H + \Gamma$ where $\mu_m=S_m(-K_X;H)\to S(-K_X;H)=\frac{1}{r(n+1)}$ ($m\to \infty$) and $\Gamma\sim_\bQ -(1-r\mu_m)K_X$ is effective. By \cite{BJ-delta}*{Corollary 3.6}, there exist constants $\epsilon_m\in (0,1)$ ($m\in\bN$) depending only on $X$ such that $\epsilon_m\to 1$ ($m\to \infty$) and
\[
S(-K_X;v) > \epsilon_m\cdot S_m(-K_X;v)
\]
for all valuations $v\in \Val_X^*$ with $A_X(v)<\infty$ and all $m\in \bN$. Perturbing the $\epsilon_m$, we will further assume that $\epsilon_m(1-r\mu_m)< \frac{n}{n+1}$. Combining with our first assumption we see that
\[
\delta_{Z,m}(X) > \epsilon_m\cdot \delta_Z(X)\ge \frac{(n+1)\epsilon_m}{n}
\]
for any subvariety $Z\subseteq X$ of dimension $\ge 1$. It follows that $(X,\frac{(n+1)\epsilon_m}{n}D)$ is klt in a punctured neighbourhood of $x$ and so does $(X,\frac{(n+1)\epsilon_m}{n}\Gamma)$. Note that $-(K_X+\frac{(n+1)\epsilon_m}{n}\Gamma)\sim_\bQ -(1-\frac{n+1}{n}\epsilon_m(1-r\mu_m))K_X$ is ample, thus by the following Lemma \ref{lem:lct at isolated pt}, there exists some $\lambda=\frac{\mu}{\mu+1}\cdot \frac{n+1}{n}>1$ (where $\mu=\frac{A_X(v)}{v(\fm_x)}>n$) such that
\[
A_X(v)\ge \lambda\epsilon_m \cdot v(\Gamma) = \lambda\epsilon_m \cdot v(D)
\]
where the last equality holds since $x\not\in H$. Since $D$ is arbitrary we obtain $A_X(v)\ge \lambda\epsilon_m \cdot S_m(\V,\cF;v)$; letting $m\to \infty$ we deduce $A_X(v)\ge \lambda S(-K_X;v)>S(-K_X;v)$. This completes the proof.
\end{proof}

The following result is used in the above proof.

\begin{lem} \label{lem:lct at isolated pt}
Let $x\in X$ be a smooth point on a projective variety. Let $D$ be an effective $\bQ$-divisor on $X$ such that $(X,D)$ is klt in a punctured neighbourhood of $x$ and $-(K_X+D)$ is ample. Let $v\in\Val_X^*$ be a valuation with $A_X(v)<\infty$ that's centered at $x$ and let $\mu=\frac{A_X(v)}{v(\fm_x)}$. Then 
\begin{enumerate}
    \item $\mu\ge \dim X$ and equality holds if and only if $v=c\cdot \ord_E$ for some $c>0$, where $E$ is the exceptional divisor of the blowup of $x$.
    \item $A_X(v)\ge \frac{\mu}{\mu+1}\cdot v(D)$. 
\end{enumerate}
\end{lem}

\begin{proof}
Let $n=\dim X$. The first part follows from the fact that $(X,\fm_x^n)$ is lc and the only lc place is the exceptional divisor coming from the blowup of $x$. The second part essentially follows from the proof of \cite{Z-cpi}*{Theorem 1.6}, which we reproduce here for reader's convenience. Let $\cJ=\cJ(X,D)$ be the multiplier ideal of $(X,D)$. We may assume that $(X,D)$ is not lc at $x$ (otherwise $A_X(v)\ge v(D)$ and we are done), hence $\cJ_x\neq \cO_{X,x}$. By assumption we have $\cJ=\cO_X$ in a punctured neighbourhood of $x$. Since $-(K_X+D)$ is ample, we have $H^1(X,\cJ)=0$ by Nadel vanishing and hence a surjection $H^0(\cO_X)\twoheadrightarrow H^0(\cO_X/\cJ)\twoheadrightarrow H^0(\cO_{X,x}/\cJ_x)$. Since $h^0(X,\cO_X)=1$, we see that $\cJ_x=\fm_x$ and thus $v(\cJ)=v(\fm_x)=\frac{A_X(v)}{\mu}$. Through the definition of multiplier ideals we also have $v(\cJ)\ge v(D)-A_X(v)$. Combined with the previous equality it implies $A_X(v)\ge \frac{\mu}{\mu+1}\cdot v(D)$.
\end{proof}

In order to prove the K-stability of smooth hypersurfaces of Fano index two, it remains to verify the two conditions in Lemma \ref{lem:(n+1)/n-criterion}. The following lemma takes care of the (easier) second condition.

\begin{lem} \label{lem:beta(E_x)>0 on index two}
Let $X\subseteq \bP^{n+1}$ be a smooth Fano hypersurface of degree $d$. Let $r=n+2-d$ be its Fano index. Assume that $d\ge 3$ and $n+1\ge r^2$. Then $\beta_X(E_x)>0$ for any $x\in X$ where $E_x$ is the exceptional divisor of the ordinary blowup of $x$.
\end{lem}

\begin{proof}
Let $H$ be the hyperplane class on $X$, let $T=T(H;E_x)$ be the pseudo-effective threshold and $\eta=\eta(H;E_x)$ the movable threshold (see Lemma \ref{lem:S<=T+eta}). Clearly $1\le \eta\le T\le d$. Let $\pi\colon Y\to X$ be the blowup of $x$. Then as $\pi^*H-E_x$ is nef, we see that
\[
(\pi^*H-E_x)^{n-2}\cdot (\pi^*H-\eta E_x)\cdot (\pi^*H- T E_x)\ge 0,
\]
thus $\eta T\le d$ and $\eta\le \sqrt{d}$. By Lemma \ref{lem:S<=T+eta}, we then have 
\[
S(H;E_x)\le \frac{1}{n+1}T+\frac{n-1}{n+1}\eta \le \frac{d}{(n+1)\eta}+\frac{n-1}{n+1}\eta.
\]
When $1\le \eta\le \sqrt{d}$, the right hand side of the above inequality achieves its maximum at either $\eta=1$ or $\eta=\sqrt{d}$, hence as $3\le d\le n+1$ and $n+1\ge r^2$ we obtain
\[
S(-K_X;E_x)=r\cdot S(H;E_x) \le \max\left\{ \frac{(n+2-d)(n-1+d)}{n+1}, \frac{rn\sqrt{d}}{n+1} \right\} < n = A_X(E_x).
\]
In other words, $\beta_X(E_x)>0$.
\end{proof}

We now focus on checking the first condition of Lemma \ref{lem:(n+1)/n-criterion}. The basic idea, similar to the proof of Lemma \ref{lem:delta small deg}, is to apply Theorem \ref{thm:delta adj via flag} to an admissible flag of complete intersection subvarieties on the hypersurface $X$. At the end we relate $\delta_Z(X)$ to the stability threshold of a divisor of degree close to $4$ on a curve $C$ (i.e. the $1$-dimensional subvariety in the chosen flag). However, this only gives the na\"ive bound $\delta_Z(X)\ge \delta(-\frac{2}{n+1}K_X|_C)=\frac{n+1}{2n}$ (since $\delta(L)=\frac{2}{\deg L}$ for any ample line bundle $L$ on a curve) and is not good enough for our purpose. To get a better estimate, we choose a flag such that $C$ intersects $Z$ in at least two points $P,Q$ (which is possible since $\dim Z\ge 1$). We still have the freedom to choose another point $R\neq P,Q$ on $C$ to put in our flag. The key observation is that (asymptotically) basis type $\bQ$-divisors of degree $4$ on $C$ that are compatible with $R$ have multiplicity $2$ at the point $R$ and therefore must be log canonical at one of $P$ or $Q$ for degree reason. In other words, the stability threshold along $Z\cap C$ is at least one and this is exactly what we need.

We work out the details in the next several lemmas. The first thing is to make sure that the admissible flag we want to use exists.

\begin{lem} \label{lem:sm hyp sec thru 2 pts}
Let $Y\subseteq \bP^{m+1}$ be a smooth hypersurface of dimension $m\ge 2$ and let $P\neq Q$ be two distinct points on $Y$. Let $H\subseteq Y$ be a general hyperplane section containing both $P$ and $Q$. Then $H$ is smooth unless $m=2$ and $Y$ contains the line joining $P$ and $Q$.
\end{lem}

\begin{proof}
Let $\ell\subseteq \bP^{m+1}$ be the line joining $P$ and $Q$ and let $\cM\subseteq |\cO_Y(1)|$ be the linear system of hyperplane sections containing $P,Q$. If $\ell\not\subseteq Y$, then $\cM$ only has isolated base points $\ell\cap Y$ and by Bertini's theorem $H$ is smooth away from these points. On the other hand, since $H$ is general, it is different from the tangent hyperplane of any $x\in \ell\cap Y$, hence $H$ is also smooth at any $x\in \ell\cap Y$. Thus we may assume that $\ell\subseteq X$. Again $H$ is smooth away from $\ell$ by Bertini's theorem. The tangent hyperplanes of $x\in \ell$ give a $1$-dimensional family of members of $\cM$. Hence they are different from $H$ as long as $\dim \cM = m-1\ge 2$. It follows that $H$ is smooth when $m\ge 3$. 
\end{proof}

In the remaining part of this subsection, let $X\subseteq \bP^{n+1}$ be a smooth hypersurface of degree $n\ge 4$ and $Z\subseteq X$ a subvariety of dimension at least one. We divide into two cases to show that $\delta_Z(X)\ge \frac{n+1}{n}$. First we treat the case when $X$ doesn't contain the secant variety of $Z$.

\begin{lem} \label{lem:secant not in X}
In the above notation, assume that there exist closed points $P\neq Q\in Z$ such that the line joining $P$ and $Q$ is not contained in $X$. Then $\delta_Z(X)\ge \frac{n+1}{n}$.
\end{lem}

\begin{proof}
By Lemma \ref{lem:sm hyp sec thru 2 pts}, there exists a flag
\[
\Y \quad:\quad X=Y_0\supseteq Y_1\supseteq \cdots \supseteq Y_n
\]
on $X$ such that each $Y_i$ $(1\le i\le n-1$) is a smooth hyperplane section of $Y_{i-1}$ containing $P,Q$ and $Y_n$ is a smooth point on the curve $Y_{n-1}$ that's different from $P,Q$. Let $\V$ be the complete linear series associated to $-K_X$, let $\Y^{(j)}$ be the truncated flag given by $Y_0\supseteq \cdots\supseteq Y_j$ and let $\W^{(j)}$ be the refinement of $\V$ by $\Y^{(j)}$. It is equipped with a filtration $\cF$ induced by $\Y$. By Example \ref{ex:graded series from ample div}, $\W^{(j)}$ is almost complete and it is clear that $F(\W^{(j)})=0$. Since any $m$-basis type $\bQ$-divisor $D$ of $\W^{(j)}$ that's compatible with $\cF$ can be written as (see the discussion at the beginning of Section \ref{sec:filter by div})
\[
D=S_m(\W^{(j)};Y_{j+1})\cdot Y_{j+1} + \Gamma
\]
where $\Gamma|_{Y_{j+1}}$ is an $m$-basis type $\bQ$-divisor of $\W^{(j+1)}$, we have (see \eqref{eq:c_1 of refinement})
\begin{equation} \label{eq:recursive formula for c_1(W)}
    c_1(\W^{(j+1)}) = \left.\left( c_1(\W^{(j)}) - S(\W^{(j)};Y_{j+1})\cdot Y_{j+1}\right)\right|_{Y_{j+1}}.
\end{equation}
Therefore by Lemma \ref{lem:S of ac W} and induction on $j$ we have 
\[
c_1(\W^{(j)})\sim_\bQ -\left(1-\frac{j}{n+1}\right)K_X|_{Y_j}=2\left(1-\frac{j}{n+1}\right)H
\]
for all $0\le j\le n-1$ and $S(\W^{(j)};Y_{j+1})=S(c_1(\W^{(j)});Y_{j+1})=\frac{2}{n+1}$ for $0\le j\le n-2$. By Theorem \ref{thm:delta adj via flag} (applied to $j=n-1$), we see that in order to prove $\delta_Z(X)=\delta_Z(\V)\ge \frac{n+1}{n}$, it suffices to show that $\delta_{Z\cap Y_{n-1}}(Y_{n-1};\W^{(n-1)},\cF)\ge \frac{n+1}{n}$. As $Z\cap Y_{n-1}$ contains at least two points $P,Q$ and $\deg c_1(\W^{(n-1)}) = \frac{2}{n+1}(-K_X\cdot H^{n-1}) = \frac{4n}{n+1}$, this follows from the next lemma.
\end{proof}

\begin{lem} \label{lem:delta on curve}
Let $C$ be a smooth curve, let $\W$ be a multi-graded linear series with bounded support containing an ample series, let $P_1,\cdots,P_r,Q$ be distinct points on $C$, let $Z=P_1\cup\cdots\cup P_r$ and let $\cF$ be the filtration on $\W$ induced by $Q$. Assume that $\W$ is almost complete and $F(\W)=0$. Then 
\[
\delta_Z(C;\W,\cF)\ge \frac{2r}{\deg c_1(\W)}.
\]
\end{lem}

\begin{proof}
Any $m$-basis type $\bQ$-divisor $D$ of $\W$ that's compatible with $\cF$ has the form $D=S_m(\W;Q)\cdot Q + \Gamma$ for some effective $\bQ$-divisor $\Gamma$. Since $Q\not\in Z$, in order for $Z$ to be contained in the non-lc center of $(C,\lambda D)$, we need $\mult_{P_i}(\lambda\Gamma)>1$ for all $i=1,\cdots,r$. It follows that
\[
\delta_{Z,m}(\W,\cF)\ge \frac{r}{\deg \Gamma} = \frac{r}{\deg D - S_m(\W;Q)}.
\]
Letting $m\to \infty$ we obtain
\[
\delta_Z(\W,\cF)\ge \frac{r}{\deg c_1(\W) - S(\W;Q)}.
\]
The lemma then follows since $S(\W;Q)=S(c_1(\W);Q)=\frac{1}{2}\deg c_1(\W)$ by Lemma \ref{lem:S of ac W}.
\end{proof}

The opposite case is when $X$ contains the secant variety of $Z$.

\begin{lem} \label{lem:secant in X, n>=4}
Let $X\subseteq \bP^{n+1}$ be a smooth hypersurface of degree $n\ge 4$ and $Z\subseteq X$ a subvariety of dimension at least one. Assume that there exists closed points $P\neq Q\in Z$ such that the line joining $P,Q$ is contained in $X$. Then $\delta_Z(X)\ge \frac{n+1}{n}$.
\end{lem}

\begin{proof}
The proof is similar to Lemma \ref{lem:secant not in X}, except that we use a slightly different flag. Consider a flag
\[
\Y \quad:\quad X=Y_0\supseteq Y_1\supseteq \cdots \supseteq Y_n
\]
on $X$ such that each $Y_i$ $(1\le i\le n-2$) is a smooth hyperplane section of $Y_{i-1}$ containing $P$ and $Q$, $Y_{n-1}$ is the line joining $P,Q$ and $Y_n$ is a smooth point on $Y_{n-1}$ that's different from $P,Q$. We use the same notation $\W^{(j)}$ and $\cF$ as in Lemma \ref{lem:secant not in X}. We claim that $\W^{(j)}$ is almost complete and $F(\W^{(j)})=0$ for all $0\le j\le n-1$. Indeed this is evident when $0\le j\le n-2$ by Example \ref{ex:graded series from ample div}, so it remains to consider the case $j=n-1$. For ease of notation, let $S=Y_{n-2}$, $L=Y_{n-1}$ and let $H$ be the hyperplane class. It is straightforward to check that on the surface $S$ (which is a smooth surface of degree $n$ in $\bP^3$) we have
\begin{enumerate}
    \item $(H\cdot L)=1$, $(H^2)=n$, $(L^2)=2-n$ and $(H-L)^2=0$,
    \item $H-L$ is nef.
\end{enumerate}
They together imply that $H^0(S,mH-jL)\neq 0$ ($m,j\in\bN$) if and only if $0\le j\le m$ and that
\begin{equation} \label{eq:S(H;L)}
    S(H;L)=\frac{1}{(H^2)}\int_0^1 (H-xL)^2 \rd x = \frac{2}{3}-\frac{1}{3n}.
\end{equation}
By Kodaira vanishing, we also have $H^1(S,\cO_S(mH-jL))=0$ whenever $m-j > n-4$ and thus the natural map $H^0(S,\cO_S(mH-jL))\to H^0(L,\cO_L(mH-jL))$ is surjective when $m-(j+1) > n-4$. As $\cO_L(mH-jL)\cong \cO_L(m+j(n-2))$, we see that the complete linear series $\V$ associated to $H$ on $S$ has almost complete refinement by $L$. Since $\W^{(n-2)}$ is almost complete (with respect to $H$), most graded pieces of $\W^{(n-2)}$ is a complete linear series $|mH|$ for some $m\in\bN$ (see Example \ref{ex:graded series from ample div}); hence its refinement $\W^{(n-1)}$ by $L$ is also almost complete. As $L\cong\bP^1$, the linear systems $|\cO_L(mH-jL)|$ are all base point free, thus it is not hard to check that $F(\W^{(n-1)})=0$. By Lemma \ref{lem:S of ac W}, \eqref{eq:recursive formula for c_1(W)} and \eqref{eq:S(H;L)}, we find (as in the proof of Lemma \ref{lem:secant not in X})
\[
S(\W^{(j)};Y_{i+1})=\frac{2}{n+1}<\frac{n}{n+1}
\]
when $0\le j\le n-3$ and 
\[
S(\W^{(n-2)};Y_{n-1})=S(c_1(\W^{(n-2)});Y_{n-1})=S\left(\frac{6}{n+1}H;L\right)=\frac{6}{n+1}\left(\frac{2}{3}-\frac{1}{3n}\right)<\frac{n}{n+1}
\]
where the last inequality uses $n\ge 4$. Using \eqref{eq:recursive formula for c_1(W)} one more time, we also obtain
\[
\deg c_1(\W^{(n-1)}) = \frac{6}{n+1} \left(H - \left(\frac{2}{3}-\frac{1}{3n}\right)L \cdot L \right) = \frac{4}{n+1}\left(n-1+\frac{1}{n}\right) < \frac{4n}{n+1}.
\]
Hence by Lemma \ref{lem:delta on curve}, noting that $Z\cap Y_{n-1}$ contains at least two points $P,Q$, we deduce that $\delta_{Z\cap Y_{n-1}}(\W^{(n-1)},\cF)\ge \frac{n+1}{n}$ and therefore the lemma follows from Theorem \ref{thm:delta adj via flag} and the above computations.
\end{proof}

We are ready to prove the K-stability of index two hypersurfaces.

\begin{proof}[Proof of Theorem \ref{thm:index two}]
By Lemma \ref{lem:cubic 3fold} we may assume that $n\ge 4$. 
It suffices to verify the two conditions of Lemma \ref{lem:(n+1)/n-criterion}. The first condition follows from Lemmas \ref{lem:secant not in X} and \ref{lem:secant in X, n>=4} and the second condition follows from Lemma \ref{lem:beta(E_x)>0 on index two}.
\end{proof}

\appendix

\section{Stability thresholds of cubic surfaces} \label{sec:appendix}

In this appendix, we compute the $\delta$-invariants of all smooth cubic surfaces and give the proof of Theorem \ref{thm:cubic delta}. Throughout the section we let $X\subseteq \bP^3$ be a smooth cubic surface, $x\in X$ a closed point and $C=T_x(X)\cap X$.

The proofs are similar between different cases. Note that the first case, i.e. $\mult_x C=3$, is already treated by Theorem \ref{thm:delta at Eckardt pt}. We work out the details when $C$ has a tacnode and sketch the argument in the remaining cases.

\begin{lem} \label{lem:tacnode}
Assume that $C$ has a tacnode at $x$. Then $\delta_x(X)=\frac{27}{17}$ and it is computed by the $($unique$)$ divisor that computes $\lct_x(X,C)$.
\end{lem}

\begin{proof}
By assumption, $C=L\cup Q$ where $L$ (resp. $Q$) is a line (resp. conic) and $L$ is tangent to $Q$ at $x$. We have $L=(u=0)$ and $Q=(u-v^2=0)$ in some local coordinates $u,v$ around $x$. Let $\pi\colon Y\to X$ be the weighted blow up at $x$ with weights $\wt(u)=2$, $\wt(v)=1$ and let $E\subseteq Y$ be the exceptional divisor. Note that $E$ is the unique divisor that computes $\lct_x(X,C)$. Let $\tL$ (resp. $\tQ$) be the strict transform of $L$ (resp. $Q$). We have $(\tL^2)=-3$, $(\tQ^2)=-2$, $(\tL\cdot \tQ)=0$ and $-\pi^*K_X-4E \sim \tL + \tQ$. It follows that the stable base locus of $-\pi^*K_X-tE$ is contained in $\tL\cup \tQ$ for all $0\le t\le 4$ and we have
\begin{equation} \label{eq:negative part-tacnode}
    N_\sigma(-\pi^*K_X-tE) =
    \begin{cases}
        0 & \text{if } 0\le t\le 1, \\
        \frac{t-1}{3}\tL & \text{if } 1<t\le 2, \\
        \frac{t-1}{3}\tL+\frac{t-2}{2}\tQ & \text{if } 2<t\le 4
    \end{cases}
\end{equation}
where $N_\sigma(L)$ (resp. $P_\sigma(L)$) denotes the negative (resp. positive) part in the Zariski decomposition of $L$. Therefore,
\begin{align} \label{eq:volume tacnode}
    \vol(-\pi^*K_X-tE) & = ((P_\sigma(-\pi^*K_X-tE))^2 \\
    & =
    \begin{cases}
        3-\frac{1}{2}t^2 & \text{if } 0\le t\le 1, \\
        3-\frac{1}{2}t^2+\frac{1}{3}(t-1)^2 & \text{if } 1<t\le 2, \\
        3-\frac{1}{2}t^2+\frac{1}{3}(t-1)^2+\frac{1}{2}(t-2)^2 & \text{if } 2<t\le 4
    \end{cases} \nonumber
\end{align}
In particular, $T(-K_X;E)=4$ (as $\vol(-\pi^*K_X-4E)=0$) and $S(-K_X;E)=\frac{17}{9}$. Let
\[
\lambda=\frac{27}{17}=\frac{A_X(E)}{S(-K_X;E)}.
\]
Clearly $\delta_x(X)\le \lambda$ by definition and it remains to show $\delta_x(X)\ge \lambda$. Let $\V$ be the complete linear series associated to $-K_X$ and let $\W$ be its refinement by $E$. By Theorem \ref{thm:delta adj via div}, it is enough to prove $\delta(E,\Delta_E;\W)\ge \lambda$ where $\Delta_E={\rm Diff}_E(0)=\frac{1}{2}P_0$ and $P_0$ is the (unique) singular point of $Y$. Note that $P_0\not\in \tL\cup \tQ$.

% As $E$ is the unique lc place of $(X,\frac{3}{4}C)$ and $-(K_X+\frac{3}{4}C)$ is ample, we may find $H\sim_\bQ -(K_X+\frac{3}{4}C)$ such that $(X,\frac{3}{4}C+H)$ is still lc with $E$ being an lc place. It follows that $\W$ is almost complete by Lemma \ref{lem:surface almost complete}.
% To see this, let $m,i\in\bN$ with $i\le 4m$ and let 
% \[
% G=-m\pi^*K_X-iE-\lceil N_\sigma(-m\pi^*K_X-iE) \rceil.
% \]
% Since $E\cong\bP^1$ and $H^0(Y,\cO_Y(-m\pi^*K_X-iE))=H^0(Y,\cO_Y(G))$, it suffices to show that the natural map
% \[
% H^0(Y,\cO_Y(G))\to H^0(E,\cO_E(G|_E))
% \]
% given by Lemma \ref{lem:D|_F along plt boundary} is surjective. This would follow if $H^1(Y,\cO_Y(G-E))=0$. As $(G\cdot \tL)\ge 0$ and $(G\cdot \tQ)\ge 0$, it is straightforward to check that $G$ is nef. Moreover
% \[
% G-E-(K_Y+\frac{3}{4}(\tL+\tQ))\sim_\bQ G-\frac{1}{4}\pi^*K_X
% \]
% is nef and big, thus $H^1(Y,\cO_Y(G-E))=0$ by Kawamata-Viehweg vanishing (see \cite{KMM}*{Theorem 1-2-5}) and this proves the claim.
By Lemma \ref{lem:surface almost complete}, $\W$ is almost complete. Therefore, by Corollary \ref{cor:adj delta ac}, it suffices to show
\begin{equation} \label{eq:delta(E) over tacnode}
    \delta(E,\Delta_E+\lambda F(\W);c_1(M_{\vb}))\ge \lambda.
\end{equation}
As in the proof of Theorem \ref{thm:delta at Eckardt pt}, noting that
\begin{align*}
    \lambda c_1(M_{\vb}) + \lambda F(\W) & \sim_\bQ \lambda c_1(\W) \sim _\bQ \lambda (-\pi^*K_X-S(-K_X;E)\cdot E)|_E \\
    & \sim_\bQ -A_X(E)\cdot E|_E \sim_\bQ -(K_Y+E)|_E\sim_\bQ -(K_E+\Delta_E),
\end{align*}
we see that \eqref{eq:delta(E) over tacnode} is equivalent to saying the pair $(E,\Delta_E+\lambda F(\W))$ is K-semistable.

We apply Lemma \ref{lem:F(W) formula on surface} to compute $F(\W)$. Let $P_1=\tL\cap E$ and $P_2=\tQ\cap E$. Then $\Supp(F(\W))\subseteq P_1\cup P_2$. We have $\vol_{Y|E}(-\pi^*K_X-tE)=-\frac{1}{2}\cdot \frac{\rd}{\rd t}\vol(-\pi^*K_X-tE)$
% \[
% f_{\W,P_i}(t)=\mult_{P_i}(N_\sigma(-\pi^*K_X-tE)\cdot E)
% \]
% and $\vol_{\W}(t)=\vol_{Y|E}(-\pi^*K_X-tE)=-\frac{1}{2}\cdot \frac{\rd}{\rd t}\vol(-\pi^*K_X-tE)$ 
by \cite{LM-okounkov-body}*{Corollary C}. Combined with \eqref{eq:negative part-tacnode}, \eqref{eq:volume tacnode} and Lemma \ref{lem:F(W) formula on surface}, we deduce
\begin{align*}
    \mult_{P_1} F(\W) & = \frac{1}{(-K_X)^2}\left(\int_1^2 \frac{t-1}{3}\cdot \frac{t+2}{3} \rd t + \int_2^4 \frac{t-1}{3}\cdot \frac{2(4-t)}{3}\rd t \right) = \frac{17}{54}, \\
    \mult_{P_2} F(\W) & = \frac{1}{(-K_X)^2} \int_2^4 \frac{t-2}{2}\cdot \frac{2(4-t)}{3}\rd t = \frac{4}{27}.
\end{align*}
Thus $(E,\Delta_E+\lambda F(\W))\cong (\bP^1,\frac{1}{2}P_0+\frac{1}{2}P_1+\frac{4}{17}P_2)$, which is K-semistable by Lemma \ref{lem:K-ss of pair on P^1}. This finishes the proof.
\end{proof}

The follow result is used in the above proof.

\begin{lem} \label{lem:K-ss of pair on P^1}
Let $\Delta=a_1 P_1+\cdots+a_m P_m$ where $P_1,\cdots,P_m$ are distinct points on $\bP^1$ and $a_i\in (0,1)$ $(i=1,\cdots,m)$ satisfy $a_1+\cdots+a_m<2$ $($i.e. $(\bP^1,\Delta)$ is log Fano$)$. Then 
\[
\delta(\bP^1,\Delta) = \frac{1-\max_{1\le i\le m}\{a_i\}}{1-\frac{1}{2}(a_1+\cdots+a_m)}.
\]
In particular, $(\bP^1,\Delta)$ is K-semistable if and only if $a_1+\cdots+a_m\ge 2a_i$ for all $1\le i\le m$.
\end{lem}

\begin{proof}
We have $S(-K_{\bP^1}-\Delta;P)=\frac{1}{2}\deg(-K_{\bP^1}-\Delta)=1-\frac{1}{2}(a_1+\cdots+a_m)$ and $A_{\bP^1,\Delta}(P)=1-\mult_P(\Delta)$ for any $P\in \bP^1$. The result then follows from the definition of stability thresholds.
\end{proof}

\begin{lem} \label{lem:cusp}
Assume that $C$ has a cusp at $x$. Then $\delta_x(X)=\frac{5}{3}$ and it is computed by the $($unique$)$ divisor that computes $\lct_x(X,C)$.
\end{lem}

\begin{proof}
The proof is very similar to that of Lemma \ref{lem:tacnode} so we only sketch the steps. In local coordinates, $C=(u^2-v^3=0)$ around $x$. Let $\pi\colon Y\to X$ be the weighted blow up at $x$ with $\wt(u)=3$, $\wt(v)=2$ and let $E$ be the exceptional divisor. Let $\tC$ be the strict transform of $C$. We have $T(-K_X;E)=6$,
\[
N_\sigma(-\pi^*K_X-tE) =
    \begin{cases}
        0 & \text{if } 0\le t\le 3, \\
        \frac{t-3}{3}\tC & \text{if } 3<t\le 6
    \end{cases}
\]
and
\[
\vol(-\pi^*K_X-tE) = 
    \begin{cases}
        3-\frac{1}{6}t^2 & \text{if } 0\le t\le 3, \\
        \frac{1}{6}(6-t)^2 & \text{if } 3<t\le 6.
    \end{cases}
\]
Thus $S(-K_X;E)=3$. Let $\lambda=\frac{5}{3}=\frac{A_X(E)}{S(-K_X;E)}$ and let $\W$ be the refinement by $E$ of the complete linear series of $-K_X$. As in the proof of Lemma \ref{lem:tacnode}, $\W$ is almost complete and it suffices to show that $(E,\Delta_E+\lambda F(\W))$ is K-semistable. Note that $\Delta_E=\frac{1}{2}P_0+\frac{2}{3}P_1$ where $P_0,P_1$ are the two singular points of $Y$. By Lemma \ref{lem:F(W) formula on surface}, we find
\[
\mult_{P_2} F(\W) = \frac{1}{(-K_X)^2} \int_3^6 \frac{t-3}{3}\cdot \frac{6-t}{3} \rd t = \frac{1}{6}
\]
where $P_2=\tC\cap E$. Thus $(E,\Delta_E+\lambda F(\W))\cong (\bP^1,\frac{1}{2}P_0+\frac{2}{3}P_1+\frac{5}{18}P_2)$ which is K-semistable by Lemma \ref{lem:K-ss of pair on P^1}. This concludes the proof.
\end{proof}

\begin{lem} \label{lem:irred node}
Assume that $C$ is irreducible and has a node at $x$. Then $\delta_x(X)=\frac{12}{7}$ and it is computed by the ordinary blow up of $x$.
\end{lem}

\begin{proof}
Again we only sketch the steps. Let $\pi\colon Y\to X$ be the ordinary blow up of $x$ and let $E$ be the exceptional divisor. Let $\tC$ be the strict transform of $C$. We have $T(-K_X;E)=2$,
\[
N_\sigma(-\pi^*K_X-tE) =
    \begin{cases}
        0 & \text{if } 0\le t\le \frac{3}{2}, \\
        (2t-3)\tC & \text{if } \frac{3}{2}<t\le 2
    \end{cases}
\]
and
\[
\vol(-\pi^*K_X-tE) = 
    \begin{cases}
        3-t^2 & \text{if } 0\le t\le \frac{3}{2}, \\
        3(2-t)^2 & \text{if } \frac{3}{2}<t\le 2.
    \end{cases}
\]
Thus $S(-K_X;E)=\frac{7}{6}$. Let $\lambda=\frac{12}{7}=\frac{A_X(E)}{S(-K_X;E)}$ and let $\W$ be the refinement by $E$ of the complete linear series of $-K_X$. Since $\W$ is almost complete by Lemma \ref{lem:surface almost complete}, it suffices to show that $(E,\lambda F(\W))$ is K-semistable as in the proof of Lemma \ref{lem:tacnode} (note that $\Delta_E=0$ in this case). By Lemma \ref{lem:F(W) formula on surface}, we see that $F(\W)=\mu(P_1+P_2)$ for some $\mu>0$ where $\{P_1,P_2\}=\tC\cap E$. Thus $(E,\lambda F(\W))$ is K-semistable by Lemma \ref{lem:K-ss of pair on P^1} (regardless of the value of $\mu$). This proves the lemma.
\end{proof}

\begin{lem} \label{lem:two lines}
Assume that $C$ is a union of three lines and $\mult_x C=2$. Then $\delta_x(X)=\frac{18}{11}$ and it is computed by the ordinary blow up of $x$.
\end{lem}

\begin{proof}
Write $C=L_1+L_2+L_3$ where $L_1\cap L_2 = x$. Let $\pi\colon Y\to X$ be the ordinary blow up of $x$ and let $E$ be the exceptional divisor. Let $\tL_i$ be the strict transform of $L_i$ ($i=1,2$). We have $T(-K_X;E)=2$,
\[
N_\sigma(-\pi^*K_X-tE) =
    \begin{cases}
        0 & \text{if } 0\le t\le 1, \\
        \frac{t-1}{2}(\tL_1+\tL_2) & \text{if } 1<t\le 2
    \end{cases}
\]
and
\[
\vol(-\pi^*K_X-tE) = 
    \begin{cases}
        3-t^2 & \text{if } 0\le t\le 1, \\
        4-2t & \text{if } 1<t\le 2.
    \end{cases}
\]
Thus $S(-K_X;E)=\frac{11}{9}$. Let $\lambda=\frac{18}{11}=\frac{A_X(E)}{S(-K_X;E)}$ and let $\W$ be the refinement by $E$ of the complete linear series of $-K_X$. As in previous cases, $\W$ is almost complete and it suffices to show that $(E,\lambda F(\W))$ is K-semistable (note that $\Delta_E=0$). By Lemma \ref{lem:F(W) formula on surface}, we have $F(\W)=\mu(P_1+P_2)$ for some $\mu>0$ where $P_i=\tL_i\cap E$. Thus $(E,\lambda F(\W))$ is K-semistable by Lemma \ref{lem:K-ss of pair on P^1} (regardless of the value of $\mu$). This proves the lemma.
\end{proof}

\begin{lem} \label{lem:line+conic}
Assume that $C=L\cup Q$ where $L$ is a line, $Q$ is a conic and they intersect transversally at $x$. Then $\delta_x(X)=\frac{9}{25-8\sqrt{6}}$ and it is computed by the weighted blow up at $x$ with $\wt(u)=1+\sqrt{6}$ and $\wt(v)=2$ $($where $u$ resp. $v$ is the local defining equation of $L$ resp. $Q)$.
\end{lem}

\begin{proof}
For each $a,b>0$, let $\nu_{a,b}$ be the quasi-monomial valuation over $x\in X$ defined by $\nu_{a,b}(u)=a$ and $\nu_{a,b}(v)=b$. We first identify the minimizer of $\frac{A_X(\nu_{a,b})}{S_X(\nu_{a,b})}$. 
For this choose coprime integers $a,b>0$ and let $\pi\colon Y=Y_{a,b}\to X$ be the weighted blow up at $x$ with $\wt(u)=a$ and $\wt(v)=b$. Let $E$ be the exceptional divisor and let $\tL$ (resp. $\tQ$) be the strict transform of $L$ (resp. $Q$). Assume that $b<2a$. Then similar to the calculations in previous cases we have (c.f. \eqref{eq:negative part line+conic})
\[
N_\sigma(-\pi^*K_X-tE) =
    \begin{cases}
        0 & \text{if } 0\le t\le b, \\
        \frac{t-b}{a+b}\tL & \text{if } b<t\le \frac{a(2a+3b)}{2a+b}, \\
        \frac{2t-2a-b}{b}\tL + \frac{(2a+b)t-a(2a+3b)}{b^2}\tQ & \text{if } \frac{a(2a+3b)}{2a+b}<t\le a+b,
    \end{cases}
\]
\begin{align*}
    \vol_{Y|E}(-\pi^*K_X-tE) & = \big(P_\sigma (-\pi^*K_X-tE) \cdot E \big) = -\frac{1}{2}\cdot \frac{\rd}{\rd t}\vol(-\pi^*K_X-tE) \\
    & = \begin{cases}
        \frac{t}{ab} & \text{if } 0\le t\le b, \\
        \frac{t+a}{a(a+b)} & \text{if } b<t\le \frac{a(2a+3b)}{2a+b}, \\
        \frac{4(a+b-t)}{b^2} & \text{if } \frac{a(2a+3b)}{2a+b}<t\le a+b,
    \end{cases}
\end{align*}
and
\[
S(-K_X;E)=\frac{2}{(-K_X)^2}\int_0^{a+b} t\cdot \vol_{Y|E}(-\pi^*K_X-tE)\rd t=\frac{10a^2+19ab+3b^2}{9(2a+b)}.
\]
Note that $\frac{A_X(\nu_{a,b})}{S_X(\nu_{a,b})}$ only depends on the ratio $\frac{a}{b}$, thus by continuity \cite{BLX-openness}*{Proposition 2.4} we have 
\[
\frac{A_X(\nu_{a,b})}{S_X(\nu_{a,b})}=\frac{9(a+b)(2a+b)}{10a^2+19ab+3b^2}
\]
for all $a,b\in\bR_+$. It achieves its minimum $\lambda = \frac{9}{25-8\sqrt{6}}$ when $\frac{a}{b}=\mu:=\frac{1+\sqrt{6}}{2}>\frac{1}{2}$. In particular, we have $\delta_x(X)\le \lambda$. It remains to show $\delta_x(X)\ge \lambda$. 

Choose a sequence of coprime integers $a_m,b_m>0$ ($m=1,2,\cdots$) such that $\mu_m:=\frac{a_m}{b_m}\to \mu$ ($m\to \infty$). Let $\pi_m\colon Y_m=Y_{a_m,b_m}\to X$ be the corresponding weighted blow up and let $E_m$ be the exceptional divisor. Let $P_1^{(m)}=\tL\cap E_m$, $P_2^{(m)}=\tQ\cap E_m$ and let $\W^{(m)}$ be the refinement by $E_m$ of the complete linear series associated to $-K_X$. As before $\W^{(m)}$ is almost complete by Lemma \ref{lem:surface almost complete}. Using the above calculations and Lemma \ref{lem:F(W) formula on surface}, we have
\[
F(\W^{(m)}) = \frac{c_1^{(m)}}{b_m} P_1^{(m)} + \frac{c_2^{(m)}}{a_m} P_2^{(m)}
\]
where
\[
c_1^{(m)} = \frac{20\mu_m^3-8\mu_m^2+\mu_m+1}{9\mu_m(2\mu_m+1)^2}, \quad c_2^{(m)} = \frac{4}{9(2\mu_m+1)^2}.
\]
Let $\Delta_m={\rm Diff}_{E_m}(0)=(1-\frac{1}{b_m})P_1^{(m)}+(1-\frac{1}{a_m})P_2^{(m)}$, $\lambda_m=\frac{A_X(E_m)}{S(-K_X;E_m)}$ and let
\begin{equation} \label{eq:r_m}
    r_m:=\delta(E_m,\Delta_m+\lambda_m F(\W^{(m)}); \lambda_m c_1(M_{\vb}^{(m)}))=\delta(E_m,\Delta_m+\lambda_m F(\W^{(m)}))
\end{equation}
where the last equality follows from \eqref{eq:c_1 of refinement}. Then by Corollary \ref{cor:adj delta ac} we obtain
\begin{equation} \label{eq:L+Q delta>=}
    \delta_x(X)\ge \min\{\lambda_m,r_m \lambda_m\}
\end{equation}
for all $m$. Note that $\lambda_m\to \lambda$ as $m\to \infty$.

We claim that $r_m\to 1$ as $m\to \infty$. Since $E_m\cong\bP^1$, by Lemma \ref{lem:K-ss of pair on P^1} this is equivalent to 
\[
\frac{1-\mult_{P_1^{(m)}}(\Delta_m+\lambda_m F(\W^{(m)}))}{1-\mult_{P_2^{(m)}}(\Delta_m+\lambda_m F(\W^{(m)}))} \to 1
\]
when $m\to\infty$. It is straightforward (though a bit tedious) to check that
\[
\text{LHS}=\mu_m\cdot \frac{1-\lambda_m c_1^{(m)}}{1-\lambda_m c_2^{(m)}}\to \mu\cdot \frac{9\mu(2\mu+1)^2-\lambda (20\mu^3-8\mu^2+\mu+1)}{9\mu(2\mu+1)^2-4\lambda \mu} = 1.
\]
This proves the claim. Letting $m\to \infty$ in \eqref{eq:L+Q delta>=} we obtain $\delta_x(X)\ge \lambda$ as desired.
\end{proof}

\begin{cor} \label{cor:irrat'l global delta}
In the situation of Lemma \ref{lem:line+conic}, let $0<\varepsilon\ll 1$ be a rational number. Then the pair $(X,(1-\varepsilon)C)$ is log Fano and $\delta(X,(1-\varepsilon)C)=\frac{9}{25-8\sqrt{6}}$.
\end{cor}

\begin{proof}
We continue to use the notation from Lemma \ref{lem:line+conic}. Since $(X,C)$ is lc, it is clear that $(X,(1-\varepsilon)C)$ is log Fano. Let $L\cap Q=\{x,y\}$, let $\nu$ be the quasi-monomial valuation that computes $\delta_x(X)$ in Lemma \ref{lem:line+conic}, and let $\lambda=\frac{9}{25-8\sqrt{6}}$. Note that $\nu$ is an lc place of $(X,C)$, i.e. $A_X(\nu)=\nu(C)$. Then by Lemma \ref{lem:line+conic} we get $A_{X,(1-\varepsilon)C)}(\nu)=\varepsilon A_X(\nu)=\varepsilon \lambda S_X(\nu)=\lambda S_{X,(1-\varepsilon)C}(\nu)$ and hence $\delta(X,(1-\varepsilon)C)\le \lambda$. To get the reverse inequality, we shall prove
\begin{equation} \label{eq:delta_z}
    \delta_z(X,(1-\varepsilon)C)\ge \lambda
\end{equation}
for any closed point $z\in X$. In any case, we have $A_X(\nu)\ge \nu(C)$ and hence $A_{X,(1-\varepsilon)C}(\nu)\ge \varepsilon A_X(\nu)\ge \delta_z(X)\cdot \varepsilon S_{X}(\nu)=\delta_z(X)\cdot S_{X,(1-\varepsilon)C}(\nu)$ for any divisorial valuation $\nu$ whose center contains $z$. It follows that $\delta_z(X,(1-\varepsilon)C)\ge \delta_z(X)$ and hence by Lemma \ref{lem:line+conic}, \eqref{eq:delta_z} holds when $z\in \{x,y\}$. If $z\not\in\Supp(C)$ then $\nu(C)=0$ for any divisorial valuation $\nu$ whose center contains $z$, hence by the definition of stability thresholds we get $\delta_z(X,(1-\varepsilon)C)=\frac{\delta_z(X)}{\varepsilon}\ge \lambda$ when $0<\varepsilon\ll 1$. Thus it remains to consider the case when $z\in \Supp(C)\setminus\{x,y\}$. For simplicity, we assume $z\in Q$ (the other case $z\in L$ is similar). Consider the refinement (denote it by $\W$) by $Q$ of the complete linear series associated to $-K_X$. Note that $\delta(\W)>0$. Since $-(K_X+(1-\varepsilon)C)\sim_\bQ -\varepsilon K_X$, by Theorem \ref{thm:delta adj via div} we have
\[
\delta_z(X,(1-\varepsilon)C)=\varepsilon^{-1}\delta_z(X,(1-\varepsilon)C;-K_X)\ge \min\left\{\frac{A_{X,(1-\varepsilon)C}(Q)}{S_{X,(1-\varepsilon)C}(Q)}, \varepsilon^{-1}\delta_z(\W)\right\}\ge \lambda
\]
when $0<\varepsilon\ll 1$. Hence \eqref{eq:delta_z} also holds in this case. The proof is now complete.
\end{proof}

\begin{proof}[Proof of Theorem \ref{thm:cubic delta}]
This follows from the combination of Theorem \ref{thm:delta at Eckardt pt}, Lemmas \ref{lem:tacnode}, \ref{lem:cusp}, \ref{lem:irred node}, \ref{lem:two lines}, \ref{lem:line+conic} and Corollary \ref{cor:irrat'l global delta}.
\end{proof}

\bibliography{ref}

\begin{bibdiv}
\begin{biblist}

\bib{AGP}{article}{
      author={Arezzo, Claudio},
      author={Ghigi, Alessandro},
      author={Pirola, Gian~Pietro},
       title={Symmetries, quotients and {K}\"ahler-{E}instein metrics},
        date={2006},
     journal={J. Reine Angew. Math.},
      volume={591},
       pages={177\ndash 200},
}

\bib{BC-okounkov-body}{article}{
      author={Boucksom, S\'{e}bastien},
      author={Chen, Huayi},
       title={Okounkov bodies of filtered linear series},
        date={2011},
     journal={Compos. Math.},
      volume={147},
      number={4},
       pages={1205\ndash 1229},
}

\bib{BdFFU-valuation}{incollection}{
      author={Boucksom, S.},
      author={de~Fernex, T.},
      author={Favre, C.},
      author={Urbinati, S.},
       title={Valuation spaces and multiplier ideals on singular varieties},
        date={2015},
   booktitle={Recent advances in algebraic geometry},
      series={London Math. Soc. Lecture Note Ser.},
      volume={417},
   publisher={Cambridge Univ. Press, Cambridge},
       pages={29\ndash 51},
}

\bib{BE-compatible-basis}{article}{
      author={Boucksom, S\'{e}bastien},
      author={Eriksson, Dennis},
       title={Spaces of norms, determinant of cohomology and {F}ekete points in
  non-{A}rchimedean geometry},
        date={2021},
     journal={Adv. Math.},
      volume={378},
       pages={107501, 124},
}

\bib{BHJ-DH-measure}{article}{
      author={Boucksom, S\'{e}bastien},
      author={Hisamoto, Tomoyuki},
      author={Jonsson, Mattias},
       title={Uniform {K}-stability, {D}uistermaat-{H}eckman measures and
  singularities of pairs},
        date={2017},
     journal={Ann. Inst. Fourier (Grenoble)},
      volume={67},
      number={2},
       pages={743\ndash 841},
}

\bib{BJ-delta}{article}{
      author={Blum, Harold},
      author={Jonsson, Mattias},
       title={Thresholds, valuations, and {K}-stability},
        date={2020},
     journal={Adv. Math.},
      volume={365},
       pages={107062},
}

\bib{BSKM-linear-growth}{article}{
      author={Boucksom, S\'{e}bastien},
      author={K\"{u}ronya, Alex},
      author={Maclean, Catriona},
      author={Szemberg, Tomasz},
       title={Vanishing sequences and {O}kounkov bodies},
        date={2015},
     journal={Math. Ann.},
      volume={361},
      number={3-4},
       pages={811\ndash 834},
}

\bib{BLX-openness}{article}{
      author={Blum, Harold},
      author={Liu, Yuchen},
      author={Xu, Chenyang},
       title={Openness of {K}-semistability for {F}ano varieties},
        date={2019},
        note={To appear in Duke Math. J.,
  \href{https://arxiv.org/abs/1907.02408}{\textsf{arXiv:1907.02408}}},
}

\bib{BLZ-optimal-destabilize}{article}{
      author={Blum, Harold},
      author={Liu, Yuchen},
      author={Zhou, Chuyu},
       title={Optimal destabilization of {K}-unstable {F}ano varieties via
  stability thresholds},
        date={2019},
        note={To appear in Geom. Topol.,
  \href{https://arxiv.org/abs/1907.05399}{\textsf{arXiv:1907.05399}}},
}

\bib{Bou-NObody}{article}{
      author={Boucksom, S\'{e}bastien},
       title={Corps d'{O}kounkov (d'apr\`es {O}kounkov,
  {L}azarsfeld-{M}usta\c{t}\v{a} et {K}aveh-{K}hovanskii)},
        date={2014},
     journal={Ast\'{e}risque},
      number={361},
       pages={Exp. No. 1059, vii, 1\ndash 41},
}

\bib{BX-separatedness}{article}{
      author={Blum, Harold},
      author={Xu, Chenyang},
       title={Uniqueness of {K}-polystable degenerations of {F}ano varieties},
        date={2019},
     journal={Ann. of Math. (2)},
      volume={190},
      number={2},
       pages={609\ndash 656},
}

\bib{CDS}{article}{
      author={Chen, Xiuxiong},
      author={Donaldson, Simon},
      author={Sun, Song},
       title={K\"ahler-{E}instein metrics on {F}ano manifolds, {I-III}},
        date={2015},
     journal={J. Amer. Math. Soc.},
      volume={28},
      number={1},
       pages={183\ndash 197, 199\ndash 234, 235\ndash 278},
}

\bib{CS-Noether-Fano}{article}{
      author={Cheltsov, I.~A.},
      author={Shramov, K.~A.},
       title={Log-canonical thresholds for nonsingular {F}ano threefolds},
        date={2008},
     journal={Uspekhi Mat. Nauk},
      volume={63},
      number={5(383)},
       pages={73\ndash 180},
}

\bib{CZ-cubic-delta}{article}{
      author={Cheltsov, Ivan},
      author={Zhang, Kewei},
       title={Delta invariants of smooth cubic surfaces},
        date={2019},
     journal={Eur. J. Math.},
      volume={5},
      number={3},
       pages={729\ndash 762},
}

\bib{Der-finite-cover}{article}{
      author={Dervan, Ruadha\'{\i}},
       title={On {K}-stability of finite covers},
        date={2016},
     journal={Bull. Lond. Math. Soc.},
      volume={48},
      number={4},
       pages={717\ndash 728},
}

\bib{Der-uKs}{article}{
      author={Dervan, Ruadha\'{\i}},
       title={Uniform stability of twisted constant scalar curvature
  {K}\"{a}hler metrics},
        date={2016},
     journal={Int. Math. Res. Not. IMRN},
      number={15},
       pages={4728\ndash 4783},
}

\bib{DK01}{article}{
      author={Demailly, Jean-Pierre},
      author={Koll\'{a}r, J\'{a}nos},
       title={Semi-continuity of complex singularity exponents and
  {K}\"{a}hler-{E}instein metrics on {F}ano orbifolds},
        date={2001},
     journal={Ann. Sci. \'{E}cole Norm. Sup. (4)},
      volume={34},
      number={4},
       pages={525\ndash 556},
}

\bib{Don-K-stability-defn}{article}{
      author={Donaldson, S.~K.},
       title={Scalar curvature and stability of toric varieties},
        date={2002},
     journal={J. Differential Geom.},
      volume={62},
      number={2},
       pages={289\ndash 349},
}

\bib{ELMNP}{article}{
      author={Ein, Lawrence},
      author={Lazarsfeld, Robert},
      author={Musta\c{t}\u{a}, Mircea},
      author={Nakamaye, Michael},
      author={Popa, Mihnea},
       title={Restricted volumes and base loci of linear series},
        date={2009},
     journal={Amer. J. Math.},
      volume={131},
      number={3},
       pages={607\ndash 651},
}

\bib{FO-delta}{article}{
      author={Fujita, Kento},
      author={Odaka, Yuji},
       title={On the {K}-stability of {F}ano varieties and anticanonical
  divisors},
        date={2018},
     journal={Tohoku Math. J. (2)},
      volume={70},
      number={4},
       pages={511\ndash 521},
}

\bib{Fuj-div-stable}{article}{
      author={Fujita, Kento},
       title={On {$K$}-stability and the volume functions of
  {$\mathbb{Q}$}-{F}ano varieties},
        date={2016},
     journal={Proc. Lond. Math. Soc. (3)},
      volume={113},
      number={5},
       pages={541\ndash 582},
}

\bib{Fuj-alpha}{article}{
      author={Fujita, Kento},
       title={K-stability of {F}ano manifolds with not small alpha invariants},
        date={2019},
     journal={J. Inst. Math. Jussieu},
      volume={18},
      number={3},
       pages={519\ndash 530},
}

\bib{Fuj-plt-blowup}{article}{
      author={Fujita, Kento},
       title={Uniform {K}-stability and plt blowups of log {F}ano pairs},
        date={2019},
     journal={Kyoto J. Math.},
      volume={59},
      number={2},
       pages={399\ndash 418},
}

\bib{Fuj-valuative-criterion}{article}{
      author={Fujita, Kento},
       title={A valuative criterion for uniform {K}-stability of {$\mathbb
  Q$}-{F}ano varieties},
        date={2019},
     journal={J. Reine Angew. Math.},
      volume={751},
       pages={309\ndash 338},
}

\bib{Gol-delta-large-aut}{article}{
      author={Golota, Aleksei},
       title={Delta-invariants for {F}ano varieties with large automorphism
  groups},
        date={2020},
     journal={Internat. J. Math.},
      volume={31},
      number={10},
       pages={2050077, 31},
}

\bib{HLS-acc-mld}{article}{
      author={Han, Jingjun},
      author={Liu, Jihao},
      author={Shokurov, V.~V.},
       title={{ACC} for minimal log discrepancies of exceptional
  singularities},
        date={2019},
  note={\href{https://arxiv.org/abs/1903.04338}{\textsf{arXiv:1903.04338}}},
}

\bib{JM-valuation}{article}{
      author={Jonsson, Mattias},
      author={Musta\c{t}\u{a}, Mircea},
       title={Valuations and asymptotic invariants for sequences of ideals},
        date={2012},
     journal={Ann. Inst. Fourier (Grenoble)},
      volume={62},
      number={6},
       pages={2145\ndash 2209 (2013)},
}

\bib{Jow-NObody}{article}{
      author={Jow, Shin-Yao},
       title={Okounkov bodies and restricted volumes along very general
  curves},
        date={2010},
     journal={Adv. Math.},
      volume={223},
      number={4},
       pages={1356\ndash 1371},
}

\bib{Kho-polytope}{article}{
      author={Khovanski\u{\i}, A.~G.},
       title={The {N}ewton polytope, the {H}ilbert polynomial and sums of
  finite sets},
        date={1992},
     journal={Funktsional. Anal. i Prilozhen.},
      volume={26},
      number={4},
       pages={57\ndash 63, 96},
}

\bib{KO-P^n}{article}{
      author={Kobayashi, Shoshichi},
      author={Ochiai, Takushiro},
       title={Characterizations of complex projective spaces and
  hyperquadrics},
        date={1973},
     journal={J. Math. Kyoto Univ.},
      volume={13},
       pages={31\ndash 47},
}

\bib{Kol-mmp}{book}{
      author={Koll{\'a}r, J{\'a}nos},
       title={Singularities of the minimal model program},
      series={Cambridge Tracts in Mathematics},
   publisher={Cambridge University Press, Cambridge},
        date={2013},
      volume={200},
        note={With a collaboration of S{\'a}ndor Kov{\'a}cs},
}

\bib{Laz-positivity-1}{book}{
      author={Lazarsfeld, Robert},
       title={Positivity in algebraic geometry. {I}},
      series={Ergebnisse der Mathematik und ihrer Grenzgebiete. 3. Folge. A
  Series of Modern Surveys in Mathematics [Results in Mathematics and Related
  Areas. 3rd Series. A Series of Modern Surveys in Mathematics]},
   publisher={Springer-Verlag, Berlin},
        date={2004},
      volume={48},
        note={Classical setting: line bundles and linear series},
}

\bib{Li-equivariant-minimize}{article}{
      author={Li, Chi},
       title={{K}-semistability is equivariant volume minimization},
        date={2017},
     journal={Duke Math. J.},
      volume={166},
      number={16},
       pages={3147\ndash 3218},
}

\bib{AIM}{misc}{
      author={List, AIM~Problem},
       title={K-stability and related topics},
        date={2020},
        note={available at
  \href{http://aimpl.org/kstability}{http://aimpl.org/kstability}},
}

\bib{LM-okounkov-body}{article}{
      author={Lazarsfeld, Robert},
      author={Musta\c{t}\u{a}, Mircea},
       title={Convex bodies associated to linear series},
        date={2009},
     journal={Ann. Sci. \'{E}c. Norm. Sup\'{e}r. (4)},
      volume={42},
      number={5},
       pages={783\ndash 835},
}

\bib{LX-higher-rank}{article}{
      author={Li, Chi},
      author={Xu, Chenyang},
       title={Stability of {V}aluations: {H}igher {R}ational {R}ank},
        date={2018},
     journal={Peking Math. J.},
      volume={1},
      number={1},
       pages={1\ndash 79},
}

\bib{LX-cubic-3fold}{article}{
      author={Liu, Yuchen},
      author={Xu, Chenyang},
       title={K-stability of cubic threefolds},
        date={2019},
     journal={Duke Math. J.},
      volume={168},
      number={11},
       pages={2029\ndash 2073},
}

\bib{LX-Kollar-comp}{article}{
      author={Li, Chi},
      author={Xu, Chenyang},
       title={Stability of valuations and {K}oll\'{a}r components},
        date={2020},
     journal={J. Eur. Math. Soc. (JEMS)},
      volume={22},
      number={8},
       pages={2573\ndash 2627},
}

\bib{Nak-Zariski-decomp}{book}{
      author={Nakayama, Noboru},
       title={Zariski-decomposition and abundance},
      series={MSJ Memoirs},
   publisher={Mathematical Society of Japan, Tokyo},
        date={2004},
      volume={14},
}

\bib{Oda-slcCY-Kss}{article}{
      author={Odaka, Yuji},
       title={A generalization of the {R}oss-{T}homas slope theory},
        date={2013},
     journal={Osaka J. Math.},
      volume={50},
      number={1},
       pages={171\ndash 185},
}

\bib{OS-alpha}{article}{
      author={Odaka, Yuji},
      author={Sano, Yuji},
       title={Alpha invariant and {K}-stability of {$\mathbb{Q}$}-{F}ano
  varieties},
        date={2012},
     journal={Adv. Math.},
      volume={229},
      number={5},
       pages={2818\ndash 2834},
}

\bib{PW-dP}{article}{
      author={Park, Jihun},
      author={Won, Joonyeong},
       title={K-stability of smooth del {P}ezzo surfaces},
        date={2018},
     journal={Math. Ann.},
      volume={372},
      number={3-4},
       pages={1239\ndash 1276},
}

\bib{SS-two-quadric}{article}{
      author={Spotti, Cristiano},
      author={Sun, Song},
       title={Explicit {G}romov-{H}ausdorff compactifications of moduli spaces
  of {K}\"{a}hler-{E}instein {F}ano manifolds},
        date={2017},
     journal={Pure Appl. Math. Q.},
      volume={13},
      number={3},
       pages={477\ndash 515},
}

\bib{SZ-bsr-imply-K}{article}{
      author={Stibitz, Charlie},
      author={Zhuang, Ziquan},
       title={K-stability of birationally superrigid {F}ano varieties},
        date={2019},
     journal={Compos. Math.},
      volume={155},
      number={9},
       pages={1845\ndash 1852},
}

\bib{Tian-YTD}{article}{
      author={Tian, Gang},
       title={K-stability and {K}\"ahler-{E}instein metrics},
        date={2015},
     journal={Comm. Pure Appl. Math.},
      volume={68},
      number={7},
       pages={1085\ndash 1156},
}

\bib{Tian-criterion}{article}{
      author={Tian, Gang},
       title={On {K}\"{a}hler-{E}instein metrics on certain {K}\"{a}hler
  manifolds with {$C_1(M)>0$}},
        date={1987},
     journal={Invent. Math.},
      volume={89},
      number={2},
       pages={225\ndash 246},
}

\bib{Tian-dP}{article}{
      author={Tian, G.},
       title={On {C}alabi's conjecture for complex surfaces with positive first
  {C}hern class},
        date={1990},
     journal={Invent. Math.},
      volume={101},
      number={1},
       pages={101\ndash 172},
}

\bib{Tian-K-stability-defn}{article}{
      author={Tian, Gang},
       title={K\"ahler-{E}instein metrics with positive scalar curvature},
        date={1997},
     journal={Invent. Math.},
      volume={130},
      number={1},
       pages={1\ndash 37},
}

\bib{Xu-quasimonomial}{article}{
      author={Xu, Chenyang},
       title={A minimizing valuation is quasi-monomial},
        date={2020},
     journal={Ann. of Math. (2)},
      volume={191},
      number={3},
       pages={1003\ndash 1030},
}

\bib{Z-cpi}{article}{
      author={Zhuang, Ziquan},
       title={Birational superrigidity and {$K$}-stability of {F}ano complete
  intersections of index 1},
        date={2020},
     journal={Duke Math. J.},
      volume={169},
      number={12},
       pages={2205\ndash 2229},
        note={With an appendix by Zhuang and Charlie Stibitz},
}

\bib{ZZ-delta-cone}{article}{
      author={Zhang, Kewei},
      author={Zhou, Chuyu},
       title={Delta invariants of projective bundles and projective cones of
  {F}ano type},
        date={2020},
        note={To appear in Math. Z.,
  \href{https://arxiv.org/abs/2003.06839}{\textsf{arXiv:2003.06839}}},
}

\end{biblist}
\end{bibdiv}

\end{document}